\documentclass{amsart}
                        \input xypic

\usepackage{amsthm, amsfonts, amssymb, amsmath, latexsym, enumerate, times}
\usepackage[all]{xy}
\usepackage[latin1]{inputenc}
\usepackage{mathrsfs}
\usepackage{array}
\usepackage{comment}
\usepackage{paralist}

\newtheorem{theorem}{Theorem}[section]
\newtheorem{lemma}[theorem]{Lemma}
\newtheorem{claim}[theorem]{Claim}
\newtheorem{corollary}[theorem]{Corollary}
\newtheorem{proposition}[theorem]{Proposition}

\theoremstyle{definition}
\newtheorem{definition}[theorem]{Definition}
\newtheorem{remark}[theorem]{Remark}
\newtheorem{example}[theorem]{Example}

\newcommand{\Sing}{\rm Sing}
\newcommand{\Seg}{\rm Seg}
\newcommand{\Bsl}{\rm Bsl}

\newcommand{\cal}{\mathcal}

\newcommand{\R}{{\mathcal R}}
\renewcommand{\L}{{\mathcal L}}
\newcommand{\Cc}{{\mathcal C}}

\newcommand{\FF}{{\mathbb F}}

\newcommand{\G}{{\mathbb G}}

\newcommand{\C}{{\mathbb C}}
\newcommand{\p}{{\mathbb P}}

\newcommand{\Reg}{\operatorname{Reg}}
\newcommand{\map}{\dasharrow}

\def\geq{\geqslant}
\def\leq{\leqslant}

\begin{document}
 
\title[On secant defective varieties]{On secant defective varieties, in particular of dimension 4}

%author one
\author{Luca Chiantini}
\address{Dipartimento di Ingegneria dell'Informazione
e Scienze Matematiche,
Universit\`a di Siena, Via Roma, 56,
53100 Siena, Italia}
\email{chiantini@unisi.it}

\author{Ciro Ciliberto}
\address{Largo Bientina 8,
 00199 Roma, Italia}
\email{cilibert@axp.mat.uniroma2.it}

%author two
\author{Francesco Russo}
\address{Dipartimento di Matematica e Informatica, Universit\`a di Catania, Citt\`a Universitaria, Viale A. Doria 6, 95125 Catania, Italia}
\email{frusso@dmi.unict.it}
 
 
 
\maketitle

\begin{abstract} In this paper we prove some general results on secant defective varieties. Then we focus on the 4--dimensional case and we give the full classification of secant defective 4--folds. This paper has been inspired by classical work by G. Scorza.
\end{abstract}

\section*{Introduction} 
Let $X\subset \p^r$ be an $n$--dimensional, irreducible, reduced, projective variety. The variety $X$ is said to be \emph{secant defective}, or simply defective, if its \emph{secant variety} $S(X)$ has dimension $s(X)$ smaller than the expected dimension, which is $\sigma(X)= \min\{r,2n+1\}$. The study and classification of secant defective varieties is a classical subject which goes back to the Italian school of algebraic geometry and received in modern times a lot of attention by various authors, too many to be mentioned here. Curves are never defective. The classification of defective surfaces is classical (for a modern account, see \cite {WDV}): a surface is defective if and only if it is either a cone or the Veronese surface of conics in $\p^5$. The classification of secant defective threefolds goes back to G. Scorza in \cite {Scorza}, and a modern proof has been given in \cite {ChCi} (see also \cite {CC3}). The case of 4--folds has been again considered by Scorza in the memoir \cite {Scorza1}, in which he outlines the classification of secant defective 4--folds, with the use of beautiful geometric ideas. However Scorza's argument are often longwinded and obscure, and his classification eventually contains some gaps, which would be too long to precisely indicate here. The present paper is devoted to the full classification of secant defective 4--folds, filling up Scorza's gaps and simplifying, in several points, his arguments, though often we follow his line of thinking. Our tools, as for Scorza, consist in delicate projective geometric analyses of the various cases which defective 4--folds  fall in. These cases correspond to the different values certain invariants of the 4--fold $X$ acquire. The main invariants in question are basically $f(X)=9-s(X)$ and $\gamma(X)$, which is the dimension of the contact locus with $X$ of a general bitangent space to $X$, and one has $3\geq \gamma(X)\geq f(X)\geq 1$. The difference $4-\gamma(X)$ is dubbed the \emph{species} of $X$.

The outcome of our analysis can be summarized in the following classification theorem:

\begin{theorem}\label{thm:main} let $X\subset \p^r$ be a secant defective fourfold. Then the following cases occur:
\begin {itemize}
\item [(i)] $X$ is a cone;

\item [(ii)] $X$ sits in a $5$ or 6--dimensional cone over a curve;

\item [(iii)]  $X$ sits in a $5$--dimensional cone over a surface;

\item [(iv)] $X$ is the Segre variety $\Seg(2,2)$ product of $\p^2$ times $\p^2$ in   $\p^ 8$;

\item [(v)] $X$ is a scroll in 3--spaces as in Example \ref {ex:ex1} below;

 \item [(vi)] $X$ is a scroll in 3--spaces as as in Example \ref {ex:ex2} below;
 
 \item [(vii)] $r=9$ and $X$ is a scroll in 3--spaces  as in Example \ref {ex:ex3} below;
 
 \item [(viii)]  $X$ is an internal projection of the Veronese 4--fold of quadrics in  $\p^{14}$ from finitely many points 
with the property that it contains at most finitely many lines through its general point;

\item [(ix)] $r=11$ and $X$ is the projection   of the Veronese 4--fold of quadrics in $\p^{14}$ from the plane spanned by a conic on it;

\item [(x)]  $r=9$ and $X$ is the projection  of the Veronese 4--fold of quadrics in $\p^{14}$ from a $4$--space spanned by a rational normal quartic curve on it;

\item [(xi)] $r=10$ and $X$ is a hyperplane section of  the Segre variety ${\rm Seg}(2,3) \subset\p^{11}$;

\item [(xii)] $r=9$ and either $X$ sits in a cone with vertex a line over a hyperplane section of ${\rm Seg}(2,2)$ in $\p^7$ or  $X$ sits  in a cone with vertex a point over the Segre variety ${\rm Seg}(2,2)\subset \p^ 8$;

\item [(xiii)]    $9\leq r\leq 11$ and $X$ sits in a cone with vertex a line over the projection in $\p^{r-2}$ of the Veronese $3$--fold $V_{3,2}$;

\item [(xiv)]  $r=9$ and $X$ sits in a $6$--dimensional cone over the Veronese 
surface $V_{2,2}$ in $\p^ 5$.

\item [(xv)]  $r=9$ and $X$ sits in a cone with vertex a line over a (defective) 3--fold in $\p^7$ sitting in a cone with vertex a line over the Veronese surface $V_{2,2}$;

\item [(xvi)]  $X$ is swept out by a $3$--dimensional family $\cal R$ of lines and
it is singular along a linear space $\Pi$ of dimension $\varepsilon$, with
$2\leq \varepsilon \leq 3$, which is intersected in one point by the general line in $\cal R$, and $X$ projects from $\Pi$ to a 3--dimensional variety $Y\subset \p^{8-\varepsilon}$ (with general fibres unions of lines of $\cal R$), which contains a $4$--dimensional family $\cal C$ of (generically irreducible) conics such that there is a conic in $\cal C$ passing through two general points of $Y$ and the counterimage of the general conic of $\cal C$ via the projection from $\Pi$ is a non--developable scroll spanning a 4--space with a line directrix sitting in $\Pi$;

\item [(xvii)]  $X$ is swept out by a 4--dimensional family $\cal S$ of surfaces spanning a 4--space, such that two general surfaces in $\cal S$ intersect at a point. In this case the general surface in $\cal S$ is rational and $X$ itself is rational; 

\item [(xviii)] the general projection of $X$ in $\p^9$ sits in a 6--dimensional cone with vertex a 3--space over the Veronese surface $V_{2,2}$.
\end{itemize}

If $X$ is smooth, only the cases (iv), (viii), (ix), (x), (xi), (xii) and $X$ sits  in a cone with vertex a point over the Segre variety ${\rm Seg}(2,2)\subset \p^ 8$, and may be case (xvii) occur.

\end{theorem}

Theorem \ref {thm:main} summarizes the results contained in Proposition  \ref {terracini} and Theorems \ref {thm:lowcod}, \ref {prop:lowrk}, \ref {nu1}, \ref {prop:acc2}, \ref {prop:acc5}, \ref {thm:lof}, \ref {thm:wd5}, \ref {thm:wd6}. All the cases listed in Theorem \ref {thm:main} do occur, except, may be, case (xvii) for which we do not have any example.

The paper is organized as follows. In \S \ref {sec:not} we collect some notation. In \S \ref {sec:cont} we introduce some of the main invariants of a defective variety and we recall some important facts and notions, like Terracini's Lemma, the second fundamental form of a variety, etc. In \S \ref {sec:4} we prove the classification Theorem \ref {thm:lowcod} for $n$--dimensional varieties with very low dimensional secant variety. Section \ref {sec:scrolls} is devoted to the classification of defective 4--folds which are scrolls in 3--spaces. Sections \ref  {sec:top}, \ref  {sec:top2} are devoted to the classification of $n$--folds, in particular 4--folds, of the top species. In \S \ref {sec:WD} we classify 4--folds of the second species and in \S \ref {sec:first} we classify 4--folds of the first species. In an Appendix we prove some formulae, originally due to E. Bertini, valid for non--degenerate rational curves in $\p^4$ which are used in \S \ref {sec:first} for the classification of 4--folds of the first species.

Till possible, we kept the discussion as general as possible, referring to an $n$--fold and not simply to a $4$--fold. \medskip

{\bf Acknowledgements:} Ciro Ciliberto acknowledges the MIUR Excellence Department Project awarded to the Department of Mathematics, University of Rome Tor Vergata, CUP E83C18000100006.   Francesco Russo has been partially supported by the PTR Project of Unict ``Propriet\`a algebriche locali e globali di anelli associati a curve e ipersuperfici''. The authors are members of the GNSAGA of CNR.

\section {notation}\label {sec:not}

\subsection {}  In this paper we  work over the complex field
$\mathbb C$. Let $X\subseteq\p^r$ be an irreducible  projective
scheme. We will denote by $\deg(X)$ the {\it degree} of
$X$ and  by $\dim(X)$ the {\it dimension} of $X$. If $X$ is
reducible, by $\dim(X)$ we mean the maximum of the dimensions of
its irreducible components. We will denote by $\mathcal O_X$ the
structure sheaf of $X$ and by $\mathcal I_X$ the ideal sheaf of
$X$ in $\p^r$. 

We will denote by $\Reg(X)$ the open Zariski subset of smooth points of $X$ and by
$\Sing(X)$ the closed subset of $X$ formed by its singular points. If $x\in X$ we denote by 
$T_x(X)$  the \emph {Zariski tangent space} to $x$ at $X$. If $x\in \Reg(X)$ we denote by $T_{X,x}\subseteq \p^r$ the \emph{embedded tangent space} to $X$ at $x$, which is a linear subspace of $\p^r$ of dimension equal to $\dim(X)$. If either $P\subseteq T_{X,x}$ is a linear subspace passing through $x$ or $T_{X,x}\subseteq P$, then $P$ is called a \emph{tangent subspace} to $X$ at $x$.
If $p_0,\ldots, p_m\in \Reg(X)$, then 
we denote by $T_{X,p_0,\ldots, p_m}$ the join of $T_{X,p_i}$ with $1\leq i\leq m$.  

Let $X \subseteq \p^r$ be a variety. For $h$ a positive integer, an \emph{$h$--secant line}  to $X$ is any line $\ell$ not contained in $X$ such that the intersection scheme of $X$ with $\ell$ has length at least $h$. An $h$--secant line is called \emph{proper} if it cuts $X$ at exactly $h$ distinct points. 

Let $X \subseteq \p^r$ be a variety. We will denote by ${\rm Sym}^2(X)$ the \emph{2--fold symmetric product} of $X$, i.e., the variety of all unordered pairs $[x,y]$, with $x,y\in X$. 

 \subsection {} Let $v_{d,r}: \p^r \to \p^{N(d,r)}$, with $N(d,r)={{d+r}\choose r}-1$,  be the \emph{d--th Veronese map} of $\p^r$. Its image is denoted by $V_{d,r}$ and is called the \emph {$(d,r)$--Veronese variety}.
 
 Let $r_1,r_2$ be non--negative integers. One has the \emph{Segre map} $s_{r_1,r_2}: \p^{r_1}\times \p^{r_2} \to \p^{r_1r_2+r_1+r_2}$, whose image is denoted by $\Seg(r_1,r_2)$ and called the \emph{$(r_1,r_2)$--Segre variety}.
 
 Let $n,r$ be two non--negative integers, with $r\geq n$. Then $\mathbb G(n,r)$ denotes the \emph{Grassmann variety} of type $(n,r)$. A point of $\mathbb G(n,r)$ represents a linear subspace of dimension $n$ in $\p^r$, which we also call a $n$--(sub)space of $\p^r$. 
 
 If $Z\subseteq \p^r$, the \emph{span} of $Z$ is denoted by $\langle Z\rangle$.
 
 A $2$--dimensional irreducible, projective subvariety $S$ of $\mathbb G(1,n)$  is called a \emph{congruence} of lines in $\p^n$. In particular, for $n=3$, a congruence $S$ of lines in $\p^3$ is called of type $(a,b)$ if given a general point $x\in \p^3$ and a general plane $P\subset \p^3$, there are exactly $a$ [resp. $b$] lines of the congruence passing through $x$ [resp. contained in $P$]. It is well known that congruences of type (1,1) consist of the set of lines intersecting two given skew lines. 
 
 \subsection{} If $X\subseteq \p^r$ is a variety and $P$ is a linear space of dimension $k$ in $\p^r$, we can consider the \emph{projection} $\pi: X\dasharrow \p^{r-k-1}$ of $X$ from $P$. Recall that the projection is a finite morphism onto its image if $P\cap X=\emptyset$. 
 In this case we will say that the projection is \emph{external}. If $P\subseteq X$ we say that the projection is \emph{internal}. 
 
 The variety $X\subseteq \p^r$ is said to be \emph{non--degenerate} if  $\langle X\rangle =\p^r$.
 
If $X\subseteq \p^r$ is a variety, an \emph{extension} of $X$ is a variety $X'\subseteq \p^{r+k}$, with $k\geq 1$, such that $X$ is the scheme theoretical intersection of $X'$ with a $\p^r$, but $X'$ is not a cone.  If there is an extension of $X$, then $X$ is said to be \emph{extendable}.

 \subsection{} If $X$ is a variety, it is well known the notion of linearly equivalent divisors on $X$ and the relation of this notion with line bundles and linear systems.

 Let $X$ be a variety and $L$ a line bundle on $X$. Then we can consider the \emph{complete linear system} $|L|$ determined by $L$, as well as linear subsystems of it. If $\cal L$ is a linear system, we can consider its \emph{base locus} $\Bsl(\cal L)$, formed by all points $x\in X$ belonging to all divisors in $\cal L$. The divisorial part of $\Bsl(\cal L)$ is called the \emph{divisorial base locus} of $\cal L$. The \emph{movable part} of $\cal L$ is the linear system consisting of all divisors in $\cal L$ minus the divisorial base locus.
 
 If $X\subseteq \p^r$ is a variety, we can consider the \emph{hyperplane line bundle} $\cal O_X(1)$ of $X$. The \emph{hyperplane linear system} $\cal H_X$ (or simply $\cal H$ if there is no ambiguity) of $X$ is the subsystem of $|\cal O_X(1)|$ which is cut out on $X$ by the hyperplanes of $\p^r$.  The variety $X\subseteq \p^r$ is said to be \emph{linearly normal} if there is no variety $X'\subseteq \p^{r+1}$ such that $X$ is isomorphic to $X'$ via an external projection from a point. Equivalently, $X\subseteq \p^r$ is {linearly normal} if $h^0(X,\mathcal O_X(1))=r+1$, i.e., if $\cal H_X=|\cal O_X(1)|$.  
 
 A linear system $\cal L$ of dimension $r$ determines a map $\phi_\cal L: X\dasharrow \p^r$. The linear system $\cal L$ is said to be \emph{birational} if $\phi_\cal L$ is birational onto its image. 
 
 If $\cal L_1,\cal L_2$ are linear systems on the variety $X$, then we can consider their \emph{minimal sum} $\cal L_1\oplus \cal L_2$, which consists of the minimum linear system which contains all divisors of the form $D_1+D_2$ with $D_i\in \cal L_i$ with $i=1,2$. 
 
 If $\cal L$ is a linear system on the variety $X$ and $Z$ is a subscheme of $X$, we denote by $\cal L(-Z)$ the sublinear system of $\cal L$ consisting of all divisors of $\cal L$ containing the scheme $Z$. In particular, if $x\in X$ is a point, then $\cal L(-x)$ consists of all divisors in $\cal L$ containing $x$, whereas we denote by $\cal L(-2x)$ the set of all divisors in $\cal L$ containing the scheme defined by $\mathfrak m_x^2$, where $\mathfrak m_x$ is the maximal ideal corresponding to the point $x$. In particular, if $X\subseteq \p^r$ and $x\in X$, then $\cal H_X(-x)$ is cut out on $X$ by all hyperplanes passing through $x$, whereas, if $x\in \Reg(X)$, then $\cal H_X(-2x)$ is cut out on $X$ by all hyperplanes containing $T_{X,x}$, and it is called the  hyperplane system of $X$  \emph{tangent} at $x$. 
 
 \subsection{} Let $X \subseteq \p^r$. A \emph{family} $\cal F$ of subschemes of $X$ is for us a closed subscheme of the Hilbert scheme of $X$. The family is said to be \emph{filling} $X$ if, given a general point $x\in X$, there is a scheme of $\cal F$, i.e., parameterized by a point of $\cal F$, passing through $x$.

\section{Secant varieties and contact loci}\label{sec:cont}

 \subsection{} \label {sbs:a} Let $X\subseteq\p^r$ be a non--degenerate, projective variety of dimension $n$.
We will denote by $S(X)$ its {\it secant
variety}, namely the Zariski closure in $\p^r$ of the
set of all  \emph{proper secant lines} to $X$, i.e. lines
$\langle p_0,p_1\rangle$ with  $p_0,p_1$ distinct points in $X$. One has  $X=S(X)$ if and 
only if $X=\p^ r$.

One can consider the {\it abstract  secant variety} $S_X$
of $X$, i.e. $S_X\subseteq  {\rm Sym}^2(X)\times\p^r$ is the
Zariski closure of the set of all pairs $([p_0,p_1],x)$ such
that $p_0,p_1\in X$ are distinct points and $x\in
\langle p_0,p_1\rangle$. One has the surjective morphism $p_X:S_X\to
S(X)\subseteq\p^r$, i.e., the projection to the second factor.
Hence:
\begin{equation} \label{defect} s(X):=  \dim (S(X))\leq \sigma(X):=
\min\{r,\dim(S_X)\}= \min\{r,2n+1\}.\end{equation}

The integer $\sigma(X)$ is called the {\it expected
dimension} of $S(X)$. One says that $X$  is $1$--{\it defective}, 
or simply \emph{defective}, when
strict inequality holds in (\ref{defect}). One defines
$$\delta(X):=\sigma(X)-s(X)$$

\noindent to be the {\it defect} of $X$. A slightly different concept of defect, is the
one of \emph {fibre defect} defined as
$$f(X):=2n+1-s(X)$$
which is the dimension of the general fibre of $p_X$, i.e., the dimension of the
family of secant lines to $X$
passing through the general point of $S(X)$. 

\begin{remark}\label {rem:def} One has 
$$f(X)=\delta(X) \quad {\rm  if }\quad r\geq 2n+1$$
\noindent otherwise
\begin{equation}\label{eq:part}
f(X)=\delta(X)+2n+1-r.
\end{equation}
In  any case $f(X)\geq \delta(X)$ hence $f(X)=0$ implies non--defectiveness. 
\end{remark}

\subsection{} Let $X_0,X_1$ be reduced projective schemes in $\p^ r$.
The \emph{join} $J(X_0,X_1)$
of $X_0,X_1$ is the closure
in $\p^ r$ of the set\medskip

\centerline{$\{ x\in \p^r: x\in \langle p_0,p_1\rangle$ with
$p_i\in X_i$ distinct points, for $0\leq i\leq 1$,  $\}$.} \medskip

If $X$ is a variety, then  $S(X)=J(X,X)$.
If $X$ is a reducible projective scheme, then $J(X,X)$ is in general reducible
and not pure. 
It is convenient for us to denote this by $S(X)$
and call it the  {\it secant variety} of $X$.  
The reduced schemes $X$ we will be considering
 next, even if reducible, will have the property that $S(X)$ is pure. 
 For them all the considerations in \S \ref {sbs:a} can be repeated verbatim. 
 
 \subsection{}  There are several notions
 of \emph {tangential variety} to a given variety $X$ (see \cite {Zak} or  \cite {harris}, \S 15 as general references). We will adopt the following.
 
Let $X\subset \p^ r$ be an irreducible variety of dimension $n$. Consider the rational map
 
 $$s: X\times X\map \G(1,r)$$
 \noindent which sends the point $(p_0,p_1)$ off the diagonal $\Delta$ to the line
 $\langle p_0,p_1\rangle$. Resolve the indeterminacies of $s$ obtaining a morphism $\tilde s$ defined on
 some blow--up $\cal X$ of $X\times X$.

 \begin{definition}\label{def:tan} A \emph { tangent line} to $X$ is any line corresponding
 to images via the map $\tilde s$ of points of $\mathcal X$ mapping to $\Delta$. 
 
 We will denote by ${\rm Tan}(X)$ the locus of points of $\p^ r$ lying on some
 tangent line to $X$. This is the \emph{tangential variety} to $X$. It does not depend on the blow--up $\cal X$ of $X\times X$.
 \end{definition}
 
 \begin{remark}\label {rem:tan} Tangent lines are all limits of secants $\langle p_0,p_1\rangle$
 with $p_0,p_1\in X$ distinct points, when $p_0,p_1$ come toghether. Therefore 
 ${\rm Tan}(X) \subseteq S(X)$. Moreover 
${\rm Tan}(X)$ is a closed subvariety of $\p^ r$, containing the variety

$$T(X):=\overline {\bigcup_{x\in \Reg(X)} T_{X,x}}.$$

It is a well known consequence of Fulton--Hansen connectedness theorem (see \cite  {fh})
that if $X\subset \p^r$, with $r\geq 2n+1$, is smooth, then either $X$ is not defective, in which case $\dim ({\rm Tan}(X))=2n$ or $S(X)={\rm Tan}(X)$. 

Note that ${\rm Tan}(X)=T(X)$ if $X$ is smooth.
 \end{remark}

\subsection{}  

%%%
\begin{comment}
If $X\subset \p^r$ is any irreducible projective variety of dimension $n$,
if $x,y\in X$ are general points, and if $H \in  \mathcal H(-2x -2y)$ is a general hyperplane section tangent at $x$ and $y$ (i.e., a general \emph{bitangent hyperplane section}), we can consider the \emph{contact variety} of $H$, i.e., the union $\Sigma :=\Sigma_{x,y}(H)$
of the irreducible components of $\rm{Sing}(H)$ containing $x$ and $y$. Since $x,y$ are general points, an obvious monodromy argument shows that $\Sigma$ is equidimensional, and we denote by $\nu(X)$ its dimension, which we will call the \emph{singular defect} of $X$. Of course $\nu(X)\leq n-1$, and we set $\nu(X)=-1$ if $\mathcal H(-2x-2y)$ is empty. We also set $h(\Sigma) = r - dim(\mathcal H(-\Sigma))$. This is the number of conditions imposed by $\Sigma$ on the divisors of $\mathcal H$ containing it.
\end{comment}
%%%

Terracini's Lemma
describes the tangent space to a secant variety
at a general point of it
(see \cite{Terr1} or, for modern versions, \cite{Adl, WDV, Dale1, Zak}). In the case of interest for us, we may state it as follows.

\begin{theorem} [Terracini's Lemma] \label{terracini1}  Let $X\subset\p^r$ be a  projective variety. 
If $p_0,p_1\in X$ are
general points and $x\in \langle p_0,p_1\rangle$ is a general point of $S(X)$, then
$$T_{S(X),x}=\langle T_{X,p_0},T_{X,p_1}\rangle = T_{X,p_0,p_1}.$$
%Moreover if $x,y\in X$ are general points, if $H \in  \mathcal H(-2x -2y)$ is a general bitangent hyperplane section of $X$ and $\Sigma$ is its contact variety of dimension $\nu(X)$ then $$2 \leq h(\Sigma) \leq 2(1 + \nu(X)) - \delta(X)$$ $$2\nu(X)\geq \delta(X).$$ 
%In particular, if $X$ is defective then $n - 1 \geq \nu(X) > 0$, i.e., the general bitangent hyperplane section of $X$ has a contact variety of positive dimension.\end{comment}
\end{theorem}

\begin{remark}\label {rem:ff} Let us keep the notation of Terracini's Lemma. One computes 

$$f(X)=\dim (T_{X,p_0}\cap T_{X,p_1})+1.$$
\end{remark}

\begin{lemma}\label {lem:ff} Let $X\subseteq \p^r$ be a non--degenerate variety of dimension $n\geq 2$ and let $Y$ be a general hyperplane section of $X$. Then 
$$f(Y)=\max \{0, f(X)-1\}.$$
\end{lemma}

\begin {proof} The assertion is trivial if $X=\p^r$, so we may assume that $X$ is strictly contained in $\p^r$ and it is not a linear subspace because it is non--degenerate. 

Let $Y$ be the section of $X$ with a general hyperplane $H$ and let $p_0,p_1$ be general points of $Y$. Then $p_0,p_1$ are also general points of $X$.  One has $T_{Y,p_i}=H\cap T_{X,p_i}$ for $i=0,1$. 

If $T_{X,p_0}\cap T_{X,p_1}=\emptyset$, so that $f(X)=0$, then also $T_{Y,p_0}\cap T_{Y,p_1}=\emptyset$ and $f(Y)=0$, and the assertion holds. 

Suppose now $T_{X,p_0}\cap T_{X,p_1}\neq \emptyset$. Note that neither $p_0$ nor $p_1$ lie in $T_{X,p_0}\cap T_{X,p_1}$, otherwise the general tangent space to $X$ would contain the general point of $X$ and therefore $X$ would coincide with its general tangent space, so it would be a linear space, a contradiction. Then $T_{X,p_0}\cap T_{X,p_1}$ is skew with the line $\langle p_0,p_1\rangle$, because $p_0,p_1$ are not in $T_{X,p_0}\cap T_{X,p_1}$ and if a point $p\in \langle p_0,p_1\rangle$ different from $p_0,p_1$ lies in $T_{X,p_0}\cap T_{X,p_1}$, then the whole line $\langle p_0,p_1\rangle=\langle p,p_1\rangle=\langle p_0,p\rangle$ would be contained in $T_{X,p_0}\cap T_{X,p_1}$, a contradiction. Then by the generality of the hyperplane $H$, $H\cap T_{X,p_0}\cap T_{X,p_1}=T_{Y,p_0}\cap T_{Y,p_1}$ has dimension $\dim (T_{X,p_0}\cap T_{X,p_1})-1$ and the assertion follows. \end{proof}

\subsection{}  Given an irreducible projective variety $X\subseteq \p^ r$ of dimension $n$, 
the \emph{Gauss map} of $X$ is the rational map
$$g_X: X\dasharrow \G(n,r)$$

\noindent  defined at the smooth points of $X$ by mapping $x\in \Reg(X)$
to the tangent space $T_{X,x}$.  It is well known that, if $x\in X$ is a general point, then the closure of the
fibre of $g_X$ through $x$ is a linear subspace $\Gamma_{X,x}$ of $\p^ r$, which we  denote by $\Gamma_x$ if there is no danger of confusion.

\begin{definition}\label {def:gauss} In the above setting, $\Gamma_{x}$ is called the general \emph {Gauss fibre}
of $X$ and its dimension is called the \emph{tangential defect} of $X$,  denoted by $t(X)$.
We will set $\theta(X)=t(S(X))$. \end{definition}

Note that, if $X$ is smooth, then $t(X)=0$ (see \cite  {Zak}).

\begin{remark} \label {rem:hg} In the above setting, there is a dense open set $U\subset \Reg(X)$, which via $g_{X|U}$ is fibred over $g_X(U)$ in open subsets of projective spaces of the same dimension $t(X)$. This implies that  
if $Y=X\cap H$ is the section of $X$ with a general hyperplane and $x\in Y$, then 
$\Gamma_{X,x}\cap H= \Gamma_{Y,x}$. Hence

$$t(Y)= \max \{0,t(X)-1\}.$$
\end{remark}

%Let  $X\subset \p^ r$ be non--degenerate variety such that $s(X)<r$. Terracini's Lemma implies that $\theta(X)>0$. More precise information about $\theta(X)$ will be provided below.

\subsection {} \label{ssec:paf}Given $x\in S(X)$ a general point,
i.e., $x\in \langle p_0,p_1\rangle$ is a general point, with $p_0,p_1\in X$
general points, consider the Zariski closure of the set\medskip

\centerline{  $\{ p\in \Reg(X):  T_{X,p} \subseteq T_{S(X),x}=T_{X,p_0,p_1} \}$. }\medskip

\noindent  Note that this set depends only on $p_0,p_1$ and not on $x\in S(X)$. We will denote by $\Gamma_{X,p_0,p_1}$ the union of all irreducible components of this locus containing $p_0$ and $p_1$. Since $p_0,p_1$ are general points, by monodromy $\Gamma_{X,p_0,p_1}$ is equidimensional and we let  $\gamma(X)$ be its dimension, which, by generality of  $p_0,p_1$, does not depend on $p_0,p_1$.
Note that $\Gamma_{X,p_i}\subseteq \Gamma_{X,p_0,p_1}$
for $i=0,1$. We set
$$\Pi_{X,p_0,p_1}= \langle \Gamma_{X,p_0,p_1} \rangle.$$
\noindent Since  $\Pi_{X,p_0,p_1}$ contains $\langle p_0,p_1\rangle$,
then it contains $x$.

If there is no danger of confusion, we write $\Gamma_{p_0,p_1}$ and $\Pi_{p_0,p_1}$
instead of $\Gamma_{X,p_0,p_1}$ and  $\Pi_{X,p_0,p_1}$.

\begin{definition} In the above setting, we will call
$ \Gamma_{X,p_0,p_1}$  the {\it tangential contact locus} of
$X$ at $p_0,p_1$. We will call
$\gamma(X)$ the {\it contact defect} of $X$.
\end{definition}

\begin{proposition}\label {rem:hh} Let $X$ be a projective, defective variety, of dimension $n$,   let $Y=X\cap H$ with $H$ a general hyperplane and let $p_0,p_1$ be general points of $Y$. Then 
$\Gamma_{X,p_0,p_1}\cap H \subseteq \Gamma_{Y,p_0,p_1}$.

In particular, $\gamma(Y)\geq \gamma(X)-1$ and if the equality holds then 
$\Gamma_{X,p_0,p_1}\cap H =\Gamma_{Y,p_0,p_1}$. \end{proposition}

\begin{proof} First we prove that $T_{Y,p_0,p_1}=T_{X,p_0,p_1}\cap H$. It is clear that 
$T_{Y,p_0,p_1}\subseteq T_{X,p_0,p_1}\cap H$, so it suffices to prove that the two spaces have te same dimension. By the same argument we made in the proof of Lemma \ref {lem:ff}, we have that $T_{Y,p_0}\cap T_{Y,p_1}=H\cap T_{X,p_0}\cap T_{X,p_1}$. Then 
\[
\begin{aligned}
&\dim (T_{X,p_0,p_1}\cap H)=2n-\dim(T_{X,p_0}\cap T_{X,p_1})-1=&\cr
&=2n-(\dim(T_{Y,p_0}\cap T_{Y,p_1})+1)-1=2(n-1)-\dim(T_{Y,p_0}\cap T_{Y,p_1})=\dim(T_{Y,p_0,p_1})& \cr
\end{aligned}
\]
as wanted.

Now, let $x\in \Gamma_{X,p_0,p_1}\cap H$. Then $T_{X,x}\subseteq T_{X,p_0,p_1}$, hence
\[
T_{Y,x}\subseteq T_{X,x}\cap H\subseteq T_{X,p_0,p_1}\cap H=T_{Y,p_0,p_1},
\]
proving the assertion. 
\end{proof}

The following result is contained in \cite {CC4}, Proposition 3.9.

\begin{proposition}\label {lem:a} Let $X\subset \p^ r$ be a
 non--degenerate, projective variety such that $s(X)<r$. Let $p_0,p_1\in X$ be general
points. Then:

\begin {itemize}

\item  [(i)] Let  $q_0,q_1$ be general points on $\Gamma_{p_0,p_1}$,
such that $q_i$ specializes to $p_i$, for $i=0,1$.
Then $\Gamma_{p_0,p_1}=\Gamma_{q_0,q_1}$;

\item [(ii)] $\Gamma_{p_0,p_1}$ is smooth at $p_0,p_1$; moreover it is
either irreducible or it consists of two irreducible components of the same dimension $\gamma(X)$
each containing one of the points $p_0,p_1$ as its general point;

\item [(iii)] $f(\Gamma_{p_0,p_1})=f(X)$;

\item [(iii)] $\Pi_{p_0,p_1}=S(\Gamma_{p_0,p_1})$ equals the general Gauss fibre $\Gamma_{S(X),x}$ of $S(X)$ (in particular $S(\Gamma_{p_0,p_1})$ is a linear space), whereas  $\Gamma_{p_0,p_1}\neq \Pi_{p_0,p_1}$;

\item [(iv)] $\theta(X)= \dim( \Pi_{p_0,p_1})=2\gamma(X)+1-f(X)$.
\end{itemize}
\end{proposition}

%\begin {remark}\label {rem:unique} As a consequence of part (ii) of the Proposition \ref  {lem:a}, we have the following fact: if $p_0,p_1$ are general points of $X$, then  $\Gamma_{X,p_0,p_1}$ is the unique tangential contact locus containing $p_0,p_1$.\end{remark}

\subsection {} Let $X\subset\p^r$ be a non--degenerate, 
projective variety and let $p\in X$
 be a general point. Consider
the projection of $X$ with centre $T_{X,p}$ to $\p^ {r-n-1}$. We call
this a {\it general tangential projection} of $X$, and we
denote it by $\tau_{X,p}$, or by $\tau_{p}$, or simply by $\tau$.
 We denote by $X_{p}$, or by $X_1$, its image, and by $n_1$ its dimension. As a consequence of Terracini's Lemma, we have
 \begin{equation}\label{eq:tl}
 n_1=n-f(X).
 \end{equation}

 %\begin{remark}\label {rem:acc}  Note that $X_1$ is non--degenerate in $\p^ {r-n-1}$. Moreover, if $X$ is linearly normal, then also $X_1$ is linearly normal. Indeed, being $X_1$ not linearly normal, implies that $\dim (\vert \cal O_X(1)\otimes \cal I_{2p}\vert\geq r-n$. But then the hyperplane linear system $\cal H_X$ cannot be complete, because $\dim(\cal H_X(-2p))=r-n-1$. \end{remark}
  
 \begin{definition} Let $p_0,p_1\in X$ be general points. Let $
\Psi_{X,p_0,p_1}$ be the component of the fibre of $\tau_{X,p_1}$
containing $p_0$. 
It is called the \emph{projection contact locus}
of $X$ at $p_0,p_1$ and, by \eqref {eq:tl}, its dimension is 
$f(X)$.
\end {definition}

If there is no danger of confusion, we will write $\Psi_{p_0,p_1}$
rather than $\Psi_{X,p_0,p_1}$.

\begin{remark}\label{conctactin} One has that  $\Gamma_{p_0,p_1}$
contains $\Psi_{p_0,p_1}$. Indeed, if $p\in \Psi_{X,p_0,p_1}$ is a general point (in particular if $p=p_0$), then the 
projection of $T_{X,p}$ from $T_{X,p_1}$ is the tangent space to $X_1$ at the point $\tau(p)=\tau(p_0)$. In particular $T_{X,p}\subset T_{X,p_0,p_1}$.
Hence
\begin{equation}\label{eq:cc}
\gamma(X) \geq f(X).
\end{equation}
Actually the above argument implies that $\gamma(X)=t(X_1)+f(X)$.
Therefore $\gamma(X) =f(X)$ holds
if and only if the Gauss map of $X_1$ is generically finite to its image,
which is equivalent to say that it is birational to its image (see \cite {CC4}, 
Remark 3.6). Thus, if $\gamma(X) =f(X)$, then $\Psi_{p_0,p_1}$ is the
irreducible component of $\Gamma_{p_0,p_1}$ containing $p_0$.

If $X$ is defective, then $f(X)>0$, so that
 $\gamma(X)$  is also positive and $X_1$ is strictly
 contained in $\p^ {r-n-1}$. Otherwise $n_1=r-n-1$ by  \eqref {eq:tl} yields $f(X)=2n+1-r$ which, together with \eqref {eq:part}, implies $\delta(X)=0$, a contradiction. 
 
By Lemma \ref {lem:ff}, if $Y=X\cap H$ with $H$  a general hyperplane, then
$$f(Y)= \max \{0,f(X)-1\}.$$
This implies that, if $p_0,p_1$ are
general points of $Y$, then $\Psi_{Y,p_0,p_1}= H\cap \Psi_{X,p_0,p_1}$. 
\end{remark}

\subsection {} Other relevant items related to the secant variety
$S(X)$ are the so called \emph
    {entry loci}.

\begin{definition}\label{def:el} Let $x\in  S(X)$ be a point.
We define the \emph{entry locus} $E_{x}$
of $x$ with respect to $X$ as the closure of the set\medskip

\centerline{ $\{ z\in X:$  there is $y\in X$ with $y \neq z$ and $ x\in \langle y,z\rangle\}$.}
\end{definition}

A useful information about the entry loci is provided by the following result
(see \cite {CC4}, Proposition 3.13 or \cite {IR}, Proposition
2.2).

\begin{proposition}\label {lem:dimel} Let $X\subset \p^ r$ be an irreducible
variety with $s(X)<r$ and let $x\in S(X)$ be a general point. Then
$E_x$ is pure of dimension $f(X)$.  
\end{proposition}

\begin{remark}\label {rem:elocus} Terracini's Lemma implies that the entry locus $E_x$ is contained
in the tangential contact locus $\Gamma_{p_0,p_1}$ for all distinct
pairs of points $p_0,p_1\in X$ such that $x\in \langle p_0,p_1\rangle.$
Therefore again $f(X)\leq \gamma(X)$.\end{remark}

\subsection {} \label {susec:difetti} Let $X\subset\p^r$ be a
projective variety. Let $x\in X$ be a general point and let $H$ be a general hyperplane 
containing $T_{X,x}$. This is a \emph {general tangent hyperplane} to $X$.
Consider the Zariski closure $\Gamma_{X,x}(H)$ of the set
$$\{ p\in \Reg(X):  T_{X,p} \subseteq H \}.$$
By the genericity assumptions and by the aforementioned properties of the Gauss map,
this is a linear subspace of $\p^ r$. We call it the \emph{general tangent hyperplane contact locus} of $X$ at
$x$ and we will denote by $d(X)$ its dimension, which does not depend on $x$ and $H$. If there is no danger of confusion, we may write  $\Gamma_{x}(H)$ rather than  $\Gamma_{X,x}(H)$.
Of course $\Gamma_x(H)$ contains
$\Gamma_{x}$, thus 
$$d(X)\geq t(X).$$ 

\noindent Moreover 
the \emph {dual variety} $X^ *$ of $X$ has dimension 
$$\dim(X^ *)=r-1-d(X)$$
\noindent and therefore $d(X)$ is called the \emph {dual defect} of $X$.

Similarly, let  $p_0,p_1\in X$ be general points, and let $H$ be
a general hyperplane containing $T_{X,p_0,p_1}$, i.e. a \emph {general 
bitangent hyperplane} to $X$.
Consider the Zariski closure of the set
$$\{ p\in \Reg(X):  T_{X,p} \subseteq H \}.$$

\begin{definition}\label {def:bit}  The union
 $\Gamma_{X,p_0,p_1}(H)$ of all irreducible components
of the aforementioned locus  containing either $p_0$ or $p_1$ is called the \emph{general bitangent hyperplane contact locus} of $X$ at
$p_0,p_1$. We will denote by $\epsilon(X)$ its dimension, which, by genericity,  does not depend on $p_0,p_1$ and $H$. We will set  
$$\Pi_{X,p_0,p_1}(H):=\langle \Gamma_{X,p_0,p_1}(H)\rangle.$$
\end{definition}

We will write $\Gamma_{p_0,p_1}(H)$ and $\Pi_{p_0,p_1}(H)$ instead of $\Gamma_{X,p_0,p_1}(H)$ and $\Pi_{X,p_0,p_1}(H)$ if there is no danger of confusion.

Of course $\Gamma_{p_0,p_1}(H)$ contains
$\Gamma_{p_0,p_1}$, and therefore 

\begin{equation}\label {eq:epsi}
\epsilon(X)\geq \gamma(X).\end{equation}

\begin{remark}\label {rem:acca}  By  \eqref {eq:cc} and \eqref {eq:epsi}  we see that 
\[
\epsilon(X)\geq \gamma(X)\geq f(X).
\]
More precisely, looking at the general tangential projection, we have that $\epsilon(X)> f(X)$ if and only if $d(X_1)=\epsilon(X)-f(X)>0$, i.e., if and only if
$X_1$ is dual defective. 

%Let $Y=X\cap H$ be a general hyperplane section of $X$. With an argument similar to the one of the proof of Proposition \ref {rem:hh} one proves that 
%$$\epsilon(Y)= \max \{0,\epsilon(X)-1\}.$$
\end{remark}

Recall the following result which parallels Proposition \ref  {lem:a} (see  \cite {WDV}, Theorem 1.1 and \cite {CC5}, Theorem 2.4):

\begin{theorem}\label{terracini}  Let $X\subset\p^r$
be a defective, projective variety. Let $p_0,p_1\in X$ be
general points and $H$  a general hyperplane 
containing $ T_{X,p_0,p_1}$. One has:

\begin{itemize}
\item [(i)] $\Gamma_{p_0,p_1}(H)$ is smooth at $p_0,p_1$; moreover it is
either irreducible or it consists of two irreducible components of the same dimension $\epsilon(X)$
each containing one of the points $p_0,p_1$ as its general point;

\item [(ii)] $f(\Gamma_{p_0,p_1}(H))=f(X)$;

\item [(ii)] $\Pi_{p_0,p_1}(H)=S(\Gamma_{p_0,p_1}(H))$, whereas  $\Gamma_{p_0,p_1}(H)\neq \Pi_{p_0,p_1}$;

\item [(iv)] $\dim( \Pi_{p_0,p_1}(H))\leq 2\epsilon(X)+1-f(X)$.
\end{itemize}
\end{theorem}

Moreover we have the following:

\begin{proposition}\label{difetti} Let $X\subset \p^r$ be a defective
projective variety of dimension $n$. Then:

\begin{itemize}

\item [(i)] $r\geq n+3$ and $f(X) \leq n-1$. Moreover if $f(X)=n-1$, 
which occurs if $r=n+3$, then 
$X$ is either a cone over a curve or a cone over the Veronese surface $V_{2,2}$ in $\p^ 5$;

\item [(ii)] if $f (X)=\gamma(X):=m$
then for  $p_0,p_1\in X$ general, $\Gamma_{p_0,p_1}$ is a quadric of rank at least $2$ in $\p^ {m+1}$. Furthermore $\Psi_{p_0,p_1}$ is an
irreducible component of $\Gamma_{p_0,p_1}$, hence  either it coincides with  $\Gamma_{p_0,p_1}$ or 
it is a linear subspace of dimension $m$;

\item [(iii)] in particular if $f (X)=\epsilon(X):=m$ then for general choices of $p_0,p_1$ and $H$ containing $T_{X,p_0,p_1}$, one has  $\Gamma_{p_0,p_1}(H)=\Gamma_{p_0,p_1}$ and (ii) holds. 
\end{itemize}
\end{proposition}

\begin{proof} Since $X_1$ is non--degenerate and strictly contained 
in $\p^ {r-n-1}$ (see Remark \ref {conctactin}), 
we have $r-n-1\geq 2$, proving the first assertion in (i).
Similarly $n_1=n-f(X)\geq 1$, proving the second assertion in (i).
The rest of (i) follows from  \cite {ChCi}, Theorem 3.5.

 Let us prove (ii). 
By Theorem \ref {lem:a}, (iv), 
we have that $\Gamma_{p_0,p_1}$ is a hypersurface
in $\Pi_{p_0,p_1}$, which has dimension $m+1$.  It has to be a quadric otherwise the general
secant to $X$ would be a trisecant. By the Trisecant Lemma, this cannot happen unless 
$r=n+1$ (see \cite {CC}), which is not the case.  Then $\Gamma_{p_0,p_1}$ is a quadric 
but it cannot be a linear space, hence $\Gamma_{p_0,p_1}$ 
has rank at least 2. The rest of  assertion (ii) is clear.

Part (iii) follows from (ii) and Remark \ref {rem:acca}.  \end{proof} 

Varieties $X$ for which part (ii) of Proposition \ref {difetti} holds, i.e., for  $p_0,p_1\in X$ general $\Gamma_{p_0,p_1}$ is a quadric of rank at least $2$,
are \emph {locally quadratic entry locus (LQEL)} varieties in the sense of \cite {rus} (see \cite{R2} for the theory of LQEL varieties and applications). If
$\gamma(X)=f(X)=1$, we will say that they present
the  \emph{locally conic entry locus (LCEL)} case.

According to Scorza's terminology (see  \cite {Scorza1}), we will say that
a defective variety $X\subset \p^ r$ of dimension $n$ is of the \emph{species}
$k$ if $k=n-\gamma(X)$. The \emph{top species} for varieties of
dimension $n$ is $n-1$. Such a variety $X$ has $\gamma(X)=f(X)=1$, hence
it presents the LCEL case. Note that varieties $X$ of dimension $n$ 
of a given species $k$ could be
further classified according to the value of the 
invariant $\epsilon(X)\geq n-k=\gamma(X)$. There are however some constraints, as the the following
lemma indicates.

\begin{lemma}\label {lem:mbus}  Let $X\subset \p^ r$ be a non--degenerate variety if dimension $n$.  Then $\epsilon(X)=n-1$ if and only if $\gamma(X)=n-1$.
\end{lemma}

\begin{proof} One implication is trivial. As for the other, assume $\epsilon(X)=n-1$. If $f(X)=n-1$ then also $\gamma(X)=n-1$ and we are done. If $f(X)<n-1$ then  the image $X_1$ of a general tangential projection of $X$ has a curve as a dual
(see Remark \ref {rem:acca}). Then $X_1$, which is dual to a curve, is a hypersurface, and $t(X_1)=n_1-1=n-f(X)-1$. By Remark \ref {conctactin} we have $\gamma(X)=t(X_1)+f(X)=n-1$, as wanted. \end{proof}

\subsection {}\label {subsec:sff} Let $X\subset \p^ r$ be a non--degenerate variety. 
Let $x$ be a smooth point of $X$. One can consider the \emph{second fundamental
form} ${\rm II}_{X,x}$ of $X$ at $x$ (see \cite {GrHarr} as a general reference). 
This is a linear system of quadrics in $\p_{X,x}:=\p(T_x(X))$. Recall that 
$T_x(X)$ is the \emph {Zariski tangent space} to $x$ at $X$, hence $\p_{X,x}$, the projective space associated to $T(X)$, is a projective space of dimension $n-1$.

Consider the general tangential projection $\tau_x: X\map X_1$.
This is not defined at $x$. Let us blow--up $x$ in $X$, getting $\pi_x: \tilde X\to X$
 and consider the exceptional divisor
$E:=E_x$, which may be identified with $\p_{X,x}$. We have therefore the rational map
$$\tilde \tau_x=\tau_x\circ \pi_x: \tilde X\map X_1\subseteq \p^ {r-n-1}$$
\noindent whose restriction to $E$ is defined by a linear system of quadrics, which is
precisely ${\rm II}_{X,x}$. We will denote by $B_{X,x}$ its base locus scheme, which consists of 
the points of $E_x$ corresponding to directions of lines having intersection multiplicity 
larger than $2$ with $X$ at $x$ (see \cite {GrHarr}). Hence it contains
the subscheme $L_{X,x}$ of $E$ whose points correspond to tangent directions of the lines 
on $X$ passing through $x$.

We may write ${\rm II}_{x}, L_x, B_x$ raher than ${\rm II}_{X,x}$, $L_{X,x}$, $B_{X,x}$ etc., if there is no danger of confusion. We will set 
\[
b(X)=\dim(B_x)\quad \text{and}\quad \ell(X)=\dim(L_x).
\]
 Note that if $\ell(X)\geq 0$, then 
 $\ell(X)+n-1$ is the dimension of the union of components of the Hilbert scheme of lines contained in $X$ and filling $X$. 

It is useful to collect here the main results about the second fundamental form 
that we will use later.  First of all, it is a classical result by Terracini  (see \cite {GrHarr} for a 
modern treatment) that

\begin{equation*}\label {eq:terr}
\dim (T(X))=n+1+\dim(\tilde \tau_x(E)) 
\end{equation*}
Therefore, if $X$ is smooth and  defective, then by Remarks \ref {rem:def} and \ref {rem:tan} one has 

\begin{equation}\label {eq:terr2}
\dim(\tilde \tau_x(E))=n-f(X)=\dim(X_1).
\end{equation}
Thus $\tilde \tau_{x\vert E}$ is dominant onto $X_1$, with general fibres of dimension $f(X)-1$. Since $X_1$ is non--degenerate in $\p^ {r-n-1}$, one has

\begin{equation}\label {eq:terr3}
\dim  ({\rm II}_{X,x} )=r-n-1.
\end{equation}

As the following result shows, \eqref  {eq:terr3} holds, to a certain extent, even relaxing the smoothness assumption on $X$.

\begin{lemma}\label {lem:secondff} Let $X\subset \p^ r$ be a 
non--degenerate, defective
variety of dimension $n$. Let $x\in X$ be a general point and assume that the general fibre of the tangential projection $\tau_x$ is irreducible and contains $x$. Then
$\tilde \tau_{x\vert E}$ is dominant onto $X_1$, with general fibres of dimension $f(X)-1$. Hence
 \eqref  {eq:terr2} and \eqref {eq:terr3} hold. Furthermore
 \begin{equation}\label{eq:bound}
 r\leq \frac {n(n+3)}2.
 \end{equation}
\end{lemma}

\begin{proof} The assumption  implies that $\tilde \tau_{x\vert E}: E\map X_1$ is dominant. Since
$\tilde \tau_{x\vert E}$ is defined by a linear system of quadrics, one has
$$ r-n-1\leq \frac {(n-1)(n+2)} 2$$
and \eqref {eq:bound} follows. The rest of the assertion is clear. \end{proof}

\section{The codimension 4 case}\label{sec:4}

Let  $X\subset \p^ r$ be an irreducible, non--degenerate, defective
variety of dimension $n$ which is not a cone over a curve or over the Veronese
surface $V_{2,2}$. Since defective surfaces and threefolds are classified (see \cite {WDV, ChCi, CC3, Scorza}), we may assume $n\geq 4$. 

 According to Proposition \ref {difetti},
we have $r\geq n+4$. The next theorem  classifies the cases in with the equality holds. 

\begin{theorem}\label {thm:lowcod} Let $X\subset \p^ {r}$, $r\geq n+4$,
be a non--degenerate, projective, defective variety of dimension $n\geq 4$. If $f(X)=n-2$, which happens in particular if $r=n+4$, then $X$ is of one of
the following types:

\begin{itemize} 

\item [(i)] $X$ is a cone over a non--defective surface;

\item [(ii)] $X$ sits in a $(n+1)$--dimensional cone over a curve;

\item [(iii)]  $X$ sits in a $(n+1)$--dimensional cone over the Veronese 
surface $V_{2,2}$ in $\p^ 5$;

\item [(iv)] $X$ is a cone over the Veronese threefold $V_{3,2}$ in $\p^ 9$ or over a projection of $V_{3,2}$ in $\p^8$ or $\p^7$;

\item [(v)] $X$ is either a cone over the Segre variety ${\rm Seg}(2,2)\subset \p^ 8$ with vertex of dimension $n-5$, or a cone over a hyperplane section of ${\rm Seg}(2,2)$.
\end{itemize}
\end{theorem}

\begin{proof}  First let us prove that if $r=n+4$, then $f(X)=n-2$. Indeed, if $r=n+4$, then  $r<2n+1$ because $n\geq 4$, so that by \eqref {eq:part} we have $f(X)=\delta(X)+n-3\geq n-2$. By Proposition \ref {difetti}, we have  $f(X)<n-1$, hence $f(X)=n-2$. 

Next, if $f(X)=n-2$ then $s(X)=2n+1-f(X)=n+3$. So we may project generically to $\p^{n+4}$ getting an  image $X'$ of $X$. Note that the projection $X\to X'$ is finite and bijective. 

\begin{claim}\label{cl:part} If $X'$ is as in one of the cases (i)--(v), then also $X$ is as in the corresponding case (i)--(v).
\end{claim}

\begin{proof}[Proof of the Claim \ref {cl:part}] The assertion is trivial if $X'$ is as in (i), (iv) or (v). 

Let us assume that $X'$ is as in (ii), so that $X'$ lies in a cone $V'$ over a curve. Then $X'$ is swept out by a 1--dimensional family $\cal F'$ of varieties of dimension $n-1$, each lying in a $n$--space, i.e., a ruling of $V'$. 

Suppose first that the general variety in $\cal F'$ spans a $n$--space.
Since the projection $X\to X'$ is finite and birational, then $X$ is also swept out by the $(n-1)$--dimensional varieties of a family $\cal F$, which are in one--to--one correspondence with the varieties in $\cal F'$. Since the projection from $X$ to $X'$ is generic, also the projection of the general member of $\cal F$ to the general member of $\cal F'$ is finite and bijective, and  they both span a $n$--space. These $n$--spaces, as well as their projections to $\p^{n+4}$, meet along a $(n-1)$--space, hence they form a cone $V$ over a curve, in which $X$ sits. So $X$ is as in (ii). 

Suppose next that the general variety in $\cal F'$ does not span a $n$--space. This means that the general variety in $\cal F'$ is a $(n-1)$--space, and the same happens for the general member of $\cal F$. Note now that two general members of $\cal F'$ intersect along a $(n-3)$--space. Since $X'$ is a general projection of $X$, then also  the members of $\cal F$  intersect along a $(n-3)$--space. Then by Lemma 4.1 in \cite {CC3}, one of the following cases occur:
\begin{itemize}
\item [(a)] the span of the linear spaces in $\cal F$ is a $k$--space with $k\leq n+2$;
\item [(b)] all the linear spaces in $\cal F$ contain the same $(n-3)$--space;
\item [(c)] the general member of $\cal F$ cuts a fixed $(n-1)$--space along a $(n-2)$--space.
\end{itemize}

Case (a) cannot happen, because then $X$ would span at most a $(n+2)$--space, a contradiction. In case (b), $X$ would be a cone, as well as $X'$, a contradiction again. So only case (c) happens, which implies that $X$ also lies on a cone $V$ over a curve, which projects to $V'$.  

The case in which $X$ is as in (iii) is similar and can be left to the reader. \end{proof}

By Claim \ref {cl:part}, by substituting $X$ with $X'$, we may assume $r=n+4$.

Let $Y=X\cap P$ be the threefold section of $X$ with a general 7--dimensional linear subspace $P$.  By Lemma \ref  {lem:ff}, one has $f(Y)=1$, so $Y$ is 
a defective threefold which is non--degenerate in $\p^ 7$, hence $Y$ is in the list of Theorem 1.1 of  \cite {ChCi}.

If $Y$ is a cone, then we are in case (i). In fact, if $X$ has degree $d$, also $Y$ has degree $d$ and the vertex $v$ of $Y$ is a point of $Y$ of multiplicity $d$. Since $Y$ is a general 3--fold section of $X$, we see that $v$ is also a point of multiplicity $d$ for $X$, and therefore $X$ has a locus $\Pi$ of points of multiplicity $d$ which is cut out by a general $\p^7$ in a point, so $\Pi$ is a linear space of dimension $n-3$, and $X$ is a cone with vertex $\Pi$ over a surface.

Suppose $Y$ is not a cone and sits in a $4$--dimensional cone over a curve with vertex a plane. In this case we have $\gamma(Y)=2$. Indeed, if $p_0,p_1$ are general points in $Y$, 
then  $\Gamma_{Y,p_0,p_1}$  consists of a pair of irreducible surfaces, each passing through one of the points $p_0,p_1$, both spanning at most a $3$--space and together spanning $\Pi_{Y,p_0,p_1}$ which has dimension 4 (see Proposition \ref {lem:a}).

Let $x\in Y$ be a general point, which is also a general point of $X$. Consider the tangential projection $\tau_{X,x}$. The image $X_1$ of $X$ has dimension $n-f(X)=2$. The restriction of $\tau_{X,x}$ to $Y$ coincides with $\tau_{Y,x}$, whose image $Y_1$ has also dimension $3-f(Y)=2$. This implies that $X_1=Y_1$. On the other hand it is immediate to see that $Y_1$ is a cone over a curve. Then $t(X_1)=t(Y_1)=1$. By Remark \ref {conctactin}, we have $\gamma(X)=t(X_1)+f(X)=n-1$. 

Fix again two general points $p_0,p_1$ in $Y$, which are also general points on $X$.
By iterated applications of Proposition \ref {rem:hh}, we have that $\Gamma_{Y,p_0,p_1}=\Gamma_{X,p_0,p_1}\cap P$.  
By the above description of $\Gamma_{Y,p_0,p_1}$ and by Proposition \ref  {rem:hh}, $\Gamma_{X,p_0,p_1}$  is the union of a pair of irreducible Weil divisors
on $X$, each passing through one of the points $p_0,p_1$, each spanning at most a $n$--space, and together spanning $\Pi_{X,p_0,p_1}$ which has dimension $n+1$. 

 Suppose first the components of $\Gamma_{Y,p_0,p_1}$  both span a $3$--space, so that the components of $\Gamma_{X,p_0,p_1}$ both span a $n$--space. 
Then $X$ is swept out by a $1$--dimensional
family $\mathcal Z$ of irreducible divisors such that the general $Z$ in $\mathcal Z$ spans a $n$--space and 
two general such $n$--spaces span a $(n+1)$--space. This implies that all these $n$--spaces contain the same $(n-1)$--dimensional subspace $\Pi$. By projecting $X$ form $\Pi$ we obtain a curve, so we are in case (ii). 

Suppose next that the components of $\Gamma_{Y,p_0,p_1}$  do not span a $3$--space, so they are planes, pairwise meeting at a point. Then $X$ is swept out by a $1$--dimensional
family $\mathcal Z$ such that the general $Z$ in $\mathcal Z$ is a $(n-1)$--space and 
two general such $(n-1)$--spaces span a $(n+1)$--space, hence they meet along a $(n-3)$--space. Then the same argument based on Lemma 4.1 in \cite {CC3} we made in Claim \ref {cl:part}, implies that either $X$ is a cone with vertex a $(n-3)$--space, hence $Y$ is also a cone, a contradiction, or $X$ sits in a cone with vertex a $(n-1)$--space over a curve, and we are in case (ii) again.

Suppose now that $Y$ is not a cone and is contained in a $4$--dimensional cone $V$ over $V_{2,2}$ with vertex a line $L$. Then we claim we are in case (iii). 
Again we have $\gamma(Y)=2$. Indeed, if $p_0,p_1$ are general points in $Y$,  then  $\Gamma_{Y,p_0,p_1}$  varies in a 2--dimensional family $\cal F$ of irreducible surfaces spanning linear spaces of dimension 4 and contained in 3--dimensional quadric cones with vertex $L$ projecting a  conic of $V_{2,2}$.
Let $S=Y\cap \Pi$ be  the intersection of $Y$ with a general hyperplane $\Pi$ of $P\cong \p^7$, which spans the linear space $\Pi$ of dimension 6 and sits in the cone over $V_{2,2}$ with vertex  $v=\Pi\cap L$. Then $S$ has a 2--dimensional family $\cal C$ of irreducible curves spanning linear spaces of dimension 3 and contained in quadric cones with vertex $v$ projecting a conic of $V_{2,2}$: they are the hyperplane sections  of the surfaces   in $\cal F$. 

\begin{claim}\label{cl:parton} Let $S$ be the surface as above. Then there is no irreducible curve $\Gamma$ in $\Pi$ containing $v$ and such that the projection of $S$ from a general point of $\Gamma$ is a cone over the Veronese surface $V_{2,2}$.
\end{claim}

\begin{proof}[Proof of the Claim \ref {cl:parton}] We argue by contradiction and assume there is such a curve $\Gamma$. Then for a general point $x\in \Gamma$, the surface $S$ sits in the cone over $V_{2,2}$ with vertex $x$, and therefore it contains a family $\cal C_x$ of curves spanning a linear space of dimension 3 and contained in a quadric cone with vertex $x$ projecting a conic of $V_{2,2}$. 

The family $\cal C_x$ cannot be independent on $x$. Suppose in fact that $\cal C_x=\cal C$ is constant. Then if $C$ is a general curve in $\cal C$, it sits in infinitely many quadric cones with vertices moving on $\Gamma$. Then $C$ would have either degree 3 or 4. However, if $\deg(C)=4$ then $C$ is the complete intersection of two quadrics in $\p^3$, and it cannot lie on infinitely many cones, so this case is not possible. If $\deg(C)=3$, then $C$ could lie on infinitely many quadric cones, but their vertices should move on $C$. This implies that $\Gamma=C$, which is impossible, since $C$ is the general curve in $\cal C$. 

Since the family $\cal C_x$ depends on $x$, this implies that $S$ should have a 3--dimensional family $\mathcal D$ (described by all families $\cal C_x$ when $x$ moves on $\Gamma$) of irreducible curves spanning a linear space of dimension 3. Let then $p\in S$ be a general point, and consider the projection in $\p^5$ of $S$ from $p$, whose image, birational to $S$ by the Trisecant Lemma, we denote by $S'$. Then $S'$ has a family of dimension 2 of irreducible plane curves, the images of the curves of $\mathcal D$ passing through $p$. 
By the Trisecant Lemma the curves in question are conics and $S'$, spanning a $\p^5$, is defective since it contains an irreducible conic passing through two general points of it. Then $S'$ is the Veronese surface $V_{2,2}$, but this is a contradiction, because $S'$ should also contain the line corresponding to the exceptional divisor of the blow--up at $p$. This contradiction proves the Claim.\end{proof}

Next we argue by induction on the dimension $n$. If $n=4$,  let $\mathcal P$ a general pencil of hyperplanes in $\p^8$. Fix general $P\in \mathcal P$,  which cuts $Y$ on $X$ lying on a cone over $V_{2,2}$ with vertex a line $L$. If $P'$ is another general element in $\mathcal P$ it cuts $X$ along $Y'$ lying on a cone over $V_{2,2}$ with vertex a line $L'$.
The surface section $S$ of $X$ with $P\cap P':=\Pi$ (which depends only on $\mathcal P$) a priori sits inside $\Pi$ in two cones over $V_{2,2}$, one with vertex $v=L\cap \Pi$ and another with vertex $v'=L'\cap \Pi$. However, when $P'$ moves in $\mathcal P$, by Claim \ref {cl:parton} the point $v'$ cannot move, so that $v=v'$ and therefore $L\cap L'=v$. 
 This implies that the lines which are vertices of the $4$--dimensional cones over $V_{2,2}$ in which the hyperplane sections of $X$ sit, pairwise intersect each other, and this intersection is not a fixed point because the hyperplanes of $\p^8$ have no fixed point. This implies that all these lines sit in a plane, which is the vertex of a cone over $V_{2,2}$ containing $X$, so that we are in case (iii).

Suppose next we are in dimension $n>4$ and let $Z=X\cap H$ be the intersection of $X$ with a general hyperplane $H$. By induction $Z$ is in case (iii), hence it sits in a $n$--dimensional cone over $V_{2,2}$ with vertex a linear space $T$ of dimension $n-3$. If $p_0,p_1$ are general points in $Z$,  then  $\Gamma_{Z,p_0,p_1}$  varies in a 2--dimensional family $\mathfrak F$ of irreducible varieties of dimension $n-2$ spanning linear spaces of dimension $n$ and contained in quadric cones of dimension $n-1$ with vertex $T$  projecting a  conic of $V_{2,2}$. 

Let $V=Z\cap \Lambda$ be  the intersection of $Z$ with a general subspace $\Lambda$ of $H$ of dimension $n+2$, sitting in the cone over $V_{2,2}$ with vertex  $W=\Lambda\cap T$. Then $V$ has a 2--dimensional family $\mathfrak C$ of irreducible varieties of dimension $n-3$ spanning linear spaces of dimension $n-1$ and contained in quadric cones of dimension $n-2$ with vertex $W$  projecting a  conic of $V_{2,2}$: they are the hyperplane sections  of the varieties   in $\mathfrak F$. 

\begin{claim}\label{cl:parton2} Let $V$ be the variety as above. Then there is no irreducible 1--dimensional family $\mathfrak G$ of subspaces of dimension $n-4$ in $\Lambda$, containing $W$ such that the projection of $V$ from a general subspace of $\mathfrak G$ is a cone over the Veronese surface $V_{2,2}$.
\end{claim}

\begin{proof}[Proof of the Claim \ref {cl:parton2}] The proof is similar to the one of Claim \ref 
{cl:parton}, so we will be brief. 

We argue by contradiction and assume there is such a 1--dimensional family $\mathfrak G$. Then for a general space $\Xi\in \mathfrak G$, $V$ sits in the cone over $V_{2,2}$ with vertex $\Xi$, and therefore it contains a family $\mathfrak C_\Xi$ of irreducible varieties of dimension $n-2$ spanning linear spaces of dimension $n-1$ and contained in quadric cones of dimension $n-2$ with vertex $W$  projecting a  conic of $V_{2,2}$.

As in the proof of Claim \ref {cl:parton}, one proves that $\mathfrak C_\Xi$ is not constant: we leave the details to the reader. 
This implies that $V$ contains a 3--dimensional family $\mathfrak D$ of  varieties of dimension $n-3$ spanning linear spaces of dimension $n-1$.  Let  $p\in V$ be a general point and consider the projection in $\p^{n+1}$ of $V$ from $p$, whose image, birational to $V$ by the Trisecant Lemma, we denote by $V'$. Then $V'$ has a family of dimension 2 of varieties of dimension $n-3$ spanning linear spaces of dimension $n-2$, i.e., the images of the varieties in $\mathfrak D$  containing $p$. The general surface section $S'$ of $V'$ spans a $\p^5$ and it has a 2--dimensional family of irreducible plane curves, so it is the Veronese surface $V_{2,2}$. It is well known that $V_{2,2}$ is not extendable (see cite [Corollary 2.4.5]{R2}), so $V'$ is a cone over $V_{2,2}$ with vertex a linear space of dimension $n-5$. This is however impossible because $V'$ should contain the linear space of dimension $n-3$ which is the image of the exceptional divisor corresponding to $p$, whereas the cone over $V_{2,2}$ with vertex a linear space of dimension $n-5$ contains no such subspace.  This contradiction proves the Claim.\end{proof}

At this point we can argue exactly as in the case $n=4$ to see that the 
spaces $T$ of dimension $n-3$ vertices of the cones over $V_{2,2}$ containing the hyperplane section $Z$ of $X$, pairwise intersect each other in a variable space of dimension  $n-4$, hence they lie in a space of dimension $n-2$, which is the vertex of a $(n+1)$--dimensional cone over $V_{2,2}$ containing $X$, so we are in case (iii).

Suppose that $Y$ is a projection in $\p^ 7$ of the Veronese threefold $V_{3,2}$ and not a hyperplane section of the Segre variety ${\rm Seg}(2,2)$. By taking into account Example 2.5 of \cite {ChCi}, we may assume that $Y$ is not the projection of $V_{3,2}$ from a secant line to it. 
We have to prove that we are in case (iv), and it suffices to do it when $X$ has dimension 4. Indeed, if $n>4$ then by arguing by induction we have that the general hyperplane section of $X$ is a cone and therefore $X$ itself is a cone. 

We have $f(Y)=1$, hence $f(X)=2$. Then $\gamma(X)\geq f(X)=2$. On the other hand, by 
Proposition \ref {rem:hh}, we have $\gamma(X)\leq \gamma(Y)+1$, and since $\gamma(Y)=1$ we get $\gamma(X)\leq 2$ so $\gamma(X)=2$. Then for general choice of $p_0,p_1\in X$, one has  that $\Gamma_{p_0,p_1}$ is a $2$--dimensional quadric (see Proposition \ref {difetti}). Thus $X$ is swept out by a family $\mathcal Q$ of 2--dimensional quadrics such that, given two general points $p_0,p_1\in X$, there is a quadric $Q$ of $\mathcal Q$ containing $p_0,p_1$. Note that the general quadric $Q$ is irreducible, since so is the conic passing through two general points of $Y$. Thus $\Gamma_{p_0,p_1}
=\Psi_{p_0,p_1}$, because $\Psi_{p_0,p_1}$ is an irreducible component of $\Gamma_{p_0,p_1}$. 

Let $x$ be a general point of $X$ and consider the general tangential projection $\tau_x$ of $X$ and $X_1$ its image, which is a non--degenerate surface in $\p^ 3$.  Given a general quadric $Q$ in $\mathcal Q$, let us denote by $Q_1$ its image via $\tau_x$. There are  three possibilities:

\begin{itemize}
\item [(a)] $Q_1$ is a linear subspace;
\item [(b)] $Q_1$ is a 2--dimensional quadric;
\item [(c)] $Q_1$ is a conic.
\end{itemize}

Case (a) is not possible, because it would imply that the line joining two general points of $X_1$ lies in $X_1$, so $X_1$ itself is a linear space, a contradiction. Also case (b) is not possible. Indeed, this would imply that $X_1=Q_1$ is a quadric. On the other hand $X_1$ equals $Y_1$, the general tangential projection of $Y$ passing through $x$, which in turn is a projection to $\p^ 3$ of the Veronese surface $V_{2,2}$. The only way this could be  a quadric is that it is the projection of $V_{2,2}$
from a secant line. But then $Y$ would be the projection of $V_{3,2}$ from a secant line to it, contrary to our assumptions. Thus we are left with case (c), in which $Q_1$ is irreducible, otherwise $Q$ itself would be reducible, a contradiction. So $Q$ is a cone with vertex a point $v$ lying in $T_{X,x}$. 
But then $X$ is a cone, since its general tangent space contains $v$. Thus we have proved we are in case (iv).

Suppose finally that $Y$ is a hyperplane section of the Segre variety ${\rm Seg}(2,2)$, and let $H$ be the hyperplane spanned by $Y$ in the $\p^8$ where ${\rm Seg}(2,2)$ sits. Recall that ${\rm Seg}(2,2)$ contains a 4--dimensional family $\mathcal F$ of rank 4 quadric surfaces parameterized by $(\p^ 2)^ *\times (\p^ 2)^ *$: the general element $Q_{L,N}$ of $\mathcal F$ is the image of $L\times N$, with $L, N$ lines in the two factors. 

\begin{claim}\label{cl:pop1} In the above setting, $H$ is not tangent to all quadrics in $\mathcal F$.
\end{claim}

\begin{proof}[Proof of Claim \ref {cl:pop1}] Suppose by contradiction $H$ is tangent to all quadrics in $\mathcal F$. Then for all lines $L,N$ of $\p^2$, $H$ would cut $Q_{L,N}$ in a pair of lines of the type $L\times \{x\}$, $\{y\}\times N$, with $x\in N, y\in L$. 
Fix $L$ and let $N$ vary. Then we see that $H$ contains lines of the form $L\times \{x\}$, with $x$ varying on a line $N_L$ in the second factor. Now we have three possibilities:

\begin{itemize}
\item [(1)] when $L$ moves, $N_L$ moves describing the whole of $(\p^ 2)^ *$;
\item [(2)] when $L$ moves, $N_L$ moves describing a 1--dimensional family in $(\p^ 2)^ *$;
\item [(3)] when $L$ moves, $N_L$ stays fixed.
\end{itemize}

In case (1),  $H$ would contain the whole of ${\rm Seg}(2,2)$, a contradiction. One finds the same contradiction also in case (2). In case (3), $H$ would contain the image of $\p^ 2\times \p^1$ and, by symmetry, also of $\p^ 1\times \p^2$, hence $Y$ would be reducible, again a contradiction. This proves the claim.\end{proof}

We prove now that we are in case (v), and once more we may assume that $X$ has dimension 4 (recall that ${\rm Seg}(2,2)$ is not extendable, see e.g. \cite [Corollary 2.4.5]{R2}).  Then $X$ sits in $\p^8$ and, having $Y$ as hyperplane section, it has degree 6.  Again  $f(X)=\gamma(X)=2$, and $X$ is swept out by a $4$--dimensional family $\mathcal Q$ of 2--dimensional quadrics such that, given two general points $p_0,p_1\in X$, there is a unique quadric $Q$ of $\mathcal Q$ containing $p_0,p_1$. The general quadric $Q$ is irreducible, since so is the conic passing through two general points of $Y$, because of the previous claim.

Let $x$ be a general point of $Y$, hence of $X$, and consider the general tangential projection $\tau_x$ of $X$ and $X_1$ its image. This coincides with the image of a general tangential projection $Y_1$ of $Y$ through $x$, which is the same as the image of a tangential projection of ${\rm Seg}(2,2)$. Hence $X_1$ is a smooth quadric surface.  For a general quadric $Q$ in $\mathcal Q$, let, as above, $Q_1$ be its image via $\tau$.  As before, $Q_1$ cannot be a linear space. If $Q_1$ is a conic, then again $X$ is a cone and we are done. If $Q_1$ is an irreducible quadric, then it coincides with $X_1$, and therefore it is smooth. Moreover, since the general fibres of $\tau_x$ are quadrics in $\mathcal Q$, we conclude that two general quadrics in 
$\mathcal Q$ intersect transversally at a point.

Let $Q$ be the general quadric in $\mathcal Q$ and let $\Pi$ be its span. Let $\pi: X\dasharrow \p^ 4$ be the projection of $X$ from $\Pi$. Arguing as in \cite {CC4}, Claim 8.13, one sees that $\pi$ is a birational map. 

The restriction of $\pi$ to the hyperplane section $Y$ of $X$ is the projection of $Y$ to $\p^4$ from the plane spanned by a conic of $Y$, and it is well known, and easy to see,  that the image of $Y$ under this map is a quadric in $\p^4$. Hence the birational map $\pi^{-1}: \p^4\dasharrow X$ is defined by a linear system $\L$ of quadrics.

Since the general hyperplane through $\Pi$ is mapped to the general hyperplane of $\p^4$, 
the hyperplanes of $\p^4$, plus a fixed hyperplane $P'$, form a sublinear system of $\L$.   This implies that $2P'\in \L$, hence the base locus $B$ of $\L$ sits in $P'$. Moreover $P'$ has to be the image of the exceptional divisor of the blow--up of $X$ along $Q$, and therefore $P'$ is the image of the hyperplane $P$ tangent to $X$ along $Q$. Note that $P$ cuts $X$ along a threefold $Z$ of degree 6 singular along $Q$. 

The general quadric of $\mathcal Q$ maps via $\pi$ to a plane which intersects $P'$ along a line. The planes thus obtained vary in a $4$--dimensional family $\mathcal Q'$ with the property that given two general points $x,y$ of $\p^ 4$ there is a unique plane of
the family $\mathcal Q'$ containing $x,y$. The line $\langle x,y\rangle$ cuts $P'$ in a point $z$. There is then a unique secant line $r$ to the base locus $B$ of $\L$ passing through $z$ (because of the uniqueness of the quadric in $\mathcal Q$ containing two general points of $X$), so that the unique plane of $\mathcal Q'$ through $x,y$ is $\langle x,y,r\rangle$. This implies that the intersection lines of the planes in $\mathcal Q'$ with $P'$ vary in a congruence $\mathcal C$ of secants to $B$ such that given a general point $z\in P'$, there is a unique line in $\mathcal C$ containing $z$.

We claim that $\cal C$ is a congruence of type $(1,1)$.
Let  in fact $S$ be a general plane in $P'$, which is the projection via $\pi$ of a general subspace $T$ of dimension 6 contained in $P$ and containing $\Pi$. The intersection of $T$ with $X$ is surface consisting of $Q$ counted twice, plus another quadric $Q'$, with $Q$ and $Q'$ intersecting along a line.
Thus the projection of $Q'$ from $\Pi$ is a line which clearly belongs to the congruence $\cal C$, and it is the unique line on $S$ with this property. This proves our claim, hence $\cal C$ is given by all lines in $P'$ which meet two skew lines $L_1,L_2$ of $P'$ which form the base locus $B$ of $\L$.

In conclusion $\L$ is the 8--dimensional linear system of quadrics in $\p^4$ passing through two skew lines. It is well known that the image of $\p^4$ via this linear system of quadrics is $\Seg(2,2)$.  
\end{proof}

\begin {remark}\label{rem:ggg} Suppose $X$ is as in one of the cases listed in Theorem \ref {thm:lowcod}. 

In case (i) is it is clear that $X$ is defective with $f(X)=n-2$. 

Suppose we are in case (ii) so that $X$ lies in a cone $V$ over a curve $C$ with vertex a $(n-1)$--space $\Pi$. We claim that $X$ is defective with $f(X)=n-2$, unless $X$ is a cone over a curve, in which case $f(X)=n-1$ (see Proposition \ref {difetti}). In fact, 
let $x\in X$ be a general point. We may and will assume that $x$ also belong to $C$.  
Let $P$ be the ruling of $V$ through $x$. Then 
\begin{equation}\label{eq:ax1}
T_{V,x}=\langle \Pi, T_{C,x}\rangle.
\end{equation} 
Moreover $T_x:=\Pi\cap T_{X,x}$ has dimension $n-2$ and  
\begin{equation}\label{eq:ax2}
T_{X,x}=\langle T_x, T_{C,x}\rangle.
\end{equation} 
If  $T_x$ stays fixed as $x$ varies on $X$, then $X$ is a cone over $C$. If this is not the case, then 
if $p_0,p_1\in X$ are general points, then $\Pi=\langle T_{p_0}, T_{p_1}\rangle$, hence $T_{X,p_0,p_1}$ contains $\Pi$, it  also contains $P_i$ the rulings of $V$ through $p_i$ for $i=0,1$, and finally $T_{V,p_0,p_1}$ because of \eqref {eq:ax1} and \eqref {eq:ax2}. In conclusion $T_{X,p_0,p_1}=T_{V,p_0,p_1}$. Hence 
\[
\begin{split}
\dim(T_{X,p_0,p_1})&=\dim(T_{V,p_0,p_1})=\dim(T_{V,p_0})+\dim(T_{V,p_1})-\dim(\Pi)=\\
&=2(n+1)-(n-1)=n+3,
\end{split}\]
which implies $f(X)=n-2$.

In case (iii), with a similar argument, one proves again that  $X$ is defective with $f(X)=n-2$, unless $X$ is a cone over a Veronese surface $V_{2,2}$, in which case $f(X)=n-1$.  In cases (iv) and (v) it is immediate to check that $X$ is defective with $f(X)=n-2$. 
\end{remark}

\begin{remark}\label {rem:smooth} In the hypoteses of Theorem \ref {thm:lowcod}, if in addition we assume that
$X$ is smooth, then $X={\rm Seg}(2,2)$. 

Indeed, since $s(X)=2n+1-f(X)=n+3$, $X$ can be isomorphically projected to a $\p^{n+3}$ and if $X$ is smooth, Zak's theorem on linear normality (see Chapter I, Corollary 2.11 of \cite  {Zak}) implies that $3n\leq 2(n+2)$ hence $n=4$. Then the generic projection of $X$ to $\p^8$ is a Severi variety (see again Definition 1.2 and Theorem 4.7, Chapter IV, \cite  {Zak}; see also \cite {Scorza}),
whence the assertion. \end{remark}

\section{Scrolls}\label{sec:scrolls}

In this section  we will dispose of the classification of defective 4--folds $X$ non--degenerate in $\p^r$ which are scrolls in $3$--spaces over a curve. The cases $f(X)\geq 2$ (and $r=8$) have been dealt with with Proposition \ref {difetti} and Theorem \ref {thm:lowcod}. So we will assume here that $f(X)=1$ and $r\geq 9$. 

Before stating and proving the classification theorem, let us make a few examples.

  \begin{example}\label{ex:ex1} Let us consider in $\p^r$, with $r\geq 9$, a plane $\Pi$. 
  Consider  an irreducible scroll surface $\Sigma$ sitting in a $(r-3)$--space skew with $\Pi$. Let  $\cal C$ be the curve parameterizing the lines of $\Sigma$. Let $C\subset \Pi^*$ be an irreducible  curve. Suppose there is a rational dominant map $\phi: \cal C\dasharrow C$. Then we can consider the $4$--fold scroll $X$ which is the Zariski closure of the union of the $3$--spaces joining a line of $\Sigma$ corresponding to a point $c$ of $\cal C$ where $\phi$ is defined to the line corresponding to the point $\phi(c)$. Assume that $C\subset \Pi^*$ is not a line, so that the rulings of $X$ do not pass through the same point, hence $X$ is not a cone. Since the rulings of $X$ pairwise meet at a point, then also the tangent spaces to $X$ pairwise meet, hence  $X$ is defective. 
 
 Note that $X$ is singular at all points where two rulings meet.
  \end{example}
  
   \begin{example}\label{ex:ex2} 
   %Before presenting the next example, we recall a well known result by U. Morin (see \cite {Morin}, see also Lemma 4.1 from \cite  {ChCi}; Morin's result is nothing but a simple exercise in projective geometry, so we do not reproduce its proof here). Let us consider an irreducible, positive dimensional family $\cal P$ of planes in $\p^r$, with $r\geq 4$, such that two general planes of the family intersect at a point. Then one of the following cases occurs:
  % \begin{itemize}
   %\item [(a)] all the planes of the family $\cal P$ pass through the same point;
   %\item[(b)] all the planes of the family $\cal P$ intersect the same plane along a line;
   %\item[(c)] all the planes of the family $\cal P$ sit in the same 4--dimensional linear space;
   %\item [(d)] all the planes of the family $\cal P$ sit in the same 5--dimensional linear space.
   %\end{itemize}
   %Note that in case (d), the family $\cal P$ is a special family of planes, e.g., it consists of tangent planes to the Veronese surface $V_{2,2}$. 
   
   Consider an irreducible  1--dimensional family $\cal P$ of planes in $\p^r$, with $r\geq 9$,  such that two general planes of the family intersect at a point, and it does not happen that either all the planes of $\cal P$ pass through the same point or
all the planes of  $\cal P$ intersect the same plane along a line. Then by Lemma 4.1 from \cite  {ChCi} (Morin's theorem), the planes of $\cal P$ span at most a $5$--space. 
   
  Consider  an irreducible curve $C\subset \p^r$ which has a dominant rational map $\phi: C\dasharrow \cal P$, and 
    is in a sufficiently general position with respect with the planes of the family $\cal P$. Then consider the 4--dimensional scroll $X$ which is the Zariski closure of the union of all 
    $3$--spaces joining a point $x\in C$ where $\phi$ is defined, with the plane corresponding to $\phi(c)$. Then $X$ is not a cone and 
 it is defective.  As in Example \ref {ex:ex1}, $X$ is singular.
  \end{example}
  
  \begin{example}\label{ex:ex3} Let $S\subset \p^9$ be an irreducible, projective surface, and let $\Cc$ be an irreducible 1--dimensional family of curves on $S$ not passing through the same point,  such that the general curve in $\Cc$ is irreducible, spans a 3--dimensional space and two general curves of $\Cc$ intersect at a point.  Let $X$ be the 4--fold scroll which is swept out by the 3--dimensional spaces spanned by the curves of $\Cc$. Then $X$ is not a cone and 
 it is defective.  As in Example \ref {ex:ex1}, $X$ is singular.

For a specific example of this type, consider the Veronese surface $V_{2,3}\subset \p^9$. 
It contains the images of the lines of $\p^2$ which are rational normal cubics, each spanning a 3--dimensional linear space. Let $\Cc$ be an irreducible 1--dimensional family of these curves corresponding to a family of lines which do not pass through the same point. 
 
 Other infinitely many examples are obtained in the following way. Consider the Veronese 3--fold $V_{3,2}\subset \p^9$, which is defective, any two of its tangent planes meeting at a point.
We can view $V_{2,3}$ as the image of the \emph{dual Veronese map} which sends a point $H$ of the dual  $(\p^3)^*$ of $\p^3$ (i.e., $H$ is a plane of $\p^3$) to quadric $2H$ in the $\p^9$ parameterizing the quadrics of $\p^3$. Then the tangent space to  $V_{2,3}$ at the point $2H$ of consists of all the points of $\p^9$ corresponding to reducible quadrics of the form $H+H'$ with $H'$ varying in $(\p^3)^*$.
 Consider any irreducible curve $\Gamma\subset(\p^3)^*$ and consider the 4--fold $X$ which is swept out by the tangent spaces to  
 $V_{3,2}$ at the points $2H$ corresponding to points $H$ of $\Gamma$. We claim that $X$ is of the type described above. Indeed, consider the surface $S\subset \p^9$, birational to the symmetric product $\Gamma(2)$ of the curve $\Gamma$, consisting of all points $H+H'$ with $H, H'\in \Gamma$. Then the tangent space to $V_{3,2}$ at the point $2H$, where $H\in \Gamma$, cuts $S$ in the curve $C_H$, isomorphic to $\Gamma$, described by the reducible quadrics $H+H'$, with $H'$ varying in $\Gamma$. When $H$ varies in $ \Gamma$ the curve $C_H$ varies in a 1--dimensional family $\cal C$ of curves on $S$, enjoying the required property.  \end{example}
  
 \begin{theorem}\label{prop:lowrk} Let $X\subset \p^ r$, $r\geq 9$, be
 an irreducible, non--degenerate, linearly normal, defective $4$--fold scroll in 3--spaces over a curve with $f(X)=1$. 
 Then one of the following cases occurs:
 \begin {itemize}
 \item [(i)] $X$  is a cone  with vertex a point
 over a non--defective threefold;
 \item [(ii)] $X$ is as in Example \ref {ex:ex1};
 \item [(iii)] $X$ is as in Example \ref {ex:ex2};
 \item [(iv)] $r=9$ and $X$ is as in Example \ref {ex:ex3}.
  \end{itemize}
  In all these cases $X$ is singular.
  \end{theorem}
 
 \begin{proof}  If $X$ is a cone, the assumption $f(X)=1$ implies that $X$ has only one point as vertex and that its general hyperplane section is non--defective.
 Assume next that $X$ is not a cone. 
 
 \begin{claim}\label{cl:op1} Two general rulings of $X$ meet at a single point.
 \end{claim}
 
 \begin{proof} [Proof of the Claim \ref {cl:op1}] Let $p_0,p_1\in X$ be general points, and let $P_0,P_1$ be the rulings of $X$ through $p_0,p_1$ respectively. Consider the tangential projection $\tau_{p_0}: X\dasharrow X_1$. Since $X$ is defective with $f(X)=1$, then $X_1$ is a non--degenerate 3--fold in $\p^{r-5}$. Since $X$ is a scroll and $X_1$ is non--degenerate, the rulings of $X$ are mapped by $\tau_{p_0}$ to planes, hence  the ruling $P_1$ through $p_1$ intersects $T_{X,p_0}$ at a point $p$ which is the unique intersection point of $T_{X,p_0}$ and $T_{X,p_1}$. Similarly, $P_0$ has to intersect $T_{X,p_1}$ at $p$, hence $P_0$ and $P_1$ have to intersect at $p$.  \end{proof}

 Let $P$ be a general ruling of $X$. The intersection point of $P$ with another general ruling $P'$  cannot stay fixed as $P'$ varies, because we are assuming $X$ is not a cone. Then this intersection describes an irreducible curve $C_P$. Moreover if $P$ and $P'$ are general rulings, then $C_P$ and $C_{P'}$ intersect at one point, i.e., the intersection point of $P$ and $P'$. 
 
If $C_P$ is a line, then these lines pairwise meet at a point. Since we are assuming $X$ is not a cone, then they lie in one and the same plane. Then we are in case (ii). 

If $C_P$ is not a line, but spans a plane, then these planes pairwise meet at a point but cannot pass through the same point because $X$ is not a cone. So, by applying Lemma 4.1 from \cite  {ChCi},  either they intersect the same plane along a line, or they span at most a $5$--space. In the former case  we are again in case (ii) and the curves $C_P$ are lines, contrary to our assumption. So we are in the latter case and in case (iii).

Finally, assume $C_P$ spans $P$. Then take four general rulings $P_i$, for $1\leq i\leq 4$. If $P$ is another ruling and $p_i=P\cap P_i$, for $1\leq i\leq 4$, then the points $p_1,\ldots, p_4$ are independent, so they span $P$. This implies that $X$ is contained in $\langle P_1,P_2,P_3,P_4\rangle$ which has dimension at most 9. On the other hand, we have $r\geq 9$, hence $r=9$, and we are in case (iv). 
\end{proof}

\section{Top species: the irreducible case}\label {sec:top}

\subsection {} Here we will consider the case of an irreducible, non--degenerate,
 defective variety $X\subset \p^ r$ of dimension $n\geq 4$, with $r\geq 2n+1$.  
 Then we have $\delta(X)=f(X)$ (see Remark \ref {rem:def}).  The case $\delta(X)=n-2$
 is covered by Theorem \ref  {thm:lowcod}. So we may assume $1\leq \delta(X)\leq n-3$.
 
  In this section we will assume $f(X)=\delta(X)=\gamma(X)$.  
  Hence,
  for  $p_0,p_1\in X$ general points, one has that  $\Gamma_{p_0,p_1}=\Psi_{p_0,p_1}$ equals the general entry locus and it
 is a quadric of rank at least 2  in $\p^ {\delta(X)+1}$ (see Proposition \ref {difetti}, (ii)). Therefore
 $X$ is a, perhaps singular, LQEL variety (see \S \ref  {susec:difetti}). If $\delta(X)=1$, then $X$
  presents the  LCEL case. This happens when
 $X$ is a variety of the top species. 
 
 By taking a general linear section $Y$ of $X$ of dimension  $m=n-\delta(X)+1$
 we have $Y\subset \p^ s$, with $s=r-\delta(X)+1\geq  2m+1$,
 and $\delta(Y)=1$. Hence the varieties $X$  in question are extensions of varieties  
 of the top species, thus presenting the LCEL case.  Therefore, as a first instance, we 
 will limit ourselves to consider only the top species case. In this section we will also
 stick to the \emph{irreducible case}, i.e. the case 
 in which the general entry locus is an irreducible conic.  

\subsection{} From now on in this section we let  $X\subset \p^ r$ be an irreducible, non-degenerate,
 defective variety of dimension $n\geq 4$, with $r\geq 2n+1$.  We will assume that 
 $X$ is of the top species, so that $\delta(X)=f(X)=\gamma(X)=1$, and presents the irreducible case.   Hence there is an irreducible family $\Cc$ of dimension $2n-2$ of generically irreducible conics such that, given two general points $p_0,p_1\in X$, there is a unique conic of $\Cc$ containing $p_0,p_1$. 
 
\begin{lemma}\label{lem:sort} In the above setting, let $x\in X$ be a general point and let $t$ be a general tangent direction to $X$ at $x$. Then there is a unique irreducible conic of the family $\Cc$ passing through $x$ and tangent to $t$. 
\end{lemma}

\begin{proof} Let $X[2]$ be the Hilbert scheme of $X$ parameterizing the 0--dimensional subschemes of $X$ of length 2. Consider the set 
 \[
 \Omega=\{(\Gamma, \xi)\in \Cc\times X[2]: \xi\subset \Gamma\}
 \]
 The second projection $p_2: \Omega\to X[2]$ is birational. Suppose given $x$ and $t$ general, there is more than one conic through $x$ tangent to $t$. Then, by Zariski's Main Theorem,  given a general non--reduced length 2 subscheme $\xi$ of $X$ (which is a smooth point of $X[2]$), there would be a positive dimensional family of conics in $\Cc$ containing $\xi$. This would imply that the subset $\Omega'$ of $\Omega$ consisting of pairs $(\Gamma, \xi)$ such that $\xi$ is non--reduced has a component of dimension at least $2n$. The first projection $p_1: \Omega'\to \Cc$ has fibres of dimension 1, so $p_1(\Omega')$ would have dimension at least $2n-1$, a contradiction. 
 
 The irreducibility of the conic through $x$ and $t$ is clear. \end{proof}  
 
 If $x\in X$ is a general point, 
define the \emph {Scorza's map}
$$S_x:X\map T_{X,x}$$
\noindent  sending a general point $y\in X$ to
the point $T_{X,x}\cap T_{X,y}$. 
If $C_{x,y}$ is the unique conic
in $X$ passing through $x$ and $y$, then 
$$S_x(y)=
T_{C_{x,y},x}\cap T_{C_{x,y},y}.$$
If there is no danger of confusion, we write
$S$ instead of $S_x$.

\begin{proposition}\label{prop:scorza} In the above setting, Scorza's map $S_x$ is birational, hence $X$ is rational. Moreover its inverse $S_x^ {-1}$  is defined by a linear system of hypersurfaces of degree $d$ having a point of 
multiplicity $d-2$ at $x$. 
\end{proposition}

\begin{proof} We prove that $S_x$ is birational by constructing its inverse $S_x^ {-1}$. Indeed, let $y\in T_{X,x}$ be a general point. Let $r=\langle x,y\rangle$. By Lemma \ref 
{lem:sort}, there is a unique irreducible conic $C\in \Cc$ passing through $x$ and tangent to $r$ at $x$. Let $s$ be the tangent line to $C$ through $y$ other than $r$. It is tangent to $C$ at a point $z$. Then $z=S_x^ {-1}(y)$. This proves that $S_x$ is birational.

As for the rest of the assertion, it follows from the fact that $S_x^ {-1}$ maps the lines through $x$ to conics on $X$.\end{proof}

The main point of our analysis is to study Scorza's map. Specifically, following the ideas of Scorza's in \cite {Scorza1}, we will prove that, under suitable
hypotheses, its inverse is defined by a linear system of quadrics. 

 \subsection{} We will study now the family of conics $\Cc$. In doing so, we will use some ideas contained in \cite {IR}, where however $X$ is assumed to be smooth.
 
Let $x\in X$ be a general  point. We will denote by $p_x: \mathcal{F}_x\to \mathcal{C}_x$ the flat family of conics in $\Cc$ passing through $x$.  
One has $\cal F_x\subseteq \cal C_x\times X$. Let $\phi_x: \mathcal{F}_x \to X$ be the projection to the second factor, which is  birational. Then $\cal F_x$ is irreducible, so also $\cal C_x$ is irreducible with $\dim (\cal C_x)=n-1$ and 
 $\dim(\cal F_x)=n$.

The map $p_x$ has a natural section $\sigma_x$ mapping a point of $\cal C_x$ corresponding to a conic $C$, to the point $x\in C\subset \cal F_x$. The image
 ${\mathcal E}_x$ of $\sigma_x$ is contracted by $\phi_x$ to the point $x$.  Consider the blow--up $\pi_x: \tilde X\to X$ of $X$ at $x$.
So we have the following commutative diagram

\begin{equation*}\label{joindiagram1}\raisebox{.7cm}{\xymatrix{
&{\mathcal{F}_x}  \ar[d]^{p_x} \ar[dr]^{\phi_x}\ar@{-->}[r]^{\psi_x}&{\tilde X}\ar[d]^{\pi_x}\\
&\mathcal{C}_x\ar@/^1pc/[u]^{\sigma_x}&X
}}
\end{equation*}
where $\psi_x$ is  a birational map whose indeterminacy locus is contained in $\cal E_x$. 
Let $\cal S_x$ be the closed subscheme of $\cal C_x$ parameterizing those conics in $\cal C_x$ which are singular at $x$. We will denote by $\cal D_x$ the image of $\cal S_x$ via the section $\sigma_x$.
Then $\cal D_x$ is exactly the  indeterminacy locus of $\psi_x$: this follows by 
the universal property of the blow--up and by Lemma 4.3 from \cite {IN}. 
Thus $\psi_x$ is defined at the general point of $\cal E_x$ and its  inverse is defined at the general point of the exceptional divisor $E_x$ of the blow--up. Hence $\psi_x$ induces a birational map 
$\psi _{x,0} : {\mathcal E}_x \map E_x$ (see also Lemma \ref {lem:sort}).  
The geometric meaning of the map $\psi_{x,0}$ 
 is clear: it associates to a point $(c,x)\in \mathcal E_x\setminus \mathcal D_x$ the point of $E_x$ corresponding to  the direction of the tangent line 
at $x$ to the conic $C$ corresponding to the point $c\in \cal C_x$. 

Let $D_x$ be the closure in $E_x$ of the points corresponding to all directions at $x$
of irreducible components of singular conics in $\cal S_x$. We set
\[
\chi(X)= \dim (D_x),
\]
which is independent on $x\in X$.
Note that $D_x\subseteq L_x$, where $L_x$
was introduced in \S \ref  {subsec:sff}, thus $\chi(X)\leq \ell(X)$.  Given a point in $\cal S_x$ corresponding to a conic $C$ singular at $x$ the plane $\Pi_C$ spanned by $C$ is tangent to $X$ at $x$ and the points corresponding to the directions of the lines in $\Pi_C$ through $x$ describe a line $r_C$ in $E_x$.
It may either happen that $C$ has rank 2 or rank 1. In the former case we get two distinct points $p,q\in D_x$ and the line $r_C$ joins the points $p$ and $q$. In the latter case $C$ consists of a double line $Z$, whose direction corresponds to a unique point $z\in E_x$. The plane $\Pi_C$ still gives rise to a line
$r_C$ in $E_x$  through $z$. The following lemma shows that  something similar to the rank 2 case for $C$ 
happens in this situation (recall the definitions and notation from  \S \ref  {subsec:sff}).

\begin{lemma}\label {lem:tech} In the above setting, i.e., when $C$ consists of a double line $Z$, the line $r_C$ belongs to the tangent space
to the scheme $L_{x}$ at the point $z$. Every point  $z\in L_{x}$ is a base point for ${\rm II}_{x}$ and all quadrics in ${\rm II}_{x}$ are tangent to $r_C$  at $z$. \end{lemma}

\begin{proof} The plane $\Pi_C$ is tangent to $X$ at the general point of the line $Z$. Therefore we have an inclusion of normal sheaves $N_{Z,\Pi_C}\to N_{Z,X}$, hence an inclusion $N_{Z,\Pi_C}(-x)\to N_{Z,X}(-x)$, which induces an inclusion $H^ 0(Z, N_{Z,\Pi_C}(-x))\subseteq  H^ 0(Z, N_{Z,X}(-x))$. 
The first assertion immediately follows and the second is a straightforward consequence. \end{proof}

Let us denote by $J_{X,x}$, or simply by $J_x$, the Zariski closure of the union of all tangent spaces to 
$L_{x}$ in $E_x$ at the points of $D_x$.

\begin{lemma}\label{finitepsi} In the above setting, 
let $D\subseteq \cal E_x\setminus \mathcal D_x$ be an irreducible complete curve contracted to a point
by $\psi_{x,0}$. Then the general point $(c,x)\in D$ corresponds to a rank 2 conic $C=N+M$, where
$N$ is a fixed line through $x$ and $M$ decribes a cone not containing $x$, with vertex a point of $N$. 

If $q\in E_x\setminus D_x$ and if there is a curve $D$ in $\cal E_x\setminus \mathcal D_x$ such that
$\psi_{x,0}(D)=\{q\}$, then $q\in J(D_x,D_x)\cup J_{x}$.
\end{lemma}

\begin{proof}
We shall adapt to our situation  a   remark of Kebekus (see \cite{Kebekus}, proof of Theorem 3.4). Let $\psi_{x,0}(D)= p$. Take $\widetilde D$, the normalization of $D$, and consider the scheme $Y$ over $\widetilde D$ which is gotten by base--change from our family. 

Suppose first the general fiber of $\pi_Y:Y\to \widetilde D$ is irreducible, so that $Y$ is an irreducible surface.
Let $D_0\subset Y$ be the section of $\pi_Y : Y\to \widetilde D$ induced by $\sigma_x$. Note that $Y$ is smooth along $D_0$ by the hypothesis that $D\cap  \mathcal D_x$ is empty. The restriction of the differential of $\phi_Y: Y\to X$ induces a morphism $T_{\phi_Y}: N_{D_0,Y} \to l_p\simeq \C$, where  $l_p$ is the line in $ T _x (X)$ corresponding to $p\in E_x$. Since $N_{D_0,Y}$ is not trivial, the map $T_{\phi_Y}$ is not constant, so it has a zero. The corresponding curve of the family is singular at $x$. This is a contradiction.

Suppose now that every fiber of $\pi_Y:Y \to\widetilde D$ is reducible, consisting of two irreducible components each one isomorphically mapping to different lines contained in $X$. Only one of these two lines, call it $N$,  contains $x$ and has the direction corresponding to $p$.  So $N$ does not move.
Let $M$ be the other line, which varies in a 1--dimensional family. By the assumption
$D\cap \cal D_x=\emptyset$, the intersection point $y$ of $N$ with $M$
has to stay fixed when $M$ varies. 
Then $Y$ has one irreducible component of dimension 1 corresponding to $N$ and a two dimensional component mapping to a cone with vertex $y$ not passing through $x$. 

Let $q\in E_x$ be an indeterminacy point of $\psi_{x,0}^ {-1}$, and let
$D\subseteq \mathcal E_x\setminus \mathcal D_x$ be an irreducible curve such that $\psi_0(D)=q$. Assume that $q\not \in J_{x}$, in particular $q\not\in D_{x}$. Then the general point $(c,x)\in D$ corresponds to an irreducible conic $C$ containing $x$, whose tangent line has direction corresponding to $q$. By the above argument, $D$ cannot be complete, hence there is some point  $(c_0,x)\in \overline D\cap\mathcal D_x$. Let $C_0$ be the  conic corresponding to $c_0$, which is singular at $x$. Then the direction corresponding to $q$ has to belong to the Zariski tangent space to $C_0$ at $x$. 
If  $C_0$ has rank 1, then $q\in J_{x}$ by Lemma \ref  {lem:tech}, a contradiction. Then $C_0$  has rank 2, and, by the above argument, $q\in J(D_x,D_x)$. \end{proof}

Next we will restrict our attention to the case $\chi(X)\leq 0$, i.e., $\dim (D_x)\leq 0$. 

\begin{lemma}\label{locuspsi} Suppose  $\dim (D_x)\leq 0$. Then  $\psi_{x,0}^{-1}$ induces an isomorphism between $E_x\setminus J(D_x,D_x)\cup J_{x}$ and its image,  so that for every point $q\in E_x\setminus J(D_x,D_x)\cup J_{x}$ there is a unique conic in $\cal C_x$
with tangent line at $x$ having direction corresponding to $q$.
\end{lemma}

\begin{proof}  Let 
\begin{equation*}\label{diag}\raisebox{.7cm}{\xymatrix{
&\widetilde{\mathcal{E}_x}  \ar[d]_\alpha \ar[dr]^{\widetilde{\psi}_{x,0}}&\\
&\mathcal E_x\ar@{-->}[r]&E_x,
}}
\end{equation*}
be a resolution of $\psi_{x,0}:\mathcal E_x\map E_x$.

Consider a point $q\in E_x\setminus J(D_x,D_x)\cup J_{x}$, and suppose that
$\widetilde \psi_{x,0}^{-1}$ is not defined at $q$. 
By Zariski Main Theorem, there is an irreducible, complete curve $C\subseteq \widetilde{\psi}_{x,0}^{-1}(q)$. Let $D=\alpha(C)$, which we may assume to be a curve by the assumption that 
$\widetilde \psi_{x,0}^{-1}$ is not defined at $q$. 
Clearly $D$ cannot be contained
in $\mathcal D_x$, which is a finite set, since by assumption $D_X$ is finite. This gives a contradiction to Lemma \ref{finitepsi}.

Thus  $\widetilde \psi_{x,0}^{-1}$ is defined at $q$, hence also $\psi_{x,0}^{-1}$ is defined at $q$.
If $(c,x)=\psi_{x,0}^{-1}(q)$, then the conic $C$ corresponding to $c$ is smooth at $x$, with tangent line having the direction corresponding to $q$, and $C$ is irreducible because $q\not\in D_x$.
\end{proof}

If $\chi(X)\leq 0$ one has

\begin{equation}\label {eq:cac}
\dim (J(D_x,D_x))\leq 1
\end{equation}

We also need an estimate on the dimension of $J_{x}$.

\begin{lemma}\label {lem:estimate}  If $\chi(X)\leq 0$, one has $\dim(J_x)\leq n-3$.
\end{lemma}

\begin{proof} Since $\dim(D_x)\leq 0$, $J_{x}$ is a union of finitely many linear spaces $T_{L_{x},z}$ 
with $z\in D_x$. We have to prove that none of these spaces has dimension $m\geq n-2$. 

Suppose first  there is a point $z\in D_x$ such that $T_{L_{x},z}$ has dimension $n-1$, i.e., it coincides with 
$E_x$. By Lemma \ref {lem:tech}, all quadrics in ${\rm II}_{X,x}$ are cones with vertex at $z$. Then the map defined by ${\rm II}_{X,x}$ has fibres of positive dimension, a contradiction, since this map coincides with $\tilde \tau_{x|E}: E\dasharrow X_1$ which is dominant (see Lemma \ref {lem:secondff}: remember that $f(X)=1$ so that the general fibre of $\tau_x$ is an irreducible conic), hence generically finite.   

Suppose next there is a point $z\in D_x$ such that $T_{L_{x},z}$ is a hyperplane in $E_x$. Then if a quadric in  ${\rm II}_{X,x}$ has a double point at a general point $y\in E_x$, then it is singular along the whole line $\langle y,z\rangle$. Therefore the general Gauss fibre of  the image of ${\rm II}_{X,x}$  has positive dimension. Since the image of $E$ via the map determined by ${\rm II}_{X,x}$ is $X_1$, we find a contradiction because, being $\gamma(X)=f(X)$, the Gauss map of $X_1$ is generically finite (see Remark \ref  {conctactin}).  \end{proof}

\subsection{} We go back now to the study of Scorza's map $S_x$. If $\chi(X)\leq 0$, by 
\eqref {eq:cac} and Lemma \ref {lem:estimate}, a general line $l \subset E_x$ does not intersect
$J(D_x,D_x)\cup J_{x}$. It does not intersect the base locus $B_{x}$ of ${\rm II}_x$ either: indeed ${\rm II}_x$ has no fixed 
hyperplane because $X_1$ is not a linear space. Hence, by Lemma \ref {locuspsi}, 
$L=\psi_{x,0}^{-1}(l)$ is a smooth irreducible rational curve 
such that for every point $(c,x)\in L$ the conic $C$ corresponding to $c$ is smooth.
By base change
we obtain a  surface $S$ with a morphism $\beta:S\to L=\p^1$, such that every fiber $F$ of $\beta$ is irreducible and maps to a  conic in $X$ through $x$
via the tautological morphism $\phi_S: S\to X$.  We set $Y=\phi_S(S)$. 
The surface $S$ is isomorphic to a surface $\FF_e$, with $e\geq 0$. The section $E_0$ with $E^ 2_0=-e$ is contracted to $x$ on $Y$, hence $e\geq 1$. 

The next lemma is the main tool for the study of Scorza's map.

\begin{lemma}\label{Veronese} In the above setting, suppose again $\chi(X)\leq 0$.
Then $Y$ is a  Veronese surface  $V_{2,2}$.\end{lemma}

\begin{proof} First of all we claim that $T_{X,x}\cap Y=\{x\}$. Indeed, suppose $z\in Y$ is different from
$x$. Then there is  a smooth conic $C$ on $Y$ containing $x$ and $z$. The line $\langle x,z \rangle$ intersects transversally $C$ at $x$. 
If $z\in T_{X,x}$, then the plane $\langle C\rangle$ would be contained in $T_{X,x}$ because it containes two distinct tangent lines to $X$ at $x$, namely $\langle x,z \rangle$ and the tangent line to $C$ at $x$. 
Moreover the tangent line to $C$ at $x$ has clearly  intersection multiplicity larger than 2 with $X$ at $x$ (because it is the limit of lines tangent at $x$ and secant to $X$ at a general point of $C$)  and therefore the  point in $E_x$ corresponding to its direction sits in the base locus $B_{x}$ of ${\rm II}_x$, a contradiction. 

Consider now the tangential projection $\tau_x:X\map X_1$. The restriction of $\tau_x$ to $Y$ is thus defined on $Y\setminus\{x\}$ and
$\tau_x(Y)=\tilde \tau_x(l)$ is, by the genericity of $l$, a conic $\Gamma$. This yields $\dim (\langle Y\rangle)\geq 5$, because $Y$ is projected from the span of its tangent cone at $x$, which is at least a plane.

Let $E_1\subset S$ be the general curve in the linear system $\vert E_0+eF\vert$, where $F$ is a fibre of the ruling of $S\cong\FF_e$ and $E_0$ is as usual the negative section.
Then $E_1$ does not intersect $E_0$, therefore 
 $\phi_S(E_1)$ does not pass through $x$ so that it is projected onto $\Gamma$, yielding $\deg(\phi_S(E_1))=2$. Let $H$ be a general divisor in the linear system $\vert \phi_S^*(\cal O_Y(1))\vert$.
 Then $H\equiv 2E_0+bF$ (remember that the curves in $|F|$ are mapped to conics on $Y$, so $H\cdot F=2$). Since $0=H\cdot E_0=-2e+b$, we have $H\equiv 2(E_0+eF)$. Moreover, as we saw, $2=H\cdot E_1=2E_1^ 2=2e$, hence $e=1$. This implies the assertion.\end{proof}

We can now prove our main result about Scorza's map which can be seen as a strong improvement of Lemma \ref {lem:secondff}.

\begin{theorem}\label{nu1f}
Let $X\subset\p^r$, with $r\geq 2n+1$, be an irreducible, non-degenerate, defective variety, of dimension $n\geq 4$. Assume that $X$ is of the top species, presents the irreducible case and $\chi(X)\leq 0$. Then $r\leq {\frac{n(n+3)}{2}}$ and $X$ is a projection to $\p^ r$ of the Veronese variety $V_{n,2}$. \end{theorem}

\begin{proof} By Lemma \ref {Veronese}, the restriction of $S_x^ {-1}$ to a general plane $\Pi$  through $x$ maps $\Pi$ to a Veronese surface $V_{2,2}$. This implies that $S_{x}^ {-1}$ is defined by a linear system of quadrics (see Proposition \ref  {prop:scorza}), whence the claim follows.
\end{proof}

\subsection{} The previous results suffice for the classification of defective 4--folds of the top species
presenting the irreducible case. Recall the definition of $b(X)$ from \S \ref {subsec:sff}.

\begin{proposition}\label{baselocus} Let $X\subset \p^ r$, with $r\geq 9$, be a
non--degenerate, defective variety, of dimension $4$. Assume that $X$ is of the top species and presents the irreducible case.  Then $b(X)\leq 1$ and if the equality holds then, for a general point $x\in X$, one has:

\begin{itemize}
\item [(i)] either $B_x$ is a smooth conic and $r=9$;
\item [(ii)] or $B_x$ is a line and $r\leq 11$;
\item [(iii)] or $B_x$ is the union of a line $L$ and of a point $p\not\in L$, and $r\leq 10$.
\end{itemize}
In cases (ii) and (iii), one has $\epsilon (X)=2$. 
\end{proposition}

\begin{proof} Note that \eqref {eq:terr3} holds. Hence $\dim  ({\rm II}_{x} )\geq 4$
and therefore the positive dimensional irreducible components of $B_x$ are one dimensional and
either are lines or conics (see Thm. 6.1 of \cite {WDV}). 

If $B_x$ contains a conic $C$, then $B_x=C$. Moreover $C$ is smooth, otherwise $X_1$ would be a quadric cone in $\p^ 4$, contrary to the assumption that $X$ is of the top species, i.e., $\gamma(X)=f(X)=1$ (see Remark \ref {conctactin}). 

If $B_x$ contains a line $L$, then it cannot contain any other line, because the linear system of quadrics containing two skew lines has only dimension 3 and, as we saw, $B_x$ does not contain a reducible conic. Suppose there is a point $p\in L$ such that all quadrics in ${\rm II}_x$ are tangent to a given plane $\Pi$ at $p$. Then if $q\in  E_x$ is a general point, the quadrics in ${\rm II}_x$ singular at $q$ must be singular along the whole line $\langle p,q\rangle$. This implies that $X_1$ has general Gauss fibres of positive dimension, against the hypothesis that $X$ is of the top species (see again Remark \ref  {rem:acca}).  If there are two distinct points $p, q\in B_x$ off $L$, again $X_1$ would be a quadric cone in $\p^ 4$, 
which leads again to a contradiction. This proves that only cases (i), (ii), (iii) are possible. In case (i), one has $\dim  ({\rm II}_{x} )= 4$ and $r=9$. In case (ii), one has $\dim  ({\rm II}_{x} )\leq  6$ and $r\leq 11$. In case (ii), one has $\dim  ({\rm II}_{x} )\leq  5$ and $r\leq 10$.

In cases (ii), (iii), $X_1$ turns out to be a scroll in planes, and therefore $d(X_1)>0$ (see \S \ref {susec:difetti}). This proves the final assertion (see Remark \ref {rem:acca}).  \end{proof}

We can now prove the classification theorem for defective 4--folds of the top species
presenting the irreducible case. It suffices to assume the 4--fold to be linearly normal.

\begin{theorem}\label{nu1}  Let $X\subset \p^ r$, with $r\geq 9$, be a projective,
non--degenerate, linearly normal, defective variety, of dimension $4$. Assume that $X$ is of the top species and presents the irreducible case.  
Then $X$ is one of the following:

\begin{itemize}
\item [(i)]  an internal projection of the Veronese 4--fold $V_{4,2}\subset \p^{14}$ from finitely many points 
with the property that $\ell(X)\leq 0$;
\item [(ii)] the projection of  $V_{4,2} \subset\p^{14}$ from the plane spanned by a conic on it;
\item [(iii)] the projection of  $V_{4,2} \subset\p^{14}$ from a $4$--space spanned by a rational normal quartic curve on it;
\item [(iv)] a hyperplane section of  ${\rm Seg}(2,3) \subset\p^{11}$.
\end{itemize}
\end{theorem}
\begin{proof}  If $\ell(X)\leq 0$,  then also $\chi(X)\leq 0$. Then, as we saw in Theorem \ref {nu1f}, $S_x^ {-1}$ is defined by a linear system $\Lambda$ of
quadrics in $\p^ 4$ and the base locus $B$  of $\Lambda$ is finite, thus we are in case (i).  

Assume that  $\ell(X)>0$. Since $b(X)\geq \ell(X)$, by Proposition \ref {baselocus} we have $\ell(X)=1$. % and the base locus $B$ of $\Lambda$ is also of dimension 1. 
Moreover, there is an irreducible component $\cal R$ of the Hilbert scheme of lines in $X$ of dimension $4$, thus by the results of \cite {Rog}, we have the following possibilities:
\begin{itemize}
\item [(1)] either $X$ is swept out by a 1--dimensional family $\cal Q$ of 3--dimensional quadrics of rank 5;
\item [(2)] or $X$ is swept out by a 2--dimensional family $\cal P$ of planes;
\end{itemize}
and the lines of the family $\cal R$ are the lines contained in the quadrics of $\cal Q$ in case (1), the lines contained in the planes of $\cal P$ in case (2).

By Proposition \ref {baselocus}, in case (1) we have only one quadric $Q$ of the family $\cal Q$ passing through the general point $x\in X$ and in case (2) we have only one plane $P$ of the family $\cal P$ passing through $x$. 

Suppose $S_x^ {-1}$ is still defined by a linear system $\Lambda$ of quadrics, which is the case if $\chi(X)\leq 0$ (see Theorem \ref {nu1f}).  If $z\in \p^ 4$ is a general point, mapped by $S_x^ {-1}$ to a general point $x\in X$, then ${\rm II}_x$ is projectively equivalent to the projection
in $\p^ 3$ from $z$ of the linear system of cones in $\Lambda$ with vertex in $z$. By Proposition \ref {baselocus} we conclude that:
\begin{itemize}
\item [(a)] either $B$ is a smooth conic;
\item [(b)] or $B$ is a line;
\item [(c)] or $B$ is the union of a line $L$ and of a point $p\not\in L$.
\end{itemize}
Cases (a), (b) lead to cases (iii) and (ii) respectively.  In case (c), $X$ is an internal projection of the 4--fold appearing in case (ii), so it sits in $\p^ {10}$. On the other hand  $X$ also sits on ${\rm Seg}(2,3)$, since $\Lambda$ is the sum of two linear systems of hyperplanes, the ones containing the line $L$, the others containing the point $p$, with dimension 2 and 3 respectively. So we are in case (iii). 

Assume now that $S_x^ {-1}$ is not defined by a linear system of quadrics, so that 
$\chi(X)=\ell(X)=1$.  Let us denote by $B$ the base locus of ${\rm II}_x$ for $x\in X$ general. Then by Proposition \ref {baselocus} we have again the three cases (a), (b) and (c). 

Case (a) corresponds to case (1) above, i.e. $X\subset \p^ 9$ is ruled by a 1--dimensional family $\cal Q$ of smooth 3--dimensional quadrics. Let $\pi: \tilde X\to X$ be a desingularization which also resolves the base locus of the pencil $\cal Q$. Let $\tilde {\cal Q}$ be the proper transform of $\cal Q$. Let $x\in X$ be a general point, which can be seen as a general point of $\tilde X$. Recall that $X$ possesses a family $\cal C$ of (generically irreducible) conics such that given two general points of $X$ there is a conic of $\cal C$ containing them. 
Consider the strict transform $\tilde {\cal C_x}$ on $\tilde X$ of the family $\cal C_x$ of conics of $\cal C$ through $x$.  If $C$ corresponds to the general point of $\tilde{\cal C_x}$, then clearly $C\cdot Q>0$ with $Q$  the general quadric in $\tilde {\cal Q}$. On the other hand the proper transform $C'$ on $\tilde X$ of  a reducible conic of $\cal C$ with a double point at $x\in X$  is
contained in the unique element $Q_x$ of $\tilde {\cal Q}$ passing through $x$ (because $Q_x$ contains all the lines in $X$ passing through $x$). Hence $C'\cdot Q=0$.
Therefore these reducible conics are not contained in $\cal C_x$, contrary to the assumption $\chi(X)=1$. Case (a) is thus excluded.

Cases (b) and (c) correspond to case (2) above, i.e., $X$ is ruled by a 2--dimensional family $\cal P$ of planes. In case (b) one comes to a contradiction with an argument similar to the one we made in case (a), which can be left to the reader. In case (c) we have $X\subset \p^ {10}$ and the image of the general tangential projection $\tau_x$ is $X_1={\rm Seg}(1,2)\subset \p^ 5$. Thus we have two rational maps $f_1: X\map \p^ 1$ and $f_2: X\map \p^ 2$. Set $\cal L_1= f_1^* \cal O_{\p^ 1}(1)$ [resp. $\cal L_2=
f_2^* \cal O_{\p^ 2}(1)$]. The minimal sum $\cal L_1\oplus \cal L_2$ coincides with $\cal H(-2x)$ and both linear systems $\cal L_1$ and $\cal L_2$ have a base point at $x$. By moving $x$ we see that 
 $\cal L_1$ and $\cal L_2$ vary in continuous families, and therefore in larger linear systems, since 
 $X$ is rational  (see Proposition \ref {prop:scorza}). If $\cal N_1$ and $\cal N_2$ are the complete linear systems containing $\cal L_1$ and $\cal L_2$, we  have $\dim (\cal N_1)=\dim (\cal L_1)+1=2$ and 
 $\dim (\cal N_1)=\dim (\cal L_1)+1=3$, because  $\cal L_1$ and $\cal L_2$ have the base point $x$ and, by their very definition, are complete under this constraint. Finally, since the minimal sum $\cal L_1\oplus \cal L_2$ coincides with $\cal H(-2x)$, we have $\cal N_1\oplus \cal N_2=\cal H$. By using 
 $\cal N_1$ and $\cal N_2$ we have a map $f: X\map {\rm Seg}(2,3)$, which realizes $X$ as a hyperplane section of  ${\rm Seg}(2,3)$, i.e., we are in case (iv). \end{proof} 
 
 \begin{remark}\label {rem:topsmooth} Let $X$ be as in Theorem \ref {nu1} and assume in addition that $X$ is smooth. Then, either we are in cases (ii), (iii) or (iv),  or  we are in case (i) and either $X$ is the Veronese variety $V_{4,2}$ in $\p^ {14}$, or  $X$ is the projection to $\p^ {13}$ of  $V_{4,2}$  from a point on it. Indeed, any internal projection of $V_{4,2}$ from more than one point is singular, because this projection contracts the unique conic through two (distinct or infinitely near) points
to a singular point. 
 \end{remark}

 \section{Top species: the reducible case}\label{sec:top2}

 \subsection {} Next we continue the analysis of varieties $X\subset \p^ r$ of dimension $n\geq 4$, with $r\geq 2n+1$, of the top species, but we  turn to the reducible case in which the general entry locus is a reducible conic. Remember that still we have $\delta(X)=f(X)=\gamma(X)=1$. 
 
 Let $p_0,p_1$ be general points
 of $X$. Then $\Gamma_{p_0,p_1}$ consists of two incident lines $R_0,R_1$, with
  $p_i\in R_i$ and $p_i\not\in R_{1-i}$, $i=0,1$. In this case $\Psi_{p_0,p_1}=R_0$ (see 
  Proposition \ref {difetti}). This implies, by monodromy, that $R_0,R_1$ move in one and the same irreducible family
  $\R$ of lines inside $X$ filling $X$.
   
 \begin{lemma}\label {lem:move} In the above setting, either $X$ is a cone with vertex a point
 over a non--defective variety of dimension $n-1$ or  $\dim(\R)\geq n$. 
  \end{lemma}
  
  \begin{proof} Since there is some line in $\R$ through the general point of $X$,
  one has $\dim(\R)\geq n-1$. Suppose the equality holds. Then that there is a finite number $\ell$ of lines of $\R$ passing through the general point  of $X$, and, by the irreducibility of $\R$, they are exchanged by monodromy.
  
  Let $p_0,p_1$ be general
  points of $X$ and let, as above,  $\Gamma_{p_0,p_1}=R_0\cup R_1$ be the unique tangential contact locus containing $p_0,p_1$, with
  $p_i\in R_i$ and $p_i\not\in R_{1-i}$, $i=0,1$.  Since, given $p_1$, the line $R_0\in \R$ through $p_0$ is uniquely determined, by the  Remark \ref {rem:topsmooth} this is the unique line in $\R$ containing $p_0$ and $\ell=1$. 
  
  If we let $p_0$ stay fixed and move 
  $p_1$, the line $R_0$ stays also fixed, whereas $R_1$, the unique line in $\R$ through $p_1$, always intersects $R_0$.
  This proves that the lines in $\R$ pairwise meet at a point, and therefore they
  all pass through a fixed point, thus $X$ is a cone.   \end{proof}
  
  The previous lemma suggests that our analysis would take advantage from 
  some classification result for  varieties of dimension $n$ with a family of lines of 
  dimension at least $n$. Unfortunately such a classification is missing
  in general, but it is luckily available for $n=4$ (see \cite {Rog}). Therefore
  we turn our attention to the 4--dimensional case. 
  
   \begin{theorem}\label{prop:lowrk2} Let $X\subset \p^ r$, $r\geq 9$, be
 a non--degenerate, linearly normal, defective variety of dimension 4
 of the top species, not a cone, presenting the reducible case. 
 Then $X$ is a scroll over a curve and therefore one of the cases listed in Theorem \ref {prop:lowrk} occurs. 
   \end{theorem}
  
   \begin{proof}  By Lemma \ref  {lem:move},
 $X$ is covered by lines moving in a family $\R$ of dimension at least 4. Let $\cal S$ be
 an irreducible component of the Hilbert scheme of lines contained in $X$ and
 containing $\R$. By a classical result of B. Segre,  one has $\dim(\cal S)\leq 5$ and if
 the equality holds, then $X$ is swept out by a 1--dimensional family $\cal F$ of 3--spaces
 (see \cite {Rog}, Theorem 1 or \cite {Segre}).  Thus, if $\dim(\cal S)= 5$, then $X$ is a scroll and we are done. 
 
Suppose now that $\cal S=\cal R$ has dimension 4. By Theorem 2 of \cite {Rog}, we have 
 the alternatives (1) and (2) as in the proof of Theorem \ref  {nu1}, from which we keep the notation. 
 
Suppose we are in case (1).  Let $x,y\in X$ be general points and let $Q_x,Q_y$ be the 3--dimensional quadrics in $\cal Q$ containing $x$ and $y$ respectively. Since there are incident lines $r_x\subset Q_x$ and 
$r_y\subset Q_y$, respectively containing $x$ and $y$,  then $Z_{x,y}=Q_x\cap Q_y$ is not empty. We claim that $Z_{x,y}$ is a quadric surface. Indeed, if $\dim(Z_{x,y})\leq 1$, given two more general points $x'\in Q_x$ and $y'\in Q_y$, there are no lines 
$r_{x'}\subset Q_x$ and 
$r_{y'}\subset Q_y$, respectively containing $x'$ and $y'$, which meet along  $Z_{x,y}$. Hence  $Z_{x,y}$ is a surface and therefore it is a quadric (since it cannot be a plane) because $Q_x,Q_y$ have rank 5. 

Consider now the 1--dimensional family $\cal K$ of $4$--spaces spanned by the quadrics of $\cal Q$. Then two general $4$--spaces in $\cal K$ intersect in a $3$--space. Hence  the $4$--spaces in $\cal K$ all contain the same 3--space $\Pi$. Thus $X$  is contained in a 5--dimensional cone over a curve $C$ with vertex $\Pi$ and $S(X)$ would be the cone over $S(C)$ from $\Pi$, yielding $s(X)=7$, which contradicts $\delta(X)=1$.

Suppose we are in case (2). Then the planes in $\cal P$ pairwise meet. 
By Lemma 4.1 from \cite  {ChCi} and taking into account that $X$ is of the top species and is not a cone, 
we see that $X$ should be
contained in a 5--dimensional cone with vertex a plane $\Pi$  over a non--defective surface $Y$ and 
$S(X)$ would be the cone with vertex $\Pi$ over $S(Y)$. Let $\pi$ be the projection of $X$ to $Y$
from $\Pi$. Given two general points
$p_0, p_1\in X$ mapping via $\pi$ to general points  $q_0,q_1\in Y$, then $\Gamma_{X,p_0,p_1}$ is the pull--back via $\pi$ of $\Gamma_{Y,q_0,q_1}=\{q_0,q_1\}$, which contradicts $\gamma(X)=1$.
\end{proof}

\section {Second species}\label {sec:WD}

\subsection {} From now on we will concentrate on the 4--fold case. Let $X\subset \p^ r$, $r\geq 9$, be a non--degenerate, defective, projective  variety of dimension 4. By Proposition \ref {difetti} and Theorem \ref  {thm:lowcod}  we may assume $f(X)=\delta(X)=1$. By Theorem \ref {prop:lowrk} we may also assume $X$ is neither a cone nor a scroll. 
%By projecting generically, we will also assume that $r=9$.

We will consider here the case in which $X$ is of the second species, i.e. $\gamma(X)=2$.
Therefore, given $p_0,p_1\in X$ general points, the tangential contact locus $\Gamma_{p_0,p_1}$ is a surface, spanning a $4$--space $\Pi_{p_0,p_1}$ (see Proposition \ref  {lem:a}). 
Again we have two cases: the \emph {irreducible} and the \emph{reducible} case, according to the possibilities that $\Gamma_{p_0,p_1}$ is irreducible or it consists of two irreducible components each passing through one of the two point $p_0,p_1$. 

\subsection{} We first examine the irreducible case.  In this case the tangential contact surfaces move in a 4--dimensional irreducible family $\cal S$, such that given two general points $p_0,p_1\in X$ there is a unique irreducible surface $\Gamma=\Gamma_{p_0,p_1}$ containing $p_0$ and $p_1$. A priori the surfaces in $\cal S$ may have a base locus scheme $B$
which is the largest subscheme $B$ of $X$ such that $B$ is contained in all surfaces in $\cal S$.
 
There are two subcases to be considered:
\begin{itemize}
\item [(i)] two general surfaces in $\cal S$ have no isolated point in common off the base locus scheme $B$;
\item [(ii)] two general surfaces in $\cal S$ have some isolated point in common off the base locus scheme $B$.
\end{itemize}
Let $x\in X$ be a general point. Then we have a subfamily $\cal S_x$ of surfaces in $\cal S$ passing through $x$. Each irreducible component of $\cal S_x$ has dimension 2. Actually $\cal S_x$ is irreducible, because otherwise if $y\in X$ is another general point, there would be more than one surface in $\cal S$ containing $x$ and $y$. 

If we are in case (i), two general surfaces in $\cal S_x$ have a positive dimensional intersection off $B$, hence they intersect off $B$ in a curve $C$ passing through $x$. Since the surfaces in $\cal S_x$ span in general a $4$--space which varies with the surface, there are three subcases to be considered:
\begin{itemize}
\item [(i1)] the curve $C$ spans a 3--space;
\item [(i2)] the curve $C$ spans a plane;
\item [(i3)] the curve $C$ is a line.
\end{itemize}

\subsection{} In this section we discuss the subcase (i1). We prove that:

\begin{lemma}\label{lem:acc1} The subcase (i1) does not occur.
\end {lemma}

\begin{proof}
In subcase (i1) the spans of two general surfaces in $\cal S_x$ intersect along a 3--space containing the point $x$. Since they fill up $X$ they cannot lie in a 5--space, so they pass through the same $3$--space $P_x$ containing all curves which are the intersections of two general surfaces in $\cal S_x$. 

We claim that all surfaces in $\cal S_x$ contain the same curve $C_x$, passing through $x$, which spans $P_x$. In fact, let $\Gamma$ be the general surface in $\cal S_x$. If the other surfaces in $\cal S_x$ cut out on $\Gamma$ different curves $C$  passing through $x$, then these curves would, on one side, sweep out $\Gamma$, on the other all lie in $P_x$, hence $\Gamma$ would lie in $P_x$ which is a $3$--space, a contradiction. This proves our claim. 

When $x$ moves, the curve $C_x$ moves in a family $\cal C$ of dimension 3, which fills up $X$. Let $y\in X$ be another general point. Then the surface $\Gamma$ of $\cal S$ which contains $x$ and $y$, contains also $C_x$ and $C_y$ (actually, for a general point $z\in \Gamma$, it contains $C_z$), and this implies that $P_x$ and $P_y$ intersect  in a plane. This yields that there is a plane $\Pi$ such that for $x\in X$ general, the 3--space $P_x$ contains $\Pi$. Let $\pi: X\dasharrow \p^{r-3}$ the projection from $\Pi$ and let $Y$ be the image of $X$, which has to span the $\p^{r-3}$. Moreover $\dim(Y)\leq 3$ because for $x\in X$ general the curve $C_x$ is contracted to a point by $\pi$.  
Let $q_0,q_1\in Y$ be general points and let $p_0,p_1$ be points in the counterimages of $q_0,q_1$ respectively. The surface $\Gamma_{p_0,p_1}$, which spans a 4--space contaning $\Pi$, is mapped by $\pi$ to the line $\langle q_0,q_1\rangle\subset Y$. This implies that $Y$ is a linear space, a contradiction. 
\end{proof}

\subsection{} In this section we discuss the subcase (i2). We keep the notation introduced above. We prove that:

\begin{theorem}\label{prop:acc2} In  subcase (i2) then  either:
\begin{itemize}
\item [(a)]  $r=9$ and either $X$ sits in a cone with vertex a line over a hyperplane section of ${\rm Seg}(2,2)$ in $\p^7$;
%or $X$ sits  in a cone with vertex a point over the Segre variety ${\rm Seg}(2,2)\subset \p^ 8$;
\item [(b)] or $9\leq r\leq 11$ and $X$ sits in a cone with vertex a line over the projection in $\p^{r-2}$ of the Veronese $3$--fold $V_{3,2}$.
%\item [(c)] or $r=9$ and $X$ sits in a $6$--dimensional cone over the Veronese surface $V_{2,2}$ in $\p^ 5$.
\end{itemize}
\end {theorem}

\begin{proof} Let $x\in X$ be a general point and let $\Gamma, \Gamma'$ be general elements in $\cal S_x$. Then they intersect in a curve $C$ which spans a plane $P$ and passes through $x$. 

\begin{claim}\label{cl:par} Keeping $\Gamma$ fixed and letting $\Gamma'$ vary, the curve $C$ stays fixed.
\end{claim}

\begin{proof}[Proof of the Claim \ref {cl:par}] Suppose the assertion is false. Consider the family $\cal F_x$ of plane curves cut out on $\Gamma$ by the other surfaces of $\cal S_x$, and recall that $\cal S_x$ has dimension 2. 

If $\cal F_x$ has dimension 2, then, since $x$ is a general point of $\Gamma$, the surface $\Gamma$ possesses a 3--dimensional family $\cal F$ of plane curves. This implies that $\Gamma$ is either a plane or spans a 3--space, a contradiction since $\Gamma$ spans a $4$--space. In fact, since $\cal F$ has dimension 3, two general curves of $\cal F$ intersect in more than one point. If these intersection points span a plane, then $\Gamma$ is a plane. If they span a line, the planes spanning two general curves in  $\cal F$ intersect along a line, so the planes spanning curves in $\cal F$ pairwise intersect along a moving line, hence they span a $3$--space. 

Hence $\cal F_x$ has dimension 1, i.e., a general curve in $\cal F_x$ belongs to a 1--dimensional family of surfaces in $\cal S_x$. Then $\Gamma$ has a 2--dimensional family $\cal F_\Gamma$ of plane curves and therefore there is some curve of $\cal F_\Gamma$ which passes through two general points of $\Gamma$. By moving $\Gamma$ we see that $X$ possesses a family $\cal F$ of plane curves such that there is some curve of $\cal F$ passing through two general points of $X$: indeed, if $\Gamma$ is general in $\cal S$, then two general points of $\Gamma$ are also general on $X$. This yields that $\cal F$ has dimension at least 6. 

Consider now the set $I$ of all pairs $(C, \Gamma)$, with $C$ in $\cal F$ and $\Gamma$ in $\cal S$ such that $C\subset \Gamma$. There are two  projection maps $\pi_1: I\to \cal F$ and $\pi_2: I\to \cal S$. Looking at $\pi_2$, since $\cal S$ has dimension 4 and since for $\Gamma$ general in $\cal S$ the family $\cal F_\Gamma$ has dimension 2, we see that $\dim(I)=6$. On the other hand, by what we saw above, the general fibre of the map $\pi_1$ has dimension at least 1, which implies that $\cal F$ has dimension at most 5, a contradiction. \end{proof}

By Claim \ref {cl:par}, all surfaces in $\cal S_x$ contain a fixed plane curve $C$ passing through $x$. Let us denote again by $\cal F$ the family of these curves $C$, which has the property that, given a general point $x\in X$ there is a unique curve of $\cal F$ containing $x$. Hence $\cal F$ has dimension 3. Moreover the general surface $\Gamma$ in $\cal S$ contains a 1--dimensional family $\cal F_\Gamma$ of curves in $\cal F$. We denote by $\mathcal P$ the family of planes spanned by the curves in $\cal F$.

Let  $p_0, p_1\in X$ be general points. The surface $\Gamma_{p_0,p_1}$ which contains $p_0$ and $p_1$ contains also the curves $C_0,C_1$ of $\cal F$ which pass through $p_0,p_1$ respectively. Hence the planes $P_0,P_1$ in $\mathcal P$ spanned by $C_0,C_1$, lying in the 4--space spanned by $\Gamma_{p_0,p_1}$, intersect in at least one point. 

\begin{claim}\label {cl:par2}  Two general planes in $\mathcal P$ do not intersect in a single point.
\end{claim}

\begin{proof}[Proof of the Claim \ref {cl:par2}] By Lemma 4.1 from \cite  {ChCi}, the only possibilities are:
\begin{itemize}
\item [(1)] there is a point $p$ such that all the planes in $\mathcal P$ contain $p$;
\item [(2)] there is a plane $\Pi$ such that all the planes in $\mathcal P$ intersect $\Pi$ along a line.
\end{itemize}

Assume (1) happens. Let $\pi: X\dasharrow \p^{r-1}$ be the projection of $X$ from $p$. The image $Y$ of $X$ is a 4--fold, otherwise $X$ would be a cone with vertex $p$, a contradiction. Hence $\pi: X\dasharrow Y$ is a generically finite rational map. 
If $\Gamma$ is a general surface in $\cal S$, then the $4$--dimensional span of $\Gamma$ contains $p$, so $\Gamma$ is projected from  $p$ to a surface in $\p^3$.  Hence $Y$ contains a 4--dimensional family $\cal Q$ of irreducible surfaces in $3$--space such that there is a surface of $\cal Q$ through two general points of $Y$. 
This implies that the surfaces in $\cal Q$ are quadrics (by the Trisecant Lemma) and the general point of $S(Y)$ belongs to an at least 2--dimensional family of secant lines to $Y$, i.e., $f(Y)=2$. So we can apply Theorem \ref {thm:lowcod} to $Y$. Cases (i) and (ii) of Theorem, \ref {thm:lowcod} do not apply to $Y$, because in those cases the contact surfaces are not irreducible. In case (iii) of Theorem \ref {thm:lowcod} the tangential contact loci are not surfaces, so also this case does not apply. In case (iv) of Theorem \ref {thm:lowcod}, $Y$ is a cone with vertex a point $v$ and the images of the planes of $\mathcal P$ via $\pi$ are the lines generating the cones. This implies that the planes of $\mathcal P$ all intersect along the fixed line $\langle p,v\rangle$, contrary to the assumption. Suppose we are in case (v) of Theorem \ref {thm:lowcod}. If $Y$ is a cone over a hyperplane section of $\Seg(2, 2)$, the same argument as above proves that the planes of $\mathcal P$ all intersect along the fixed line, a contradiction. It remains to discuss the case in which $Y=\Seg(2,2)$. Then $X$ sits in a cone with vertex $p$ over $\Seg(2,2)$. Now the quadrics in $\cal Q$ are just all the quadrics in $\Seg(2,2)$, because the form a family of dimension 4. Two general of them pairwise intersect at a single point. On the other hand, if $y\in Y$ is a general point, two general quadrics in $\cal Q$ containing $y$ intersect along the image via $\pi$ of a curve in $\cal F$, hence they must intersect along a line, a contradiction again. 

We claim  that also case  (2) cannot happen. Indeed, let $\pi: X\dasharrow \p^{r-3}$ be the projection of $X$ from the plane $\Pi$. Then  all curves of $\cal F$ are contracted to points by $\pi$, so the image of $\pi$ is a 3--fold $Y$. Let $q_0,q_1\in Y$ be general points and let $p_0,p_1\in X$ such that $\pi(p_i)=q_i$, for $i=1,2$. We may assume that  $p_0,p_1\in X$ are also general points. Consider the surface
$\Gamma_{p_0,p_1}$. It contains  the curves $C_0,C_1$ of $\cal F$ which pass through $p_0,p_1$ respectively. The planes $P_0,P_1$ spanned by $C_0,C_1$ respectively, cut $\Pi$ in two distinct lines. This implies that the span of $\Gamma_{p_0,p_1}$ contains $\Pi$, hence the image of $\Gamma_{p_0,p_1}$ via $\pi$ is a line, precisely the line joining $q_0$ and $q_1$. This yields that $Y$ is a linear space, a contradiction since it has to span $\p^{r-3}$ and $r-3\geq 6$. \end{proof}

Finally we have to examine the case in which the planes spanned by the curves of $\cal F$ meet pairwise along a line, and then they  all pass through the same line $R$. Let $\pi: X\dasharrow \p^{r-2}$ be the projection of $X$ from $R$. All curves in $\cal F$ are contracted to points by $\pi$, so the image of $\pi$ is a 3--fold $Y$. 

\begin{claim}\label {cl:par3} $Y$ is a defective 3--fold with $f(Y)=1$, possessing a 4--dimensional family $\cal C$ of generically irreducible conics such that if $q_0,q_1\in Y$ are general points there is a unique conic in $\cal C$ containing $q_0$ and $q_1$.
\end{claim}

\begin{proof} Let $q_0,q_1\in Y$ be general points and let $p_0,p_1\in X$ be such that $\pi(p_i)=q_i$, for $i=1,2$. Then $p_0,p_1\in X$ are also general points. Consider the surface
$\Gamma_{p_0,p_1}$ spanning a 4--space $\Pi_{p_0,p_1}$, which is swept out by a family $\cal F_\Gamma$ of plane curves whose planes all contain the line $r$. Then the projection of $\Gamma_{p_0,p_1}$ from $r$ is a curve spanning  the plane image of $\Pi_{p_0,p_1}$. Hence $Y$ has a 4--dimensional family $\cal C$ of generically irreducible plane curves such that if $q_0,q_1\in Y$ are general points there is a unique curve in $\cal C$ containing $q_0$ and $q_1$. By the Trisecant Lemma, this implies that the curves in $\cal C$ are conics.\end{proof}

In conclusion, in this case $X$ sits in the cone with vertex $R$ over $Y$, which is in the list of defective threefolds in \cite {ChCi}. Actually in that list the only 3--folds $Y$ possessing a 4--dimensional family $\cal C$ of generically irreducible conics such that if $q_0,q_1\in Y$ are general points there is a unique conic in $\cal C$ containing $q_0$ and $q_1$, are the hyperplane section of ${\rm Seg}(2,2)$ in $\p^7$ and the Veronese $3$--fold $V_{3,2}$ and its projections. This leads therefore to cases (a) and (b) of Theorem \ref {prop:acc2}.\end{proof}

\begin{remark}\label{rm:acc3} If $X$ is a 4--fold as in Theorem \ref {prop:acc2}, then $X$ is in fact defective, with $f(X)=1$ and irreducible tangential contact loci  of dimension 2. 

First, suppose $X$ sits in a cone with vertex a line $r$ over a hyperplane section $Y$ of ${\rm Seg}(2,2)$ in $\p^7$. Hence we have a projection $\pi: X\dasharrow Y$.
If $p_0,p_1\in X$ are general points  and $q_i=\pi(p_i)$ with $i=0,1$, then $T_{X,p_0,p_1}$ sits in the span of $T_{Y,q_0,q_1}$ and $r$, which is an $8$--space and therefore coincides with $T_{X,p_0,p_1}$. This proves that $f(X)=1$. The general tangential contact locus is the pull back of a general conic in $Y$ via $\pi$. 

The other case of Theorem \ref {prop:acc2} can be treated similarly and can be left to the reader.\end{remark}

\begin{remark}\label{rm:acc31} Suppose $X$ is a 4--fold as in Theorem \ref {prop:acc2}. Then we claim $X$ is never smooth. 

In fact, suppose for instance that $X$ sits in a cone $V$ with vertex a line $L$ over a hyperplane section $Y$ of $\Seg(2,2)$ and let us prove that $X$ cannot be smooth. We argue by contradiction and assume $X$ smooth.

Suppose first that $X$ contains $L$. Let $p\in L$ be a general point. By projecting from $p$ we have a map $\pi: X\dasharrow \p^8$, with image a variety $Z$, which has to be a 4--fold otherwise $X$ would be a cone with vertex $p$, contrary to the assumption that $X$ is smooth. On the other hand $Z$ is contained in  the projection of the cone $V$ from $p$, which is a cone $W$ with vertex a point $q$ over $Y$.  Thus $Z=W$. Moreover $Z=W$ has to contain the image of the exceptional divisor of the blow--up of $X$ at $p$, which is a $3$--space containing $q$, and this implies that $Y$ contains at least a plane $P$. Actually $T_{X,p}=\langle P,L\rangle$. By moving $p$ along $L$, $T_{X,p}$ has to vary with $p$ by Zak's theorem on tangencies (see \cite {Zak}), and this implies that $Y$ contains infinitely many planes. It is not difficult to see that this is impossible, hence $X$ cannot contain $L$. Indeed, suppose that $Y$ contains infinitely many planes, so we may assume it is swept out by an irreducible 1--dimensional family of planes. Now remark that a plane on $\Seg(2,2)=\p^2\times \p^2$, it either of the form $\p^2\times \{p\}$ or of the form $\{q\}\times \p^2$, with $p,q$ points in $\p^2$. This implies then that there exists a curve $\Gamma\subset \p^2$ such that either $Y=\p^2\times \Gamma$ or $Y=\Gamma\times \p^2$. 
This is not possible because if $\Gamma$ is not a line, then the span of $\p^2\times \Gamma$ or of $\Gamma\times \p^2$ is a $\p^8$ (whereas $Y$ spans a $\p^7$), and if $\Gamma$ is a line, then $\p^2\times \Gamma$ and $\Gamma\times \p^2$ would be equal to $\Seg(1,2)$ which is not equal to the hyperplane section of $\Seg(2,2)$ (they have different degrees). 

Next we note that if $Y$ is singular, then $X$ is singular too. Indeed if $y\in Y$ is any point, then the plane $\Pi_y=\langle y,L\rangle$ of the cone $V$ cuts $X$ along a curve $C_y$ which, as we saw, does not contain $L$. It is then clear that if $y$ is singular for $Y$, then $X$ is singular along the whole curve $C_y$. Hence we may assume $Y$ to be smooth.

Finally, $X$ intersects $L$ at some point $p$. Again by projecting $X$ from $p$ we see, as above, that $Y$ has to contain a plane. Then we claim that $Y$ is singular and then we come to a contradiction. In fact, if $Y$ is smooth, it is well known that $Y$ is isomorphic to the incidence correspondence
\[
\{(p,r)\in \p^2 \times (\p^2)^*: p\in r\}\subset \p^2 \times (\p^2)^*=\Seg(2,2)
\]
and this clearly contains no plane. 

In conclusion this proves that $X$ cannot be smooth. 

The other case of Theorem \ref {prop:acc2} can be analysed with similar arguments.
\end{remark}

\subsection{} In this section we discuss the subcase (i3). We keep notation and conventions introduced above, in particular we assume $X$ is not a cone.

\begin{lemma}\label{lem:acc4} In the subcase (i3), if $x\in X$ is a general point, then all surfaces in $\cal S_x$ contain the same line $r_x$ passing through $x$. Hence $X$ contains a 3--dimensional family $\cal R$ of lines such that there is a unique line in $\cal R$ passing through the general point of $X$. Moreover the surfaces $\Gamma$  in $\cal S$ are ruled, swept out by a 1--dimensional family of lines in $\cal R$. 
\end{lemma}

\begin{proof} Fix $\Gamma$ a general surface in $\cal S_x$. If the other surfaces in $\cal S_x$ intersect $\gamma$ in a moving line, then $\Gamma$ would possess a positive dimensional family of lines through its general point. Then $\Gamma$ would be a plane, which is impossible. Then the assertion follows. 
\end{proof}

We will prove the following:

\begin{theorem}\label{prop:acc5} In the subcase (i3) then  either:
\begin{itemize}
\item [(1)]  $r=9$ and $X$ sits in a cone with vertex a line over a hyperplane section of ${\rm Seg}(2,2)$ in $\p^7$;
\item [(2)] or $9\leq r\leq 11$ and $X$ sits in a cone with vertex a line over  the Veronese $3$--fold $V_{3,2}$ in $\p^9$ or a projection of it in $\p^8$ or $\p^7$;
\item [(3)] or $r=9$ and $X$ sits in a cone with vertex a line over a (defective) 3--fold in $\p^7$ sitting in a cone with vertex a line over the Veronese surface $V_{2,2}$, in particular $X$ sits in a $6$--dimensional cone with vertex a 3--space over  the Veronese surface $V_{2,2}$;
\item [(4)] or $X$ is swept out by a $3$--dimensional family $\cal R$ of lines and
it is singular along a linear space $\Pi$ of dimension $\varepsilon$, with
$2\leq \varepsilon \leq 3$, which is cut out in one point by the general line in $\cal R$, and $X$ projects from $\Pi$ to a 3--dimensional variety $Y\subset \p^{r-\varepsilon-1}$ (with general fibres unions of lines of $\cal R$), which contains a $4$--dimensional family $\cal C$ of (generically irreducible) conics such that there is a conic in $\cal C$ passing through two general points of $Y$ and the counterimage of the general conic of $\cal C$ via the projection from $\Pi$ is a non--developable scroll spanning a 4--space with a line directrix sitting in $\Pi$. In any event $r\leq 10+\varepsilon$.
\end{itemize}

\end {theorem}

Before proving Theorem \ref {prop:acc5}, we need a few preliminary lemmas which are due to Scorza (see \cite {Scorza1}, \S 11).

\begin{lemma}\label{lem:one} Let $\Sigma\subset \p^4$ be an irreducible, non--degenerate, projective surface, which is a non--developable scroll. Let $s$ be a general ruling of $\Sigma$. Then there is a curve $\mathfrak C_s$  on $\Sigma$, union of irreducible components such that the tangent planes to (a smooth branch of) $\Sigma$ at the general points of each of these components intersect
$s$ at a point.  More precisely,  $\mathfrak C_s$ consists of an irreducible  component $\mathfrak C'_s$ which is unisecant the rulings, plus, may be, a set of rulings. If the curve $\mathfrak C'_s$ is independent on $s$, then $\Sigma$ has a \emph{directrix line} $r$ (i.e., a line which  intersects the rulings in one point) and in this case $\mathfrak C'=r$ and the further set of rulings is empty if the scroll is rational and the projection of $\Sigma$ from $r$ is a conic with fibres single rulings.
\end{lemma}

\begin{proof} Let $t$ be a further general ruling of $\Sigma$. Since $\Sigma$ is not developable, the tangent spaces to $\Sigma$ at the points of $t$ vary in a pencil which sits
in a 3--space $P_t$, that cannot contain $s$ which is a general ruling, because $\Sigma$ is non--degenerate in $\p^4$. Hence $P_t$ cuts $s$ at a point $q_t$. Then the plane $\langle q_t,t\rangle$ is tangent to (a branch of) $\Sigma$ at a point $p_s\in t$. As $t$ varies, the point $p_s$ describes the curve $\mathfrak C'_s$ which is unisecant  the rulings and is a component of the desired curve $\mathfrak C_s$. 
The remaining components of  $\mathfrak C_s$  are those rulings $u$ such that $P_u$ contains $s$. 

Suppose now $\mathfrak C'_s=\mathfrak C'$ does not depend on $s$. Keeping the above notation, keep $t$ fixed and let $s$ vary. Then the point $p:=p_s=\mathfrak C'\cap t$ does not depend on $s$, and therefore all rulings intersect in one point the tangent plane $\Pi_{t,p}$ to the branch of $\Sigma$ with origin $p$ containing $t$, and $\Pi_{t,p}$ contains $t$.  Similarly, all rulings intersect in one point the tangent plane $\Pi_{s,q}$ to the branch of $\Sigma$ with origin $q$ containing $s$, where $q=\mathfrak C'\cap s$ and  $\Pi_{s,q}$ contains $s$. Note now that $\Pi_{t,p}$ and $\Pi_{s,q}$ are distinct (because $t$ and $s$ being general rulings of $\Sigma$ are not coplanar), and intersect along a line, precisely the line $r$ joining the two points $\Pi_{t,p}\cap s$ and $\Pi_{s,q}\cap t$. So $\langle \Pi_{t,p},\Pi_{s,q}\rangle$ is a 3--space. Since all rulings intersect both $\Pi_{t,p}$ and $\Pi_{s,q}$ and since $\Sigma$ is non--degenerate in $\p^4$, then all rulings intersect $r$, i.e., $\Sigma$ has a line directrix. 

Finally, assume that $\Sigma$ has a line directrix $r$. Then it is clear that $\mathfrak C'=r$, because if $p\in r$ is a general point, the tangent plane to (any branch of) $\Sigma$ along $r$ contains $r$ hence it meets $s$. 

Let finally $d$ be the degree of $\Sigma$ and $g$ be its genus (i.e., the genus of the normalization $C$ of the curve which parameterizes the rulings). Let moreover $n$ be the multiplicity of $\Sigma$ along the line directrix $r$. The hyperplanes which contain $r$ and a general ruling $s$ cut out on $\Sigma$ a set of $d-n-1$ rulings, which vary in a $g^1_{d-n-1}$ on $C$ which may have  $x\geq 0$  base points, so that the base point free series is a $g^1_{d-n-1-x}$. By Riemann--Hurwitz formula, this linear series has $2(d+g-n-2-x)$ ramification points (to be counted with their multiplicities). These ramification points, plus the $x$ lines corresponding to  the base points of the $g^1_{d-n-1}$, 
 correspond exactly to the rulings $u$ such that $P_u$ contains $s$. So the number of these ruling is $2(d+g-n-2)-x$. This number is zero if and only if $2(d+g-n-2)=x$. Note  that $d-n-1-x\geq 1$, i.e., $0\leq x\leq d-n-2$, hence $d\geq n+2$. On the other hand if $2(d+g-n-2)=x$ we have $2(d+g-n-2)\leq d-n-2$, which reads $d+2g-n-2\leq 0$, that forces $d=n+2$, $g=0$ and $x=0$. Whence the assertion follows
\end{proof}

To state the second lemma, recall that an irreducible, projective surface $S\subset \mathbb G(1,r)$ is called a \emph{congruence} of lines in $\p^r$. Given an irreducible projective curve $C\subset S$ it corresponds to a scroll in $\p^r$, and it is well known that the scroll is developable if and only if for the general point $x\in C$ the tangent line to $C$ at $x$ lies in $\mathbb G(1,r)$. 

\begin{lemma}\label{lem:two} Let $S$ be a congruence of lines in $\p^r$, with $r\geq 3$. If all curves contained in $S$ correspond to developable scrolls, then one of the following occurs:
\begin{itemize}
\item [($\alpha$)] there is a plane $\Pi\subset \p^r$ such that $S$ consists of all the lines in $\Pi$;
\item [($\beta$)] there is a point $p\in \p^r$ such that $S$ consists of  lines containing $p$.
\end{itemize}

\end{lemma}
\begin{proof} By projecting down generically in $\p^3$ it suffices to prove the assertion for $r=3$. So we assume $r=3$, hence $S\subset Q\subset \p^5$, where $Q=\mathbb G(1,3)$ is a smooth quadric. 

Let $x\in S$ be a general point. Every line passing through $x$ in the tangent plane $T_{S,x}$ is tangent to some curve on $S$, and therefore this line is contained in $Q$. Hence $T_{S,x}\subset Q$ and therefore $T(S)\subseteq Q$. This implies that all tangent planes to $S$ in smooth points of $S$ belong to one and the same system of planes of $Q$, hence they pairwise intersect at a point, thus either $S$ is defective or it spans a 4--space. 

If $S$ is defective, it could either  be a cone or the Veronese surface $V_{2,2}$. The latter case is impossible, because $T(V_{2,2})$ is a cubic hypersurface in $\p^5$, not a quadric. 

\begin{claim}\label{cl:pop} If the surface $S$ is a cone, then either case ($\alpha$) or ($\beta$) occurs. 
\end{claim}

\begin{proof}[Proof of Claim \ref {cl:pop}] If $S$ is a cone, its vertex corresponds to a line $r$ in $\p^3$ such that all lines corresponding to the points of $S$ intersect $r$. Now take any curve $C$ on $S$. This corresponds to a developable surface $\Sigma$ in $\p^3$. We claim that $\Sigma$ is a cone. If not, $\Sigma$ would be the tangent developable surface  to a curve $Z$, and then all tangent lines to $Z$ would intersect $r$, which is not possible. 
Hence, whatever $C$ is, $\Sigma$ is a cone, and this implies that the lines corresponding to the points of $S$ pairwise intersect at a point, which implies that either ($\alpha$) or ($\beta$) occurs. \end{proof}

Suppose now that $S$ spans a $4$--space $\Pi$, which therefore contains also $T(S)$. This means that the quadric $Q'$ cut out by $\Pi$ on $Q$ is singular, because it contains planes. Hence $Q'$ is a cone with vertex a point corresponding to a line $s$, and the lines 
corresponding to the points of $S$ all intersect the line $s$. Then the same argument we made in the proof of Claim \ref {cl:pop} shows that we are either in case ($\alpha$) or ($\beta$). \end{proof} 

As a consequence we have the:

\begin{lemma}\label{lem:three} Let $V\subset \mathbb G(1,r)$ be a projective variety of dimension 3 possessing an irreducible family of (generically) irreducible curves $\cal C$ such that:
\begin{itemize}
\item [($\iota$)] all  curves in $\cal C$ correspond to developable scrolls;
\item [($\iota\iota$)] given two general points of $V$ there is a curve in $\cal C$ containing them.
\end{itemize}
Then there is a point $p\in \p^r$ such that any point of $V$ corresponds to a line containing $p$ (hence $r\geq 4$).  
\end{lemma}
\begin{proof} The assumption ($\iota\iota$) implies that given a  general point $x\in V$ and a tangent line $t$ to $V$ at $x$, there is a curve in $\cal C$ which passes through $x$ and tangent to the line $t$. This in turn implies that $T(V)\subseteq \mathbb G(1,r)$ and all curves in $V$ correspond to developable scrolls and the same happens for every surface $S\subset V$. By Lemma \ref {lem:two} this implies that the points of $V$  correspond to lines which pairwise intersect at a point. They cannot correspond to the lines in a plane, hence they correspond to lines passing through a fixed point.\end{proof}

Now we are ready for the:

\begin{proof} [Proof of Theorem \ref {prop:acc5}]
First of all we notice that, by Lemma \ref {lem:three}, the surfaces in $\cal S$, which are scrolls,  are not developable, otherwise $X$ would be a cone. 

Let $p_0,p_1\in X$ be general points. Let $r_0,r_1$ be the lines in $\cal R$ containing $p_0,p_1$ respectively. Then $\Gamma:=\Gamma_{p_0,p_1}$ is a non--developable scroll spanning a 4--space, of which $r_0,r_1$ are general rulings. By Lemma \ref {lem:one}, there is a point $x\in r_1$ such that the tangent plane to $\Gamma$ at $x$ intersects $r_0$. 

We first assume that, when $p_0$ moves, the point $x$ moves on $r_1$, so that it is a general point of $r_1$ and hence of $X$. 

\begin{claim}\label{cl:fun} Given a general line $r$ in $\cal R$ consider the linear space 
\[
\Pi_r=\langle \cup_{x\in r\cap \Reg(X)} T_{X,x}\rangle
\]
and set $v=\dim (\Pi_r)$. Then:
\begin{itemize}
\item [(1)] one has $4\leq v\leq 6$;
\item [(2)] all lines $s\in \cal R$ intersect $\Pi_r$;
\item [(3)] given two general lines $r,s\in \cal R$, one has $\dim(\langle \Pi_r, \Pi_s\rangle)\leq 8$.  
\end{itemize}
\end{claim}

\begin{proof}[Proof of the Claim \ref {cl:fun}] We keep the notation introduced above. Set $s:=r_0$ and $r:=r_1$ which are general lines in $\cal R$. 

We have the 2--dimensional family $\cal S_{p_1}$ of surfaces, which all contain the line $r$.
For each of these surfaces $\Gamma$, there is an $8$--space which is tangent to $X$ along $\Gamma$, and each of these $8$--spaces  is contained in a family of dimension $r-9$ of hyperplanes. Hence we have a family of dimension at most $r-7$ of hyperplanes, each of which is tangent to $X$ along a surface of $\cal S_{p_1}$. Of course 
 $\Pi_r$ is contained in all these hyperplanes. This implies that $4\leq v\leq 6$, proving (1).

The line $s$, intersecting $T_{\Gamma, x}$, intersects also $T_{X,x}$, then it intersects $\Pi_{r}$. This proves (2).

 Finally, both $\Pi_r, \Pi_s$ are contained in the $8$--space which is tangent to $X$ along $\Gamma$, proving (3). 
\end{proof}

Set now
\[
\Pi:= \cap_{r\in \cal R} \Pi_r
\]
and $\varepsilon=\dim (\Pi)$.

\begin{claim}\label{cl:fun2} If $r\in \cal R$ is a general line, then $r$ intersects $\Pi$ in a point and $1\leq \varepsilon\leq 3$.
\end{claim} 

\begin{proof}[Proof of the Claim \ref {cl:fun2}] Let $r,s,t\in \cal R$ be general lines. Then $t$ has non--empty intersection with both $\Pi_r$ and $\Pi_s$. If $t$ intersects $\Pi_r$ and $\Pi_s$ in distinct points, then $t$ lies in  $\langle \Pi_r, \Pi_s\rangle$. This is not possible, because $t\in \cal R$ is general, and then all of $X$ would lie in $\langle \Pi_r, \Pi_s\rangle$ which has dimension at most 8. Hence $t$ intersects $\Pi_r$ and $\Pi_s$ in a point of their intersection. By the genericity of $r$ and $s$ we see that $t$ intersects $\Pi$. It is not possible that $\Pi$ is a point, because then $X$ would be a cone. Hence $\varepsilon \geq 1$. On the other hand $\Pi$ is contained in all hyperplanes which are tangent to $X$ along the surfaces in $\cal S$. These hyperplanes form a $(r-5)$--dimensional family, which is not a linear system, because otherwise they would cut out a linear system on $X$ and the general divisor of this linear system would have singular points filling up $X$, against  Bertini's theorem. Hence the family of the hyperplanes in question has dimension strictly smaller that the family of hyperplanes containing $\Pi$, thus $r-5<r-(\epsilon+1)$, i.e., $\varepsilon \leq 3$. 
\end{proof}

Next we consider the projection $\pi: X\dasharrow \p^{r-\varepsilon-1}$ from $\Pi$ and we denote by $Y$ its image.

\begin{claim}\label{cl:purp} One has $2\leq \dim(Y)\leq 3$.

\end{claim}

\begin{proof}[Proof of the Claim \ref {cl:purp}] The projection $\pi$  contracts all lines of $\cal R$ to points in $Y$, hence $\dim(Y)\leq 3$. It cannot be the case that $\dim(Y)=0$ because $Y$ has to span $\p^{r-\varepsilon-1}$ and $r-\varepsilon-1\geq 8- \varepsilon\geq 5$. To finish we have to exclude that $\dim(Y)=1$. Assume, by contradiction, that this is the case. Then $X$ is swept out by a 1--dimensional family of $3$--folds (the fibres of $\pi$), each lying in a $\p^{\varepsilon+1}$. This implies $\varepsilon =3$, otherwise $X$ would be a cone with vertex $\Pi$ over a curve. Then the general tangent space to $X$  intersects $\Pi$ in at least  a plane, and  two general tangent spaces to $X$ intersect at least along a line, a contradiction. 
\end{proof}

\begin{claim}\label{cl:purp2} The case $\dim(Y)=2$ is not possible.

\end{claim}

\begin{proof}[Proof of the Claim \ref {cl:purp2}] Assume by contradiction that $\dim(Y)=2$. Then $X$ sits in the ($\varepsilon+3$)--dimensional cone $V$ with vertex $\Pi$ over $Y$. Since $X$ is not a cone, we have $\varepsilon \geq 2$. 

Let $x\in X$ be general. Then $T_{X,x}\subset T_{V,x}$, and $T_{V,x}$ is a $(\varepsilon+3)$-space which contains $\Pi$. Since the projection of $T_{X,x}$ from $\Pi$ is the general tangent space to $Y$, we see that  $T_{X,x}$ intersects $\Pi$ in  a line. 

Let $L_x$ be the line which $T_{X,x}$ cuts out on $\Pi$. Since $X$ is not a cone, the line $L_x$ moves when $x$ varies and describes a family $\cal L$ of lines in $\Pi$. 

We claim that two general lines in $\cal L$ intersect each other. Otherwise,  two general lines in $\cal L$ are skew, and then $\varepsilon=3$ by Claim \ref {cl:fun2}. Let $p_0,p_1\in X$ be general points. Then $T_{X,p_0,p_1}$, contains $\Pi$ and $T_{X,p_0}$ and $T_{X,p_1}$, hence it coincides with  $T_{V,p_0,p_1}$, which is tangent to $X$ along the union of the two surfaces which are the fibres of $\pi$ passing through $p_0$ and $p_1$. This is a contradiction, because the surface $\Gamma_{p_0,p_1}$ is irreducible. We have thus proved that two general lines in $\cal L$ meet.

If the lines in $\cal L$ pairwise meet, then they all belong to the same plane $P\subseteq \Pi$, because they cannot pass through the same point otherwise $X$ would be a cone. Consider the projection $\pi': X\dasharrow \p^{r-3}$ from $P$. The image of $\pi'$ is a surface $Y'$ because the general tangent space to $X$ intersects the centre of projection along a line. Moreover $X$ sits in the 5--dimensional cone $W$ over $Y'$ with vertex $P$. Then  we repeat the same argument as above. Namely, let $p_0,p_1\in X$ be general points. Then $T_{X,p_0,p_1}$, containing $P$, $T_{X,p_0}$ and $T_{X,p_1}$, coincides with 
$T_{W,p_0,p_1}$, which is tangent to $X$ along the union of  two surfaces, i.e.,  the fibres of $\pi'$ passing through $p_0$ and $p_1$, leading again to a contradiction. 
\end{proof}

Hence $\dim(Y)=3$. Note now that, by the proof of (3) of Claim \ref {cl:fun}, all
tangent $8$--spaces to $X$ along a surface $\Gamma$ in $\cal S$ contain $\Pi$. Under the projection  $\pi$ from $\Pi$ all these $8$--space are mapped to a $4$--dimensional family  of $(7-\varepsilon)$--spaces. If $q_0,q_1\in Y$ are general points, then $T_{Y,q_0,q_1}$ is contained in one of these $(7-\varepsilon)$--spaces. 
This implies that the two tangent spaces to $Y$ at the general points $q_0,q_1\in Y$ intersect in dimension $\lambda$ with 
\[
6-\lambda =\dim(T_{Y,q_0})+\dim(T_{Y,q_1})-\lambda\leq 7-\varepsilon
\]
which implies $\lambda \geq \varepsilon-1$. 

If $\varepsilon=3$, then two general tangent tangent spaces to $Y\subset \p^{r-4}$ intersect in dimension 2. Then a general curve section of $Y$ has tangent lines which pairwise meet, hence it would be a plane curve, whereas it spans a $\p^{r-6}$ and $r-6\geq 3$, a contradiction.

If $\varepsilon=2$, then two general tangent  spaces to $Y\subset \p^{r-3}$ intersect in dimension 1. Then the general hyperplane section of $Y$, which spans a $\p^{r-4}$, is a defective surface, hence either it is a cone or a Veronese surface $V_{2,2}$. Hence $Y$ itself is either a cone over a curve or a cone over the Veronese surface $V_{2,2}$. In either case, $X$ would be of the first species,  contradiction.

Finally, if $\varepsilon=1$, then two general tangent tangent spaces to $Y\subset \p^{r-2}$ intersect in dimension 0, i.e., $Y$ is a defective 3--fold and $X$ sits in a cone with vertex a line over $Y$. By \cite {ChCi}, then $Y$ is either a hyperplane section of ${\rm Seg}(2,2)$ (and then $r=9$), or 
 the Veronese $3$--fold $V_{3,2}$ in $\p^9$ or a projection of it in $\p^8$ or $\p^7$ (and then $9\leq r\leq 11$), or it sits in a 4--dimensional cone over the Veronese surface $V_{2,2}$ (and again $r=9$) or in a 4-dimensional cone over a curve. However the last case not possible. Indeed, if $Y$ sits in a 4-dimensional cone over a curve, then the general contact locus of $X$ would be reducible. Hence we are in one of the cases (1), (2), (3) of the statement of Theorem \ref {prop:acc5}.\medskip

Suppose next that when $p_0$ moves, the point $x$ does not move on $r_1$. If $x\in \Reg(X)$, then we can repeat the above argument with no  change, reaching the same conclusions. So let us assume that $x\in \Sing(X)$. Note that, by Lemma \ref {lem:one}, this is the case only if $\Gamma_{p_0,p_1}$ has a line directrix which passes through $x$ and, by the generality of $r_1$, it consists of singular points of $X$. We want to prove that in this case we still have a similar conclusion as in Claim \ref {cl:fun2}.   In other words, we have the:

\begin{claim}\label{cl:fun3} Suppose that when $p_0$ moves, the point $x$ does not move on $r_1$ and $x\in \Sing(X)$. Then there is a subspace $\Pi\subseteq \Sing(X)$ of dimension $\varepsilon$, with $1\leq \varepsilon\leq 3$, such that if $\ell\in \cal R$ is a general line, then $\ell$ intersects $\Pi$ in a point. Moreover if $\Gamma$ is a general surface in $\cal S$, then $\Gamma$ is a (non--developable) scroll with line directrix lying in $\Pi$. 
\end{claim} 

\begin{proof}[Proof of the Claim \ref {cl:fun3}]
In this situation, as $p_0,p_1$ vary, the point $x$ describes a variety $V\subseteq \Sing(X)$, of dimension $\varepsilon$, with $1\leq \varepsilon\leq 3$. So $V$ is described by the line directrices of the scrolls $\Gamma_{p_0,p_1}$ as $p_0, p_1$ vary. 
Hence the general line in $\cal R$ intersects $V$ in its general point. Let $x,y\in V$ be general points. Take lines $\xi,\eta\in \cal R$, with $x\in \xi$ and $y\in \eta$, which are general lines in $\cal R$,  and consider the surface $\Gamma$ in $\cal S$ containing $\xi, \eta$. Then the line joining $x$ and $y$ is exactly the line directrix of $\Gamma$, and it lies on $V$. In conclusion, $V$ is a linear space $\Pi$ proving the assertion. 
\end{proof}

Now again we consider the projection $\pi: X\dasharrow \p^{r-\varepsilon-1}$ from $\Pi$ and let $Y$ be its image. As in Claims \ref {cl:purp} and \ref {cl:purp2}, we see that $Y$ has to be a 3--fold. If $\Gamma$ is a general surface in $\cal S$, its image via $\pi$ is a  plane curve, because $\Gamma$ spans a 4--space and we are projecting it from the line directrix which sits in this 4--space. Hence $Y$ contains a family $\cal C$ of plane curves (the images of the surfaces $\Gamma$ in $\cal S$), such that  there is a curve of $\cal C$ passing through two general points of $Y$. Then, by the Trisecant Lemma,  the curves of $\cal C$ are conics. The counterimage via $\pi$ of the general conic in $\cal C$ is a scroll with a line directrix contained in $\Pi$.  Note that either $r-\varepsilon-1\leq 6$, or $Y$ is defective. Then by Theorem 1.1 of \cite {ChCi}, $r-\varepsilon-1\leq 9$, thus 
$r\leq 10+\varepsilon$. Hence we are in case (4) of the statement of Theorem \ref {prop:acc5}.\end{proof}

\begin{remark}\label{rem:lamn} If $X$ is as in cases (1), (2), (3) of Theorem \ref {prop:acc5}, then $X$ is indeed defective, $f(X)=1$ and the general contact loci of $X$ are irreducible surfaces. In case (4), which has been overlooked by Scorza in  \cite {Scorza1}, \S 11, $X$ is still defective with $f(X)=1$, because $X$ is covered by a 4--dimensional family of scrolls $\Gamma$ each spanning a 4--space and such that through two general points of $X$ there is a unique scroll $\Gamma$ containing them. Then the general secant line $r$ to $X$ is also a secant to a scroll $\Gamma$ and therefore there is a 1--dimensional family of secants to $\Gamma$, hence to $X$, containing the general point of $r$. It is however not clear if the surfaces $\Gamma$ are contact loci, neither we have examples of such 4--folds. 

Arguing as in Remark \ref {rm:acc31} one sees that all 4--folds as in Theorem \ref {prop:acc5} are singular. 
\end{remark}

\subsection{} In this section we discuss case (ii), in which two general surfaces in $\cal S$ have some isolated point in common off the base locus scheme $B$. We keep all notation and conventions introduced above, in particular we assume $X$ is not a cone. We will denote by $j$ the dimension of the intersection of the spans of two general surfaces in $\cal S$. One has $0\leq j\leq 3$. 

We will prove the:

\begin{theorem}\label{thm:lof} In case (ii) we have the following possibilities:
\begin{itemize}
\item [$(*)$]  $X$ sits in a cone with vertex a point over ${\rm Seg}(2,2)$ in $\p^8$;
\item [$(**)$]  $r=9$ and $X$ sits in a cone with vertex a line over a hyperplane section of ${\rm Seg}(2,2)$ in $\p^7$;
\item [$(***)$] or $9\leq r\leq 11$ and $X$ sits in a cone with vertex a line over the Veronese $3$--fold $V_{3,2}$ in $\p^9$ or the projection of it in $\p^8$ or $\p^7$;
\item [$(****)$]  $j=0$, hence given two general surfaces $\Gamma$, $\Gamma'$ in $\cal S$, the 4--spaces spanned by them intersect in one point which is also the intersection of $\Gamma$ and $\Gamma'$. In this case the general surface in $\cal S$ is rational and $X$ itself is rational. 
\end{itemize}

\end{theorem}

\begin{proof} We start with the following:

\begin{claim}\label{cl:fad} There is no subspace $\Pi$ of $\p^r$ with $\dim(\Pi)=k<r-1$ such that for the general surface $\Gamma$ in $\cal S$ the 4--space spanned by $\Gamma$ intersects $\Pi$ in a subspace of dimension 3. 
\end{claim}

\begin{proof}[Proof of the Claim \ref {cl:fad}] Suppose by contradiction that there is such a subspace $\Pi$. Let $x\in X$ be a general point. Then for all surfaces $\Gamma$ in $\cal S_x$ the span of such a surface is contained in $\langle \Pi, x\rangle$, which has dimension at most $k+1<r$. Hence the whole of $X$ would be contained in $\langle \Pi, x\rangle$, a contradiction because $X$ is non--degenerate in $\p^r$.
\end{proof}

\begin{claim}\label{cl:fad2} There is no plane $P$ such that for the general surface $\Gamma$ in $\cal S$ the 4--space spanned by $\Gamma$ contains $P$. 
\end{claim}

\begin{proof}[Proof of the Claim \ref {cl:fad2}] Suppose by contradiction that there is such a plane $P$. Let $\pi: X\dasharrow \p^{r-3}$ be the projection from $P$, and $Y$  the image of $\pi$.   
If $\Gamma$ is the general surface in $\cal S$, then $\Gamma$ is mapped to a line by $\pi$. Thus $Y$ is a 3--fold such that if $x,y\in Y$ are general points, there is a line in $Y$ containing $x$ and $y$. Hence $Y$ would be a 3--space, a contradiction, because $Y$ has to span the $\p^{r-3}$ and $r-3\geq 6$. \end{proof}

Let first be $j=3$. Then, since the $4$--spaces spanned by the surfaces in $\cal S$ cannot lie in a 5--space, they have to pass through the same 3--space, and this contradicts Claim \ref {cl:fad}.

Let  $j=2$. Consider two general surfaces $\Gamma _1,\Gamma_2$ in $\cal S$, and their spans $\Pi_1,\Pi_2$. Then $\Pi'=\langle \Pi_1,\Pi_2\rangle$ has dimension $6$. Let  $\Gamma$ be another general surface in $\cal S$ and $\Pi$ be its span. Of course $\Pi$ cannot be contained in $\Pi'$ and by Claim \ref {cl:fad} it cannot intersect it in a 3--space. This implies that $\Pi$ intersects $\Pi'$ in the intersection plane of $\Pi_1$ and $\Pi_2$, which contradicts Claim \ref {cl:fad2}.

Let  $j=1$. Let us argue as in the case $j=2$. Consider two general surfaces $\Gamma _1,\Gamma_2$ in $\cal S$, and their spans $\Pi_1,\Pi_2$. Then $\Pi'=\langle \Pi_1,\Pi_2\rangle$ has dimension $7$. Let  $\Gamma$ be another general surface in $\cal S$ and $\Pi$ be its span, which cannot be contained in $\Pi'$ and by Claim \ref {cl:fad} it cannot intersect it in a 3--space. Then either 
 $\Pi$ intersects $\Pi'$ in the intersection line of $\Pi_1$ and $\Pi_2$ (hence all the spans of the surfaces of $\cal S$ pass through the same line $\ell$), or it intersects $\Pi'$ in a plane which contains the two distinct intersection lines of $\Pi$ with $\Pi_1$ and $\Pi_2$. 
 
In the former case let $\pi: X\dasharrow \p^{r-2}$ be the  projection of $X$ from the line $\ell$ and $Y$ be its image. We claim that  $Y$ is a defective 3--fold such that there is an irreducible conic in $Y$ containing two general points of $Y$. Then we are in cases $(**)$ or $(***)$. To prove the claim, note first that if $\Gamma$ is a general surface in $\cal S$, then its projection from $\ell$ is not a plane. Otherwise $Y$ would be swept out by a family of planes such that through two general points of $Y$ there is a plane of the family passing, hence $Y$ would be a linear space, a contradiction. Then the general $\Gamma$ in $\cal S$ is projected from $\ell$ to a plane curve, which implies that $\dim(Y)=3$. Moreover $Y$ is covered by a family of generically irreducible plane curves such that through two general points of $Y$ there is a curve of the family passing. This implies that the plane curves are conics and $Y$ is defective.
 
In the latter case the intersection lines of the spans of two general surfaces of $\cal S$ are generically distinct but they intersect in one point. Then either all these lines lie in the same plane or they pass through the same point $p$ but do not lie in the same plane. The former case is impossible by Claim \ref {cl:fad2}. In the latter case consider the projection $\pi: X\dasharrow \p^{r-1}$ from $p$ and let $Y$ be its image. Since $X$ is not a cone, then $Y$ is a 4--fold. The general surface in $\cal S$ maps to a surface spanning a 3--space. Hence $Y$ has an irreducible, 4--dimensional family $\cal Q$ of generically irreducible contact surfaces spanning a $3$--space, such that given two general points of $Y$ there is a surface in $\cal Q$ containing the two points. By the Trisecant Lemma, the general surfaces in $\cal Q$ are irreducible quadrics and therefore $f(Y)\geq 2$. 

If $f(Y)=3$, then by 
 Proposition \ref {difetti}, (i), $Y$ would either be a cone over a curve or a cone over $V_{2,2}$. Both cases are impossible because in the former case $Y$ would not contain the family $\cal Q$ of contact quadrics, in the latter $Y$ would lie in a $7$--space, whereas $Y$ spans a $\p^{r-1}$ and $r-1\geq 8$.

So we have that $f(Y)= 2$, and $Y$ is in the list of Theorem \ref {thm:lowcod}. By inspection, the only cases which are compatible with the present situation are the cases (iii), (iv) and (v) of Theorem \ref {thm:lowcod}, and then  we are in cases  $(*), (**)$ or $(***)$. 

Finally suppose we are in case $j=0$. Then we are in case $(****)$ and we will prove that the general surface in $\cal S$ is rational and $X$ itself is rational. To see this, let $x\in X$ be a general point. If $t$ is a general tangent direction of $X$ at $x$, then there is a unique surface in $\cal S_x$ tangent to $t$. This implies that $\cal S_x$, which has dimension 2, is unirational, hence rational. Let now $\Gamma$ be a general surface in $\cal S$. We have a natural map $\rho: \cal S_x \dasharrow \Gamma$, which takes the general surface $\Gamma'$ in $\cal S_x$ and maps it to the unique intersection point of $\Gamma'$ and $\Gamma$, which is also the unique intersection point of their spans. This map is clearly dominant, hence $\Gamma$ is unirational, so it is rational. Finally let $x,y\in X$ be general points. There is a natural map $\sigma: \cal S_x\times \cal S_y\dasharrow X$ which maps the pair $(\Gamma, \Gamma')$ of general surfaces in  $\cal S_x\times \cal S_y$ to their unique intersection point. This map is clearly dominant and it is also generically injective, hence $X$ is rational. \end{proof}

\begin{remark}\label{rem:klop} In cases $(*), (**), (***)$ listed in Theorem \ref {thm:lof}, it is clear that $X$ is defective and verifies (ii). In case $(****)$ the 4--fold $X$ is still defective, since the surfaces $\Gamma$ on $\cal S$ span $4$--spaces, hence $f(X)=1$. However we do not have examples of such 4--folds, hence we do not know if this case really occurs. 

Moreover, if $X$ is smooth, as we saw more than once, cases $(**)$ and $(***)$ are not possible. 
\end{remark}

\subsection{} Here we discuss the reducible case, in which given two general points $p_0,p_1\in X$ there is a unique tangential contact surface $\Gamma_{p_0,p_1}$ containing $p_0,p_1$, which consists of two irreducible components each passing  through one of the two points $p_0,p_1$. We will denote by $\cal T$ the irreducible 2--dimensional family of surfaces of $X$ such that any  tangential contact surface $\Gamma$ is the sum of two components each in $\cal T$. Hence, given a general point $x\in X$ there is a unique surface in $\cal T$ containing $x$. We will denote by $i$ the dimension of the span of the general surface in $\cal T$. Recall that the general  tangential contact surface $\Gamma$ spans a 4--space. Hence the union of two general surfaces in $\cal T$ spans a 4--space. Therefore $2\leq i\leq 3$. Indeed, clearly $i\geq 2$ and moreover $i\leq 4$ because, as we said, the union of two general surfaces in $\cal T$ spans a 4--space. Moreover it cannot be the case that $i=4$, otherwise the span of a general surface in $\cal T$ would contain any other surface in $\cal T$, hence it would contain $X$.

\begin{proposition}\label{thm:gup} In the above  setting, if $X$ is not a cone, then  $X$ sits in a cone with vertex a plane over a surface. 
\end{proposition}

\begin{proof} With the above notation, assume that $i=3$. Then if $S,S'$ are two general surfaces in $\cal T$, one has $\dim(\langle S\rangle\cap \langle S'\rangle)=2$. Then either the spans of the surfaces in $\cal T$ lie in a projective space of dimension 3, which is not possible, or there is a plane $P$ which is contained in $\langle S\rangle$ for all $S$ in $\cal T$. By projecting from $P$ we see $X$ sits in a cone with vertex the plane $P$ over a surface. 

Assume $i=2$, then $\cal T$ is a family of planes. These planes pairwise span a $4$--space, hence they pairwise intersect at a point. Then the only possibility is that there is a plane $P$ such that all planes in $\cal T$ intersect $P$ along a line. By projecting from $P$ we see again that $X$ sits in a cone with vertex the plane $P$ over a surface. 
\end{proof}

\begin{remark}\label{rem:ging} If $X$ is a 4--fold as in Proposition \ref {thm:gup}, then $X$ is defective. In fact if $P$ is the plane vertex of the 5--dimensional cone in which $X$ sits, and if $x\in X$ is a general point, then $T_{X,x}$ intersects $P$ along a line, and therefore if $y\in X$ is another general point, $T_{X,x}$ and $T_{X,y}$ do intersect at a point.

Arguing as in Remark \ref {rm:acc31}, one sees that 4--folds $X$ as in Proposition \ref {thm:gup} are never smooth.
\end{remark}

 \section {First species}\label {sec:first}

\subsection {} Next we consider a projective,  irreducible,  non--degenerate 4--fold, not a scroll  $X\subset \p^ r$, $r\geq 9$, of the first species, i.e., $\gamma(X)=3$. Taking into account Proposition  \ref {difetti} and Theorem \ref  {thm:lowcod}  we may assume $f(X)=\delta(X)=1$. 

Given $p_0,p_1\in X$ general points, the tangent spaces to $X$ at $p_0$ and $p_1$ intersect at one point and the tangential contact locus $\Gamma_{X,p_0,p_1}$ is a 3--fold, spanning a $6$--space $\Pi_{X,p_0,p_1}$ (see Proposition \ref  {lem:a}). 
We have two cases: the \emph {irreducible} and the \emph{reducible} case, according to the possibilities that $\Gamma_{X, p_0,p_1}$ is irreducible or it consists of two irreducible components each passing through one of the two point $p_0,p_1$. 

By projecting down generically to $\p^9$, we will assume that $r=9$. 
%%%Take now a general hyperplane $H$ in $\p^9$ and consider the hyperplane section $Y=X\cap H$. Then $Y\subset \p^8$ is a projective,  irreducible,  non--degenerate 3--fold which is not defective. However if $p_0,p_1\in Y$ are general points, the tangent spaces to $Y$ at $p_0$ and $p_1$ span a hyperplane in $\p^8$, which is tangent to $Y$ at $p_0$ and $p_1$, but it is also tangent to $Y$ along the surface $\Gamma_{Y,p_0,p_1}=\Gamma_{X,p_0,p_1}\cap H$ (see Proposition \ref {rem:hh}). Hence  $Y$ is 1--weakly defective with $1$--singular defect equal to 2, hence it is as in Theorems \ref {thm:wd5} and  \ref {thm:wd6}. 
In this case, given $p_0,p_1\in X$ general points, there is a unique hyperplane bitangent to $X$ at $p_0$ and $p_1$, hence there is a unique 3--fold $\Gamma_{p_0,p_1}$ as above, which is smooth at $p_0$ and $p_1$. 

\begin{lemma}\label{lem:wd} In the above setting we have the following cases:
\begin{itemize}
\item [(i)] in the irreducible case the 3--folds $\Gamma_{p_0,p_1}$ move in a 2--dimensional family which is a linear system;
\item [(ii)] in the reducible case, the two components of $\Gamma_{p_0,p_1}$ move in a 1--dimensional (not necessarily rational) pencil.
\end{itemize}
\end{lemma}
\begin{proof} Case (i) follows from Thm. 5.10 of \cite {WDV}. Case (ii) follows from the fact that given $x\in X$ general  there is a unique component of a contact variety passing through $x$. 
\end{proof}

Let $x\in X$ be a general point and consider the tangential projection $\tau_{x}: X\dasharrow X_1\subset \p^{4}$. Since $X$ is defective and $f(X)=1$, then $X_1$ is a 3--fold. 

\begin{lemma}\label{lem:wd2} In the above setting, we have $t(X_1)=d(X_1)=2$, hence $X_1$ is swept out by a 1--dimensional family $\cal P$ of planes along which the tangent space to $X_1$ is constant, so that the dual variety $X_1^*$ of $X_1$ is a curve. Each plane of the family $\cal P$ is the image via $\tau_{x}$ of 
a 3--fold $\Gamma_{x,y}$. 
\end{lemma}

\begin{proof} The first assertion follows by Remark \ref {conctactin} and by the obvious fact that for a hypersurface $Z$ in projective space one has $t(Z)=d(Z)$. The second assertion is also trivial.\end{proof}

\subsection{} In this section we examine the irreducible case. We keep all notation and convention we introduced above.

\begin{lemma}\label{lem:wd3} In the above setting, if we are in the irreducible case, then the curve $X_1^*$ is rational.
\end{lemma}

\begin{proof}
The varieties $\Gamma_{x,y}$ passing through $x$ form a rational pencil, by Lemma \ref {lem:wd},(i), so that the family $\cal P$ is also rational. The rationality of $X_1^*$ follows. 
\end{proof}

Next we want to study more closely the general tangential projection $\tau_{x}: X\dasharrow X_1\subset \p^{4}$ and the  3--fold $X_1$. By Thm. (2.20) from \cite {GrHarr}, $X_1$ is of one of the following types:
\begin{itemize}
\item [(i)] there is a non--degenerate, irreducible, projective curve $C\subset \p^4$ such that $X_1$ is the closure of the union of the osculating planes to $C$ at all non--flex points of $C$; the tangent space to $X_1$ along a general osculating plane to $C$ is the corresponding osculating 3--space;
\item [(ii)] $X_1$ is the cone with vertex a point over a tangent developable scroll $\Sigma$ in $\p^3$;
\item [(iii)] $X_1$ is the cone with vertex a line over a plane curve.
\end{itemize}

\begin{lemma}\label{lem:wd4} In the above setting only case (iii) occurs.
\end{lemma}

\begin{proof} We follow Scorza's argument in \S 6 of \cite {Scorza1}. 

Suppose first we are in case (i). There are formulae (see \cite {Ber}, p. 490, and the Appendix) that give the degree $n_1$ of the tangent developable to $C$, the degree $n_2$ of $X_1$, i.e., the locus of osculating planes to $C$, the \emph{class} $n_3$ of $C$, i.e., the degree of $X_1^*$. For any branch $\xi$ of $C$, let $(\alpha, \alpha_1,\alpha_2)$ be the (truncated) \emph{rank sequence} of $\xi$, i.e.,  
\begin{itemize}
\item [$\bullet$] $\alpha$ is the \emph{degree} of the branch, i.e., the intersection multiplicity of the branch with the general hyperplane through the origin of the branch;
\item [$\bullet$] $\alpha+\alpha_1$ is the intersection  multiplicity of $\xi$ with a general hyperplane containing the tangent line to $\xi$ at the origin;
\item [$\bullet$] $\alpha+\alpha_1+\alpha_2$ is the intersection  multiplicity of $\xi$ with a general tangent hyperplane passing through the osculating plane to the branch. 
\end{itemize}
Note that for a general branch one has $\alpha=\alpha_1=\alpha_2=1$. Let $d$ be the degree of $C$, while, as we know, its genus is $0$. The formulae in question give
\begin{equation}\label{eq:form}
\begin{split}
n_1&=2(d-1)-\sum(\alpha-1)\\
n_2&=3(d-2)-\sum (2\alpha+\alpha_1-3)\\
n_3&=4(d-3)-\sum(3\alpha+2\alpha_1+\alpha_2-6)
\end{split}
\end{equation}
where the sums are taken over all  branches of $C$. 

Consider again the general tangential projection $\tau_{x}: X\dasharrow X_1\subset \p^{4}$.
Let $Y$ be a general hyperplane section of $X$ passing through $x$. Note that $Y$ is not defective. However if $p_0,p_1$ are general points of $Y$, the $7$--space $T_{Y,p_0,p_1}$ is tangent to $Y$ along a surface $\Gamma_{Y,p_0,p_1}$, which is the intersection of the 3--fold $\Gamma_{X,p_0,p_1}$ with the hyperplane spanned by $Y$. 

We can then consider the tangential projection $\tau_{Y,x}: Y\dasharrow Y_1\subset \p^{4}$. Since $Y$ is not defective, then $Y_1$ has dimension 3 and therefore $Y_1=X_1$. We let $s$ be the degree of the map $\tau_{Y,x}$. 

Take a general plane $\Pi$ in $\p^4$ which cuts out an irreducible rational plane curve $Z$ of degree $n_2$ on $X_1$. Let $C$ be the normalization of the counterimage $Z'$ of $Z$ via the generically finite map $\tau_{Y,x}$. We pretend $C$ to be irreducible: if this is not the case the argument is quite similar and can be left to the reader.   

Note that on $C$ there is the  $g^1_s$ consisting of the fibres of the map of $C$ to $Z$. This is the pull--back on $C$ of the movable part of the linear series cut out on $Z'$ by the rational pencil described by the surfaces $\Gamma_{Y,x,y}$ on $Y$ passing through $x$. 
The hyperplanes of $\p^4$ (which come via the tangential projections $\tau_{X,x}$ and $\tau_{Y,x}$ from hyperplanes containing $T_{X,x}$ or $T_{Y,x}$ respectively)
cut out on $C$ the pull--back to $C$ of the linear series cut out on $Z$ by the lines of $\Pi$, hence these hyperplanes cut out on $C$ 
a linear series $g^2_{sn_2}$ which is composite with the above $g^1_s$.
Hence, a hyperplane passing through $T_{Y,x}$ and tangent to $Y$ along a surface $\Gamma_{Y,x,y}$ cuts out on $C$ a divisor consisting of a divisor of the $g^1_s$, counted with multiplicity 2, plus $n_2-2$ more divisors of the $g^1_s$. 

Let $p\in C$ be a general point, which by abusing notation, we may identify with a general point of $X$ and let $q$ be its image on $Z$. Note that $q$ is also a general point of $X_1$.

The number of hyperplanes passing through $T_{Y,x}$ and $p$ and tangent to $Y$ along surfaces of the type  $\Gamma_{Y,x,y}$ which do not pass through $p$  equals 
the number of divisors of the $g^2_{sn_2}$ which contain the divisor of the $g^1_s$ containing $p$ and then another divisor of the $g^1_s$ counted with multiplicity 2. This number  coincides with the number of lines of $\Pi$ passing through $q$ and tangent to $Z$ at a point different from $q$. This number is  
$2(n_2-2)$. 

On the other hand, we can compute this number in another way. Indeed, this number is nothing but the number of tangent hyperplanes to $X_1$ passing through $q$ but not tangent to $X_1$ in $q$. An obvious duality argument shows that 
this equals the number of intersections of a general tangent hyperplane to $X_1^*$ off the contact point, i.e., $n_3-2$.

In conclusion we have $n_3-2=2(n_2-2)$, which, by using \eqref {eq:form}, reads
\[
\sum(\alpha-1)=2(d-1)+ \sum(\alpha_2+5)
\]
and implies $n_1=2(d-1)-\sum(\alpha-1)<0$, a contradiction. Hence case (i) cannot happen.

Assume next we are in case (ii). Consider  the tangent developable scroll $\Sigma\subset \p^3$ and let $\mathfrak C$ be the rational curve of degree $d$ whose tangent lines fill up $\Sigma$. For any branch $\xi$ of $\mathfrak C$, let $(\alpha, \alpha_1)$ be the (truncated) \emph{rank sequence} of $\xi$, with the same meaning as above. Then the degree of $\Sigma$, hence of $X_1$, is
\[
n_2=2(d-1)-\sum(\alpha-1)
\]
whereas the class of $\mathfrak C$, which coincides with the degree of $X_1^*$, is
\[
n_3=3(d-2)-\sum(2\alpha+\alpha_1-3).
\]
Then the same argument as above leads to the equality
\[
d+\sum(\alpha_1-1)=0
\]
which is also impossible. Then also case (ii) cannot happen, and only case (iii) can occur.
\end{proof}

\begin{theorem}\label{thm:wd5}  Let $X\subset \p^ 9$ be  a projective,  non--degenerate variety of dimension $4$, not  a scroll,  of the first species, which is not a cone. If $X$ presents the irreducible case, then $X$ sits in  a 6--dimensional cone with vertex a 3--space over a Veronese surface $V_{2,2}$ in a $\p^5$ skew with the the vertex.
\end{theorem}

\begin{proof} By Lemma \ref {lem:wd4}, given a general point $x\in X$, all hyperplanes passing through $T_{X,x}$ and tangent to a 3--fold $\Gamma_{x,y}$ also pass through a fixed $\p^6$, i.e., the join of $T_{X,x}$ with the line vertex of $X_1$. These hyperplanes form, in the dual plane of this $\p^6$ an irreducible rational curve.

Let us  consider the 2--dimensional family $\cal K$ of hyperplanes tangent to $X$ along a 3--fold of the type $\Gamma_{p_0,p_1}$. By duality, this family describes an irreducible surface $S\subset (\p^9)^*$ which possesses a family of dimension at least 2 of irreducible, rational, plane curves. 

\begin{claim} \label{cl:it} The surface $S$ spans a $\p^i$, with $2\leq i\leq 5$. \end{claim}

\begin{proof}[Proof of the Claim \ref {cl:it}] The bound $i\geq 2$ is obvious. Let us prove that $i\leq 5$. By Lemma 4.1 of \cite {ChCi}, the family $\cal P$ of the planes spanned by the plane curves on $S$ either pass through the same point $p$, or intersect a given plane $\Pi$ along lines, or they span at most a $5$--space. In the first case, by projecting $S$ from $p$ we find a surface which possesses a 2--dimensional family of lines, hence it is a plane, thus $i=3$. In the second case, by projecting $S$ from $\Pi$, we see that each plane curve of $S$ is contracted to one point, which again implies $i=3$. The Claim is thus proved. \end{proof}

More precisely:

\begin{claim} \label{cl:i5} One has $i=5$. \end{claim}

\begin{proof}[Proof of the Claim \ref {cl:i5}]

First of all $i=2$ is impossible, because then the hyperplanes of $\cal K$ vary in a 2--dimensional linear system cutting on $X$ a linear system of surfaces which are all singular with the singular points filling up $X$, contrary to Bertini's theorem.  

Assume $i=3$. In this case all the hyperplanes of $\cal K$ contain one and the same subspace $\Pi$ of dimension 5. Let $x\in X$ be a general point. The hyperplanes in $\cal K$ containing $T_{X,x}$ all contain the same $\p^6$ (see the beginning of the proof) which contains $\Pi$ and also $T_{X,x}$. Hence $T_{X,x}$ cuts $\Pi$ along a $3$--space and therefore two general tangent spaces to $X$  intersect in at least a line, a contradiction, because $f(X)=1$. 

Assume $i=4$. In this case all the hyperplanes of $\cal S$ contain one and the same subspace $\Pi$ of dimension 4. The same argument as in the case $i=3$ implies that all tangent spaces to $X$ intersect $\Pi$ along a plane. Consider the projection $\pi: X\dasharrow \p^4$, of $X$ from $\Pi$. The image of $X$ is  a curve $\Lambda$. The general hyperplane of $\cal K$ maps via $\pi$ to a general bitangent hyperplane to $\Lambda$ and the 3--folds $\Gamma_{p_0,p_1}$ would then be reducible in pairs of fibres of $\pi$, a contradiction since we are in the irreducible case.

In conclusion this proves that $i=5$. \end{proof}

By Claim \ref {cl:i5}, all the hyperplanes of $\cal K$ contain one and the same $3$--space $\Pi$. The same argument we made in the proof of the Claim \ref {cl:i5}, implies that all tangent spaces to $X$ intersect $\Pi$ along a line. Consider the projection $\pi: X\dasharrow \p^5$, of $X$ from $\Pi$. The image of $X$ is a surface $V$. The general hyperplane of $\cal K$ maps via $\pi$ to a general bitangent hyperplane to $V$, which is also tangent along an irreducible curve, the image of $\Gamma_{p_0,p_1}$. To finish our proof we need to prove the:

\begin{claim} \label{cl:iv} One has $V=V_{2,2}$ \end{claim}

\begin{proof}[Proof of the Claim \ref {cl:iv}] If $x,y\in V$ are general points, there is a bitangent hyperplane to $V$ at $x$ and $y$, hence the tangent spaces to $V$ at $x$ and $y$ intersect at a point, thus $V$ is defective with irreducible contact curve. Hence $V=V_{2,2}$ (see Thm. 1.3 of \cite {WDV}). 
\end{proof}
This ends our proof. \end{proof}

\begin{remark}\label{rem:vero} Let $X\subset \p^9$ be a 4--fold as in the statement of Theorem \ref {thm:wd5}, sitting in  a 6--dimensional cone $V$ with vertex a $3$--space $\Pi$ over a Veronese surface $V_{2,2}$. We claim that $X$ is in general defective, of the first species, presenting the irreducible case. 

Indeed, let $p_0, p_1\in X$ be general points. The space $T_{V,p_0,p_1}$ has dimension 8 and it contains $T_{X,p_0}$ and $T_{X,p_1}$, which therefore in general intersect at a point, proving that $X$ is defective with $f(X)=1$. Moreover $V$ is swept out by a 2--dimensional family of quadric cones of rank 3, which project from $\Pi$ the conics of the Veronese surface $V_{2,2}$. For each of these quadric cones $Q$, there is a hyperplane which is tangent to  $V$ along $Q$. In general the intersection of $Q$ with $X$ is an irreducible 3--fold of the type $\Gamma _{p_0,p_1}$, proving our claim.

As usual, $X$ is never smooth in this case. 
\end{remark}

\subsection{} In this section we examine the reducible case. We consider irreducible, defective, projective threefolds $X\subset \p^r$, of the first species, with $r\geq 9$, not a scroll, with $f(X)=1$. If $p_0,p_1\in X$ are general points, the tangent spaces to $X$ at $p_0$ and $p_1$ span a hyperplane which is tangent to $X$ along a {contact $3$--fold} $\Gamma_{p_0,p_1}$ which is reducible in two irreducible components each containing $p_0$ and $p_1$ respectively. By Lemma \ref {lem:wd}, these two components vary in a pencil which we denote by $\cal G$, and we denote by $\Gamma$ its general element. 

\begin{theorem}\label{thm:wd6} Let $X\subset \p^ r$ be  a projective,  non--degenerate variety of dimension 4, not  a scroll,  of the first species with $f(X)=1$, which is not a cone.
If $X$ presents the reducible case, then one of the following occurs:
\begin{itemize}
\item [(i)] $X$ lies in a 6--dimensional cone with vertex a $\p^4$ over a curve;
\item [(ii)] $X$ lies in a 5--dimensional cone with vertex a plane over a scroll surface.
\end{itemize}
\end{theorem}

\begin{proof}
Let us denote by $i$ the dimension of $\langle \Gamma\rangle$, with $\Gamma\in \mathcal G$ a general element. Since we are assuming $X$ not a scroll,  we have $i\geq 4$. On the other hand, if $\Gamma$ and $\Gamma'$ are two general elements of $\cal G$, then  $\langle \Gamma, \Gamma'\rangle$ has dimension 6 (see Proposition \ref {lem:a},(iv)). This implies that $i\leq 5$. 

If $i=4$ and if $\Gamma$ and $\Gamma'$ are two general elements of $\cal G$, then  $\langle \Gamma\rangle \cap \langle \Gamma'\rangle$
has dimension 2. By Lemma 4.1 of \cite  {ChCi}, we have only the following two possibilities:
\begin{itemize}
\item [(a)] there is a fixed plane $\Pi$ such that, for a general element $\Gamma$ of $\cal G$, one has $\Pi\subset \langle \Gamma\rangle$;
\item [(b)] there is a fixed  $4$--space $\Pi$ such that, for a general element $\Gamma$ of $\cal G$, one has  that  $\Pi\cap  \langle \Gamma\rangle$ has dimension 3.
\end{itemize}
If (a) holds, we are in case (ii) and if (b) holds we are in case (i).

 If $i=5$  and if $\Gamma$ and $\Gamma'$ are two general elements of $\cal G$, then  $\langle \Gamma\rangle \cap \langle \Gamma'\rangle$
has dimension 4. This implies that there is a 4--space $\Pi$ such that if $\Gamma$ is  a general element of $\cal G$, one has  that  $\Pi\subset  \langle \Gamma\rangle$. Then we are again in case (i). 
\end{proof}

\begin{remark}\label {rmk:hh} Note that if  (i) of Theorem \ref {thm:wd6} occurs, then $X\subset \p^r$ is actually defective and in general it has $f(X)=1$ and presents the reducible case. Indeed, if $p_0,p_1\in X$ are general points, and $V$ is the $6$--dimensional cone containing $X$, then $T_{V,p_0,p_1}$ has dimension 8 and it contains $T_{X,p_0}$ and $T_{X,p_1}$, which therefore in general intersect at a point. 

If we are in case (ii) of Theorem \ref {thm:wd6}, $X$ is still defective with, in general, $f(X)=1$. In fact, let $V$ be the $5$--dimensional cone on which $X$ sits, with vertex the plane $\Pi$. Then $V$ is swept out by a 1--dimensional family $\cal K$ of $4$--spaces, i.e., the spans of $\Pi$ with the lines of the scroll $\Sigma$ which is the base of the cone. The general such $4$--space $P$ intersects $X$ along a 3--fold $X_P$. If $x\in X$ is a general point and $P$ is the unique element of $\cal K$ containing $x$, then $T_{X,x}$ intersects $P$ in the $3$--space $T_{X_P,x}$ and therefore it intersects $\Pi$ along a line. Therefore two general tangent spaces to $X$ intersect in general at a point in $\Pi$, proving that $X$ is defective. 

It is interesting to note that $X$ is not in general of the first species. In fact, let $p_0,p_1\in X$ be general points. Then $T_{V,p_0,p_1}$ has dimension 8 and, containing $T_{X,p_0,p_1}$, it coincides with it. Now the bitangent hyperplane to $V$ at $p_0,p_1$ is only tangent along the 3--spaces generators of $V$ passing through $p_0,p_1$, hence it is bitangent to $X$ only along the surfaces in which these generators cut $X$. 

However, in some particular cases, $X$ can be of the top species. For example this is the case if the scroll $\Sigma$ is developable. In this case in fact the tangent space to $V$ is constant along the 4--spaces of the family $\cal K$ and the same argument we made above implies that a general bitangent hyperplane to $X$ is tangent along two threefolds of the form $X_P$ with $P\in \cal K$.

In any event, $X$ is never smooth in these cases.  \end{remark}

 \section* {Appendix}\label {sec:first}
  
  In this Appendix we will sketch the proof of  formulae \eqref {eq:form}. They can be deduced from the general formulae in Bertini's book  \cite [p. 490]{Ber}. In this Appendix we will not prove Bertini's formulae in their full strength, but just formulae \eqref {eq:form} that we used in the paper. We will use the notation introduced in \S \ref {sec:first}. Namely $C\subset \p^4$ is an irreducible, non--degenerate, projective curve with geometric genus 0. We denote by $d$ the degree of $C$ and by $n_1$ the degree of the tangent developable $T_1(C)$ to $C$, by $n_2$ the degree of variety $T_2(C)$ which is the closure of the union of osculating planes to $C$ in simple non--flex points,  by $n_3$ the {class} $C$, i.e., the degree of the dual $C^*$ of $C$. 
  
 For any branch $\xi$ of $C$, one has the \emph{full rank sequence} $(\alpha:=\alpha_0, \alpha_1, \alpha_2, \alpha_3)$ of $\xi$ where $\alpha, \alpha_1$ and $\alpha_2$ have been defined in the proof of Lemma \ref {lem:wd4}, and $\alpha+ \alpha_1+ \alpha_2+ \alpha_3$ is the intersection multiplicity of $\xi$ with the hyperosculating hyperplane to $\xi$ in its centre. For a generic branch $\xi$ one has $(\alpha, \alpha_1, \alpha_2, \alpha_3)=(1,1,1,1)$ and only for finitely many branches one has that $\alpha_i>1$ for some $i=0,\ldots, 3$. 
  
\begin{proposition}\label{prop:fast} In the above setting one has
\begin{equation}\label{eq:form2}
\begin{split}
\sum(\alpha-1)&+n_1=2d-2\\
\sum(\alpha_1-1)&+d+n_2=2n_1-2\\
\sum(\alpha_2-1)&+n_1+n_3=2n_2-2\\
\sum(\alpha_3-1)&+n_2=2n_3-2
\end{split}
\end{equation}
where the sums are taken over all branches of $C$. 
\end{proposition}

\begin{proof} For the first of these formulae, consider a general plane $P$ in $\p^4$, and look at the base point free $g^1_d$ which is the pull back on the normalization of $C$, that is $\p^1$, of the series cut out on $C$ by the hyperplanes passing through $P$. The number of branch points (counted with multiplicity) of this $g^1_d$ is $2d-2$. Now $n_1$ is just the number of these branch points which are simple (i.e., they count with multiplicity 1). Note that, by the genericity of $P$, a branch $\xi$ of $C$ counts for a branch point of multiplicity $\alpha-1$ and the only non--simple branch points may arise from branches of $C$. Therefore we have
\[
n_1=2d-2-\sum(\alpha-1)
\]
which is just the first formula in \eqref {eq:form2}. 

Next let us prove the last formula in \eqref {eq:form2} which is in a sense the dual of the first. Indeed,  we consider the dual curve $C^*\subset (\p^4)^*$, we fix a general plane $\Pi$ of $(\p^4)^*$ and look at the  base point free $g^1_{n_3}$ which is the pull back on the normalization of $C^*$, that is  again $\p^1$, of the series cut out on $C^*$ by the hyperplanes passing through $\Pi$. Note now that the dual of a branch $\xi$ with rank sequence  $(\alpha, \alpha_1, \alpha_2, \alpha_3)$ becomes a branch $\xi^*$ with rank sequence $(\alpha_3, \alpha_2, \alpha_1, \alpha)$ and that the tangent developable $T_1(C^*)$ to $C^*$ is dual to $T_2(C)$ and its degree is easily seen to be $n_2$. Therefore, applying to the above $g^1_{n_3}$ the same argument we applied to the $g^1_d$ to prove the first formula in \eqref {eq:form2}, we find
\[
n_2=2n_3-2-\sum(\alpha_3-1)
\]
which is the last formula in \eqref {eq:form2}.

Let us prove now the third formula  in \eqref {eq:form2}. The variety $T_2(C)$ is a hypersurface of degree $n_2$. Fix again a general plane $P$ of $\p^4$, and fix a general point $x\in P$, so that $x\not\in T_2(C)$. A general line $r$ in the pencil with centre $x$ in $P$ is a general line, so it intersects $T_2(C)$ in a reduced divisor of degree $n_2$, whose points correspond to osculating planes to $C$ in distinct points $p_1,\ldots, p_{n_2}$ which, when pulled back to the normalization of $C$, i.e., to $\p^1$, form a divisor $D_r$ of degree $n_2$. The closure of the union of the divisors $D_r$ on $\p^1$ when $r$ moves in the aforementioned pencil is a $g^1_{n_2}$ on $\p^1$. We will now compute the  ramification points of this series. They are contained in the following list:\\
\begin{inparaenum}

\item [(a)] let $\xi$ be a branch of $C$ of rank sequence $(\alpha, \alpha_1, \alpha_2, \alpha_3)$. It corresponds to a point $p_\xi$ of $\p^1$. The osculating plane $P_\xi$ to $\xi$ at its origin is contained in $T_2(C)$ and  its general point has multiplicity $\alpha_2$ for $T_2(C)$, as can be seen with a local computation which can be left to the reader. 
It is then clear that $p_\xi$ is a ramification point of multiplicity $\alpha_2$ for the $g^1_{n_2}$ (i.e., it belongs to the corresponding divisor of the $g^1_{n_2}$ with multiplicity $\alpha_2$); \\

\item [(b)] the plane $P$ intersects $T_1(C)$ in $n_1$ points, which, by generality of $P$, correspond to $n_1$ distinct tangent lines $t_1,\ldots, t_{n_1}$ to $C$ in smooth points of $C$ which correspond to points $p_1,\ldots, p_{n_1}$ of $\p^1$. We claim that $p_1,\ldots, p_{n_1}$ are simple ramification points of the $g^1_{n_2}$ (i.e., they belong to the corresponding divisors of the $g^1_{n_2}$ with multiplicity $2$). Indeed, for each $i=1,\ldots, n_1$, the line $t_i$ intersects $P$ in a point $q_i$, and the line  $r_i=\langle q_i,x\rangle$ corresponds to a divisor $D_{r_i}$ of the $g^1_{n_2}$, which contains $p_i$ with multiplicity 2 because the general point of $T_1(C)$ is a double (cuspidal) point for $T_2(C)$;\\

\item [(c)] there are $n_3$ distinct iperosculating hyperplanes $H_i$, $i=1,\ldots, n_3$, which pass through $x$. They correspond to distinct smooth points $\zeta_i\in C$, which correspond to points $z_i\in \p^1$, for $i=1,\ldots, n_3$. We claim that $z_1,\ldots, z_{n_3}$ are simple branch points of the $g^1_{n_2}$. Indeed, 
consider the osculating planes $P_i$ to $C$ at the points $\zeta_i$,  $i=1,\ldots, n_3$. One has $P_i\subset H_i$, hence $P_i$ intersects the line  $r_i=H_i\cap P$ and they count with multiplicity 2 in the intersection of the lines $r_i$ with $T_2(C)$;  in fact $H_i$ is (by genericity) simply tangent to $T_2(C)$ along $P_i$, so the line $r_i$ is tangent to $T_2(C)$, for $i=1,\ldots, n_3$. 
This implies that  $z_1,\ldots, z_{n_3}$ are simple branch points of the $g^1_{n_2}$.\\\end{inparaenum}

It is not difficult to see that these are the only ramification points of the $g^1_{n_2}$, so that Hurwitz formula gives
\[
2n_2-2=n_3+n_1+\sum(\alpha_2-1)
\]
which is just the third formula  in \eqref {eq:form2}. 

We can finally prove the second formula in \eqref {eq:form2}, which is dual to the third. In fact now we take a general plane $\Pi$ in $(\p^4)^*$, fix a general point $\eta\in \Pi$ and consider the pencil of lines through $\eta$ in $\Pi$. Consider then $T_2(C^*)$, which is dual to $T_1(C)$ and it has degree $n_1$. Then arguing as in the proof of the third formula, we get a $g^1_{n_1}$ on $\p^1$ of which we have to compute the number of branch points. This computation runs exactly in the same way as above, with the only change that $n_3$ has to be exchanged with $d$, $n_2$ with $n_1$ and $\alpha_2$ with $\alpha_1$. In conclusion Hurwitz formula now gives
\[
2n_1-2=d+n_2+\sum(\alpha_1-1)
\]
which is the second formula in \eqref {eq:form2}. \end{proof}

\begin{corollary}\label{cor:plot} In the above setting one has
\begin{equation}\label{eq:plot}
\sum \Big (4\alpha+3\alpha_1+2\alpha_2+\alpha_3-10  \Big)=5d-20.
\end{equation}

\end{corollary}

\begin{proof} It suffices to multiply the first formula in \eqref {eq:form2} by $4$, the second by 3, the third by 2, and then add up all of them. 
\end{proof}

Now we can finally prove formulae \eqref {eq:form}. The first formula in \eqref {eq:form} is just the first formula in \eqref {eq:form2}. The second formula in \eqref {eq:form} is obtained by eliminating $n_1$ between the first two formulae in \eqref {eq:form2}. As for the third formula in \eqref {eq:form}, first add up the last two formulae in \eqref {eq:form2}, obtaning
\[
n_3=n_1-n_2+\sum(\alpha_2-1)+\sum(\alpha_3-1)+4,
\]
Then substitute in this  the values of $n_1, n_2$ given by the first two of \eqref {eq:form} and the value of $\sum(\alpha_3-1)$ which can be obtained by \eqref {eq:plot}.

\end{document}